\documentclass{amsart}     

\usepackage[all]{xy}
\usepackage{graphicx}
\usepackage{mathtools}

\title{The rank filtration via a filtered bar construction}

\author{G. Z. Arone}
\address{Department of Mathematics, Stockholm University, 106 91 Stockholm, Sweden}
\email{gregory.arone@math.su.se}

\author{K. Lesh}
\address{Department of Mathematics, Union College, Schenectady, NY 12308 USA}
\email{leshk@union.edu}
\urladdr{http://www.math.union.edu/~leshk}

\numberwithin{equation}{section}

\theoremstyle{plain} 
\newtheorem{theorem}[equation]{Theorem}
\newtheorem{lemma}[equation]{Lemma}

\newtheorem{proposition}[equation]{Proposition}
\newtheorem{corollary}[equation]{Corollary}

\theoremstyle{definition}
\newtheorem{definition}[equation]{Definition}
\newtheorem{diag}[equation]{Diagram}

\newtheorem{example}[equation]{Example}
\newtheorem{remark}[equation]{Remark}
\newtheorem*{remark*}{Remark}

\theoremstyle{plain} 

\newcommand{\defining}[1]{{\emph{#1}}}

\newcommand{\integers}{{\mathbb{Z}}}
\newcommand{\naturals}{{\mathbb{N}}}

\newcommand{\complexes}{{\mathbb{C}}}

\newcommand{\Upk}{U\kern-2.5pt\left(p^{k}\right)}
\newcommand{\Uof}[1]{U\kern-2pt\left(#1\right)}

\newcommand{\Sphere}{{\bf{S}}}

\newcommand{\Ktheory}[1]{{\bf k}{#1}}
\newcommand{\stabilization}[1]{\partial_1 #1}
\newcommand{\realization}[1]{\left|\largestrut #1\right|}

\newcommand{\Inj}{\operatorname{Inj}}
\newcommand{\Gr}{\operatorname{Gr}}

\newcommand{\definedas}{\coloneqq}

\newcommand{\filter}[2]{\Rcal_{#1}#2}

\DeclareMathOperator{\id}{id}

\newcommand{\FilteredBK}[2]{\Ccal\filter{#1}{\Ccal^\bullet #2}}

\newcommand{\Hugestrut}{\mbox{{\Huge\strut}}}

\newcommand{\LARGEstrut}{\mbox{{\LARGE\strut}}}
\newcommand{\Largestrut}{\mbox{{\Large\strut}}}
\newcommand{\largestrut}{\mbox{{\large\strut}}}
\newcommand{\smallstrut}{\mbox{{\small\strut}}}

\newcommand{\collapsemap}[1]{p_{#1}}

\def\doCal#1{%
\ifx#1\doAllCalEnd\def\doAllCal{\relax}\else%
 \expandafter\edef\csname#1cal\endcsname{{\noexpand\mathcal #1}}\fi}
\def\doAllCal#1{\doCal#1\doAllCal}
\doAllCal ABCDEFGHIJHKLMNOPQRSTUVWXYZabcdefghijlmnopqrstuvwxyz\doAllCalEnd

\def\doBar#1{%
\ifx#1\doAllBarEnd\def\doAllBar{\relax}\else%
 \expandafter\edef\csname#1bar\endcsname{{\noexpand\overline{#1}}}\fi}
\def\doAllBar#1{\doBar#1\doAllBar}
\doAllBar ABCDEFGHIJHLMNOPQRSTUVWXYZabcdefghijklmnopqrstuvwxyz\doAllBarEnd
\def\doWiggle#1{%
\ifx#1\doAllWiggleEnd\def\doAllWiggle{\relax}\else%
 \expandafter\edef\csname#1wiggle\endcsname{{\noexpand\tilde{#1}}}\fi}
\def\doAllWiggle#1{\doWiggle#1\doAllWiggle}
\doAllWiggle ABCDEFGHIJHLMNOPQRSTUVWXYZabcdefghijklmnopqrstuvwxyz\doAllWiggleEnd

\def\doBold#1{%
\ifx#1\doAllBoldEnd\def\doAllBold{\relax}\else%
 \expandafter\edef\csname#1bold\endcsname{{\noexpand\bf #1}}\fi}
\def\doAllBold#1{\doBold#1\doAllBold}
\doAllBold ABCDEFGHIJHKLMNOPQRSTUVWXYZabcdefghijklmnopqrstuvwxyz\doAllBoldEnd

\newcommand{\whatever}{\text{--}}

\newcommand{\colim}{\operatorname{colim}\,}

\newcommand{\hocolim}{\operatorname{hocolim}\,}
\newcommand{\Id}{\operatorname{Id}}

\DeclareMathOperator{\Sp}{Sp}

\newcommand{\filtered}{\Top^{\mathrm{fi}}}
\newcommand{\Top}{\operatorname{Top}_*}

\newcommand{\MacLane}{Mac\,Lane }

\newcommand{\pdash}{$p$\kern1.3pt-}

\newcommand{\X}{X}   
\DeclareMathOperator{\Barr}{Bar}

\newcommand{\finiteset}[1]{\finitesetUnpointed{#1}_{+}}

\newcommand{\finitesetUnpointed}[1]{{\bf{#1}}}

\newcommand{\xlongrightarrow}[1]{\xrightarrow{\ #1\ }}

\newcommand\restr[2]{{
  \left.\kern-\nulldelimiterspace 
  #1 
  \vphantom{\big|} 
  \right|_{#2} 
  }}

\newcommand{\thatdiagram}{
\begin{gathered}
\xymatrix{
A_{\bullet}
     \ar[r]^-{f}
     \ar[d]^-{i}
  &  B_{\bullet}\ar[d]_-{j}\\
\Ecal_{\bullet}\ar[r]^-{p}_-{\simeq} \ar@/^1.0pc/[u]^-{r}
  &  \Dcal_{\bullet}. \ar@/_1.0pc/[u]_-{q}
}
\end{gathered}
}

\newtheorem*{CommutingSquareLemma}{Lemma~\ref{lemma: commuting squares}}

\newcommand{\CommutingSquareLemmaText}{
The outer square in \eqref{diagram: square of auxiliaries} is strictly commutative:
$f\circ r=q\circ p$. The inner square in \eqref{diagram: square of auxiliaries}
commutes up to homotopy, $j\circ f\simeq p\circ i$.
}

\newtheorem*{Linearity Theorem}{Theorem~\ref{theorem: linearity theorem}}

\newcommand{\LinearityTheoremText}
{For each $m$, the $\Gamma$-space $\realization{\FilteredBK{m}{\Fcal}}$ is special.}

\newtheorem*{BarFiltrationTheorem}{Theorem~\ref{theorem: bar filtration}}

\newcommand{\BarFiltrationTheoremText}
{Assume that $\Fcal$ is an augmented special $\Gamma$-space, and that the $\Ccal$ is the monad associated to an $E_\infty$-operad acting on $\Fcal$. 
Then there is an equivalence of spectra
\[
\stabilization{\!\left(\filter{m}\Fcal\right)}
\simeq
\Ktheory{\realization{\FilteredBK{m}{\Fcal}}}.
\]
That is, the sequence of spectra in the stable rank filtration
of~$\Fcal$ 
is equivalent to the infinite delooping of the following sequence of special $\Gamma$-spaces:
\[
*
\simeq \realization{\FilteredBK{0}{\Fcal}}
\to \realization{\FilteredBK{1}{\Fcal}}
\to \cdots
\to \realization{\FilteredBK{m}{\Fcal}}
\to \cdots
\realization{\Ccal \Ccal^\bullet \Fcal}
\simeq {\Fcal}.
\]
}

\begin{document}

\begin{abstract}    
Suppose $\Fcal$ is a special $\Gamma$-space equipped with a natural transformation $\Fcal\to\Sp^\infty$. Segal's infinite loop space machine~\cite{Segal-Categories} associates with $\Fcal$ a spectrum, denoted~$\Ktheory{\Fcal}$, equipped with a map $\Ktheory{\Fcal}\to H\integers$. In our previous work~\cite{Arone-Lesh-Fundamenta} we constructed a filtration of $\Ktheory{\Fcal}$ by a sequence of spectra, which we called the
\defining{stable rank filtration of~$\Fcal$}. In this paper we give a new construction of the stable rank filtration. The new construction is combinatorial in nature and avoids the process of stabilization. In particular,
we construct a sequence of special $\Gamma$-spaces whose group completion yields the stable rank filtration.
\end{abstract}

\maketitle


\section{Introduction}
\label{section: introduction}

Recall that a
\defining{$\Gamma$-space} is a pointed functor $\Fcal\colon \Gamma\to \Top$ from pointed finite sets to pointed topological spaces. Such a functor $\Fcal$ is said to be \defining{special} if the natural ``linearization" map $\Fcal(X\vee Y)\to \Fcal X\times \Fcal Y$ is a weak equivalence (and \defining{very special} if $\pi_{0}\Fcal$ takes values in groups as well). If $\Fcal$ is special, then Segal's group completion 
associates with $\Fcal$ an infinite loop space~\cite{Segal-Categories}. We denote the spectrum corresponding to the group completion of $\Fcal$ by~$\Ktheory{\Fcal}$, and call it the \defining{infinite delooping of~$\Fcal$}. It is worth noting that the process of group completion is not reversible: there is no canonical way to ``un-group-complete'' a very special $\Gamma$-space.

In a previous work~\cite{Arone-Lesh-Fundamenta}, we constructed a sequence of spectra filtering the spectrum $\Ktheory{\Fcal}$ when $\Fcal$ is an ``augmented" special $\Gamma$-space (see below). We called the sequence the
\defining{stable rank filtration}\footnote{
%
Actually, in~\cite{Arone-Lesh-Fundamenta} we used the term
``modified stable rank filtration," to distinguish our filtration from the related ``stable rank filtration" previously constructed by Rognes \cite{Rognes-Topology}. For simplicity of terminology in our current discussion, however, we will drop the word ``modified."
}
of~$\Fcal$. In this paper we give a new
construction of that stable rank filtration.
A notable feature of the new construction is that it produces
more than just a filtration of the spectrum $\Ktheory{\Fcal}$
by a sequence of spectra, as in our previous work.
Instead, it produces a filtration
of the actual special $\Gamma$-space, $\Fcal$,
by a sequence of $\Gamma$-spaces that are themselves \emph{also} special;
applying the processes of group completion and infinite delooping to our new sequence yields the stable rank filtration.

To state our results more precisely, we observe that a prototypical example of a special $\Gamma$-space is the infinite symmetric product functor, $\Sp^\infty$, which is equipped with a natural filtration by the (not special) $\Gamma$-spaces~$\Sp^{m}$.
Following our earlier work (e.g.,~\cite{Arone-Lesh-Fundamenta}), we say that a $\Gamma$-space $\Fcal$ is \defining{augmented} if there is a natural transformation $\epsilon\colon\Fcal\to \Sp^\infty$ with the property that the preimage of the basepoint of $\Sp^\infty X$ contains only the basepoint of~$\Fcal X$.
In this case, pulling back~$\epsilon$ over the
inclusions $\Sp^{m}\hookrightarrow\Sp^{\infty}$ gives a filtration of $\Fcal$ by (in general not special) $\Gamma$-subspaces. We call this
the \defining{unstable rank filtration of~$\Fcal$},
and denote it~$\filter{m}\Fcal$:
\begin{equation}\label{eq: unstable rank filtration}
\begin{gathered}
\xymatrix{
\filter{0}\Fcal\ \ar[r]\ar[d]
    & \ \filter{1}\Fcal\ \ar[r]\ar[d]
    & \ \cdots\  \ar[r]
    & \ \filter{m}\Fcal\ \ar[r]\ar[d]
    & \ \cdots\ \ar[r]
    & \ \Fcal\ar[d]^{\epsilon}
\\
\Sp^{0}\ \ar[r]
    & \ \Sp^{1}\ \ar[r]
    & \ \cdots\  \ar[r]
    & \ \Sp^{m}\ \ar[r]
    & \ \cdots\  \ar[r]
    & \ \Sp^{\infty}.
}
\end{gathered}
\end{equation}

It is also possible to associate a spectrum to a $\Gamma$-space that is not special.
Namely, given a $\Gamma$-space~$\Fcal$, define the spectrum $\stabilization{\Fcal}$
\begin{equation}   \label{eq: stabilize a Gamma space}
\stabilization{\Fcal}\definedas
\left\{
\Fcal(S^0), \Fcal\left(S^1\right), \ldots,
          \Fcal\left(S^n\right), \ldots
\right\}.
\end{equation}
We call the spectrum $\stabilization{\Fcal}$ the \defining{stabilization of~$\Fcal$}. The notation reflects the fact that the stabilization is the same as the first Goodwillie derivative of $\Fcal$~\cite{Goodwillie-Calculus-I}.
If $\Fcal$ is actually a special $\Gamma$-space, then the two ways of obtaining a spectrum are equivalent:  $\stabilization{\Fcal}\simeq\Ktheory{\Fcal}$
\cite[Theorem~4.2]{Bousfield-Friedlander}. Hence the process of stabilization can be regarded as a generalization of Segal's construction that associates a spectrum to a special $\Gamma$-space.

When $\Fcal$ is an augmented $\Gamma$-space, we can consider the $\Gamma$-spaces in the unstable rank filtration of~$\Fcal$ \eqref{eq: unstable rank filtration} and look at their stabilizations
\eqref{eq: stabilize a Gamma space} to obtain a filtration of
the stabilization~$\stabilization{\Fcal}$ of~$\Fcal$. We use the term
\defining{stable rank filtration of~$\Fcal$}
to refer to the sequence of spectra
\begin{equation}\label{eq: rank}
\stabilization{\!\left(\filter{0}\Fcal\right)}
     \rightarrow \stabilization{\!\left(\filter{1}\Fcal\right)}
     \rightarrow \cdots
     \rightarrow \stabilization{\!\left(\filter{m}\Fcal\right)}
     \rightarrow \cdots
     \rightarrow \stabilization{\Fcal}.
\end{equation}
In previous work, \cite{Arone-Lesh-Fundamenta}, we studied in detail
the rank filtration of two particular $\Gamma$-spaces.
The first is $\Sp^\infty$, where the stable rank filtration of $\stabilization{\Sp^\infty}\simeq H\integers$ is equivalent to
the filtration of $H\integers \simeq \Sp^\infty(\Sphere)$ by the spectra $\Sp^m(\Sphere)$
(here $\Sphere$ indicates the sphere spectrum).
Our other example was the augmented $\Gamma$-space that represents connective complex $K$-theory, which we denote by $\Fcal_U$ (see Example~\ref{example: U}).
Its stabilization is $\stabilization{\Fcal_U}\simeq bu$,
the connective $K$-theory spectrum, and
the stable rank filtration of $\Fcal_U$
turns out to be equivalent to Rognes's
stable rank filtration of~$bu$ \cite[Proposition 4.10]{Arone-Lesh-Fundamenta}.

For a general $\Gamma$-space~$\Gcal$, there is no reason for the infinite loop space
$\Omega^{\infty}\left(\stabilization{\Gcal}\right)$ to be the group completion of a conceptually illuminating special $\Gamma$-space.
For example, an augmented special $\Gamma$-space~$\Fcal$
does not seem to have an obvious natural model for the infinite loop space $\Omega^\infty\left(\stabilization{\filter{m}\Fcal}\right)$ as the group completion of a special $\Gamma$-space
related to~$\Fcal$. In particular, there does not seem to be 
an obvious reason \emph{a priori} for the map of spectra 
\[
\stabilization{\filter{m}\Fcal}\to \stabilization{\Fcal},
\]
which is the $m$-th stage of the stable rank filtration~\eqref{eq: rank}, to be induced by a map from some special $\Gamma$-space to~$\Fcal$. One of the contributions of this paper is to provide just such a model for the stable rank filtration.
The construction occurs on the level of $\Gamma$-spaces, not spectra; it is thus
entirely combinatorial, avoiding the process of stabilization that is at the heart of both our previous construction and that of Rognes. More precisely,
we construct a sequence of special $\Gamma$-spaces
\begin{equation}\label{eq: uncompleted}
*\simeq
\realization{\FilteredBK{0}{\Fcal}}
\to
\realization{\FilteredBK{1}{\Fcal}}
\to \cdots \to
\realization{\FilteredBK{m}{\Fcal}}
\to \cdots \to
\realization{\Ccal \Ccal^\bullet \Fcal}
{\simeq} \, {\Fcal}
\end{equation}
such that \eqref{eq: rank} is equivalent to the infinite delooping of~\eqref{eq: uncompleted}.
The argument has a similar flavor to our main construction
in~\cite{Arone-Lesh-Crelle}. Namely, we construct a large simplicial space
that models an augmented special $\Gamma$-space~$\Fcal$ in a way that is friendly to the rank filtration, and then we carefully filter the larger construction.  While the current work is
in service of a long-term program, we hope that it will also be of some independent interest.

Before describing the construction of
\eqref{eq: uncompleted}
in more detail, we say a few words about our intended application. By work of Kuhn (and Priddy) it is known that the ``chain complex of spectra'' associated with the symmetric powers filtration of $H\integers$ is exact~\cite{Kuhn-Whitehead, Kuhn-Priddy}. In our own previous work \cite{Arone-Lesh-Crelle, Arone-Lesh-Fundamenta}, we conjectured that
an analogous statement holds for the stable rank filtration of $bu$.
Because Kuhn's theorem was originally conjectured by G.~Whitehead, we dubbed our conjecture ``the $bu$-analogue of the Whitehead conjecture.''
This paper is part of a program to give a unified proof of Kuhn's theorem and its $bu$-analogue.
The new construction of the rank filtration in this paper will play a key role in our proof. Broadly speaking, the plan of our program is to use the filtration of the
bar construction $\Ccal\Ccal^\bullet \Fcal$ by
$\largestrut\FilteredBK{m}{\Fcal}$ as a model for the stable rank filtration, and to leverage the interplay between the
rank filtration and the filtration of  $\Ccal\Ccal^\bullet \Fcal$ by simplicial skeleta.

To describe the main construction of this paper, consider the following setup. Suppose that $\Fcal$ is an augmented (Definition~\ref{definition: augmented}) special $\Gamma$-space, and $\Ocal$ is an $E_\infty$-operad that acts on~$\Fcal$. We prove in Proposition~\ref{proposition: augmentation preserved} that the action of $\Ocal$ must preserve the augmentation.
Associated to $\Ocal$ we have a monad~$\Ccal$ 
(see Example~\ref{example: C}), which
has a left action on $\Fcal$ in the sense that there is a natural transformation $\mu_{\Fcal}\colon \Ccal\Fcal\to \Fcal$ satisfying appropriate associativity and unit axioms. The action $\Ccal$ also respects the augmentation, in a sense made precise in Lemma~\ref{lemma: augmentation}.

We form the two-sided bar construction $\Barr(\Ccal, \Ccal, \Fcal)$, a simplicial augmented $\Gamma$-space that is given in $k$-th simplicial degree by~$\Ccal^{k+1}\Fcal$, and is simplicially augmented by $\mu_{\Fcal}\colon\Ccal\Fcal\rightarrow\Fcal$:
\[
\Fcal
\xleftarrow{\ \mu_{\Fcal}\ }
\left\{\Largestrut
\xymatrix{
\Ccal \Fcal\
   \ar[r]
&\ \Ccal^{2}\Fcal\
         \ar@<-.6ex>[l]
         \ar@<.6ex>[l]
         \ar@<-.5ex>[r]
         \ar@<.5ex>[r]
&\ \Ccal^{3}\Fcal\
         \ar@<-1.0ex>[l]
         \ar[l]
         \ar@<1.0ex>[l]
\ldots
\quad\Ccal^{k+1}\Fcal\quad
\ldots
}
\right\}.
\]
The geometric realization of $\Barr(\Ccal, \Ccal, \Fcal)$ is again a $\Gamma$-space. The unit map
for $\Ccal$ gives a natural section
$\Fcal\rightarrow\Ccal\Fcal$ of the simplicial augmentation~$\mu_{\Fcal}$, and it
induces an ``extra degeneracy" throughout the bar construction. As a result,
the augmentation $\mu_{\Fcal}$ induces not only a map, but an actual equivalence
$\realization{
\Barr(\Ccal, \Ccal, \Fcal)}
       \xrightarrow{\ \simeq\ }{\Fcal}$.
One may view $\Barr(\Ccal, \Ccal, \Fcal)$ as defining a resolution of $\Fcal$ by free
$\Ccal$-modules.

Both the unstable rank filtration \eqref{eq: unstable rank filtration} and the stable rank filtration \eqref{eq: rank} of~$\Fcal$ can be obtained as (different!) simplicial filtrations of $\Barr(\Ccal, \Ccal, \Fcal)$. First, we discuss the \emph{unstable} filtration, which is much easier than the stable one. (See Section~\ref{section: bar} for details.) For each~$k$, the (non-special)
$\Gamma$-space~$\Ccal^k\Fcal$ has a natural augmentation
$\Ccal^k\Fcal\to  \Sp^\infty$.
It follows that the augmented $\Gamma$-space $\Ccal^k\Fcal$ is filtered by
$\Gamma$-subspaces~$\filter{m}\Ccal^k\Fcal$. Because the action of $\Ccal$ on $\Fcal$ respects the augmentation, the augmentation of $\Ccal^k\Fcal$ 
is compatible with face and degeneracy maps in the bar construction. It follows that we can apply $\filter{m}$ levelwise to~$\Barr(\Ccal, \Ccal, \Fcal)$.
We obtain, for each $m$, a simplicial $\Gamma$-subspace, which we
denote~$\filter{m}\Barr(\Ccal, \Ccal, \Fcal)$. Explicitly, it has the following form:
\begin{equation*}
\filter{m}\Fcal
\longleftarrow
{\underbrace{
\Hugestrut \left\{\Largestrut
\xymatrix{
\filter{m}\Ccal\Fcal
   \ar[r]
&\ \filter{m}\Ccal^{2}\Fcal \
         \ar@<-.6ex>[l]
         \ar@<.6ex>[l]
         \ar@<-.5ex>[r]
         \ar@<.5ex>[r]
&\ \filter{m}\Ccal^{3}\Fcal \
         \ar@<-1.0ex>[l]
         \ar[l]
         \ar@<1.0ex>[l]
\ldots
\quad\filter{m}\Ccal^{k+1}\Fcal\quad
\ldots
}
\right\}.
}_{\filter{m}\Barr(\Ccal, \Ccal, \Fcal)}}
\end{equation*}
Letting $m$ vary, we obtain a sequence of simplicial $\Gamma$-spaces
\begin{equation*}
\filter{1}\Barr(\Ccal, \Ccal, \Fcal)
\to \cdots \to
\filter{m}\Barr(\Ccal, \Ccal, \Fcal)
\to \cdots \to
\Ccal \Ccal^\bullet \Fcal=\Barr(\Ccal, \Ccal, \Fcal).
\end{equation*}
We can apply the geometric realization functor to each simplicial $\Gamma$-space, and it is not hard to show that the result is equivalent to the unstable rank filtration of $\Fcal$ (see Proposition~\ref{proposition: tautology}).

For the \emph{stable} rank filtration, we consider another, more sneaky way to filter $\Ccal\Ccal^\bullet \Fcal=\Barr(\Ccal, \Ccal, \Fcal)$ by simplicial subobjects: we filter $\Ccal\Ccal^k\Fcal$ by $\Ccal\filter{m}\Ccal^k\Fcal$ for each~$k$.
Although it is not immediately obvious, the face and degeneracy maps in $\Barr(\Ccal, \Ccal, \Fcal)$ turn out to respect this filtration
(see Proposition~\ref{proposition: structure maps respect filtration}).
In other words, for every $m$ there is a simplicial $\Gamma$-space,
denoted~$\FilteredBK{m}{\Fcal}$,
whose $k$-th space is~$\Ccal\filter{m}\Ccal^{k}\Fcal$:
\small
\begin{equation}\label{eq: intro filtered}
\FilteredBK{m}{\Fcal}=
\left\{
\xymatrix{
\Ccal \filter{m}\Fcal
   \ar[r]
&\ \Ccal\filter{m}\Ccal\Fcal \
         \ar@<-.6ex>[l]
         \ar@<.6ex>[l]
         \ar@<-.5ex>[r]
         \ar@<.5ex>[r]
&\ \Ccal\filter{m}\Ccal^{2}\Fcal \
         \ar@<-1.0ex>[l]
         \ar[l]
         \ar@<1.0ex>[l]
\ldots
\quad\Ccal\filter{m}\Ccal^{k}\Fcal\quad
\ldots
}
\right\}
.
\end{equation}
\normalsize
It is worth noting that unlike
$\Ccal\Ccal^{\bullet}\Fcal$ and $\filter{m}\Ccal\Ccal^{\bullet}\Fcal$ above, the simplicial
$\Gamma$-space in~\eqref{eq: intro filtered} has no obvious augmentation
or extra degeneracy.

Letting $m$ vary in~\eqref{eq: intro filtered}, we obtain a sequence of simplicial $\Gamma$-spaces
\small{
\begin{equation}\label{eq: bar filtration}
*\simeq
\FilteredBK{0}{\Fcal}
\to
\FilteredBK{1}{\Fcal}
\to \cdots \to
\FilteredBK{m}{\Fcal}
\to \cdots \to
\Ccal \Ccal^\bullet \Fcal=\Barr(\Ccal, \Ccal, \Fcal).
\end{equation}}
Our main theorem (Theorem~\ref{theorem: bar filtration} below) says that the corresponding sequence of geometric realizations gives a model for the stable rank filtration.

An important technical result needed in our work is that for each~$m$,
the geometric realization $\realization{\FilteredBK{m}{\Fcal}}$ is a special $\Gamma$-space.
It turns out that for most $m$ and $k$, the $\Gamma$-space $\Ccal\filter{m}{\Ccal^k\Fcal}$ at each simplicial level is
\emph{not} special (see\eqref{eq: nonexample}). Therefore, there is no obvious reason for the geometric realization
$\realization{\FilteredBK{m}{\Fcal}}$
to be a special $\Gamma$-space. Nevertheless, it turns out that it is. This is the key technical result of the paper.

\begin{Linearity Theorem}
\LinearityTheoremText
\end{Linearity Theorem}

The proof is surprisingly intricate, and Sections~\ref{section: strategy}--\ref{section: maps out of E} are devoted to the details.
But once we have the theorem, 
it is fairly straightforward to prove our main result, which is that filtration~\eqref{eq: bar filtration} induces the stable rank filtration \eqref{eq: rank} via geometric realization and
infinite delooping.

\begin{BarFiltrationTheorem}
\BarFiltrationTheoremText
\end{BarFiltrationTheorem}

Thus~\eqref{eq: bar filtration} gives a new model for the stable rank filtration, by means of a sequence of special $\Gamma$-spaces, as stated at the beginning of the introduction. In future work we will use Theorem~\ref{theorem: bar filtration} to give a simplicial model for the subquotients of the stable rank filtration in two cases of interest: the first case is $\Fcal=\Sp^\infty$, and the second is $\Fcal=\Fcal_U$, the $\Gamma$-space representing complex $K$-theory (see Examples~\ref{example: U} and~\ref{example: bu example operad action}).
It is known that the subquotients in these two cases can be described, respectively, in terms of the complex of partitions, $\Pcal_{m}$, and the complex of direct-sum decompositions,~$\Lcal_{m}$~\cite{Arone-Dwyer, Arone-Lesh-Fundamenta}.
In particular, Theorem~\ref{theorem: bar filtration} will enable us to incorporate the complexes $\Pcal_{m}$ and $\Lcal_{m}$ into the simplicial bar construction $\Barr(\Ccal, \Ccal, \Fcal)$ in such a way that
they appear in the subquotients of the filtration
\emph{as simplicial spaces}.
The interplay between the homological properties of
$\Pcal_{m}$ and~$\Lcal_{m}$ (\cite{ADL2, ADL3}) on one hand, and of the bar construction on the other hand, will play a key role in our planned proof of Kuhn's theorem and its $bu$-analogue.

The organization of the paper is as follows.
In Section~\ref{section: preliminaries} we gather some background material on augmented $\Gamma$-spaces and their rank filtrations.
In Section~\ref{section: operad actions} we recall the connection between special $\Gamma$-spaces and actions of $E_\infty$-operads. We then show that if $\Fcal$ is an augmented special $\Gamma$-space, then {\it any} action of an $E_\infty$-operad on $\Fcal$ respects the augmentation.
Strictly speaking, the main result of this section 
(Proposition~\ref{proposition: augmentation preserved}) is not needed
in the rest of the paper, but it seems to clarify the general picture.
In Section~\ref{section: bar} we introduce the two-sided bar construction $\Barr(\Ccal, \Ccal, \Fcal)$, where $\Fcal$ is an augmented special $\Gamma$-space with an 
action of an $E_\infty$-monad $\Ccal$.
The bar construction gives a resolution of $\Fcal$ by free $\Ccal$-modules, and we show how the unstable rank filtration can be realized as a simplicial filtration of the bar construction.
In Section~\ref{section: stable filtration}
we introduce our main construction: the filtration of the bar construction that will serve as a model for the stable rank filtration. We prove that our filtration really is a model for the stable rank filtration, modulo the key theorem that the terms in our filtration are special $\Gamma$-spaces (Theorem~\ref{theorem: linearity theorem}).

Sections~\ref{section: strategy}--\ref{section: maps out of E} are devoted to the proof of Theorem~\ref{theorem: linearity theorem}. In Section~\ref{section: strategy} we outline the strategy of the proof, state several intermediate results, and give the proof assuming these results. The remaining sections are devoted to proving the intermediate results.

\bigskip
\noindent{\bf Notation}\\

Let $\Ical$ be the category of finite unpointed sets and injective functions, with standard objects given by
$\finitesetUnpointed{m}=\{1,2,\ldots,m\}$, the set of
the first $m$ positive integers. For a finite \emph{pointed} set, we use $\finiteset{m}=\{0,1,\ldots, m\}$ for the pointed finite set with basepoint~$0$ and the first $m$ positive integers as its non-basepoint elements. Let $\Gamma$ denote the category of finite pointed sets and pointed functions. The set of nonnegative integers is denoted by~$\naturals$.

The monad associated with an $E_\infty$-operad is denoted by~$\Ccal$. We use $\otimes$ for a (strict) coend construction.

We write $\Top$ for the category of pointed topological spaces, and $\filtered$ for the category of filtered topological spaces (Definition~\ref{definition: Topfi}).

\section{Background on $\Gamma$-spaces}
\label{section: preliminaries}

In this section we gather some preparatory material on $\Gamma$-spaces and their rank filtrations.
We begin with some generalities about (special) $\Gamma$-spaces, and we review from \cite{Arone-Lesh-Fundamenta} the notion of $\Gamma$-spaces that are augmented over $\Sp^\infty$.
In the bulk of the section, we discuss in some detail the
special augmented $\Gamma$-space~$\Ccal$, which is the monad
associated to an $E_\infty$-operad. In later sections, we will be studying filtrations of spaces $\Ccal^k X$ obtained by iterating the augmented functor~$\Ccal$. We prepare for this work by considering
a more general set-up. In particular, we define a natural filtration on $\Ccal X$ where $X$ is a pointed space that is itself equipped with a filtration (for example, $X=\Ccal Y$).
Then we view $\Ccal$ as a functor from the category of filtered spaces to itself.
This point of view will be used in Section~\ref{section: stable filtration} in the proof that
the face and degeneracy maps in the simplicial $\Gamma$-space $\Ccal\Ccal^\bullet \Fcal$
respect the subobject $\Ccal\filter{m}\Ccal^\bullet \Fcal$, and so diagram~\eqref{eq: intro filtered} exists.

We note that if $\Fcal$ is a $\Gamma$-space, we can evaluate $\Fcal$ on pointed simplicial sets, or pointed topological spaces, by means of homotopy left Kan extension.
As a rule, we will not distinguish notationally between a $\Gamma$-space $\Fcal$ and its homotopy left Kan extensions.
Conversely, suppose $\Fcal$ is a pointed homotopy functor from the category of pointed spaces (or simplicial sets) to pointed spaces.
Somewhat loosely, we will say that such a functor $\Fcal$ is a $\Gamma$-space if $\Fcal$ is naturally equivalent to the homotopy left Kan extension of its restriction to~$\Gamma$.

Recall that a $\Gamma$-space $\Fcal$ is \defining{special} if the natural map $\Fcal(X\vee Y)\to \Fcal(X)\times \Fcal(Y)$ is an equivalence. The following is a basic example of a special $\Gamma$-space.

\begin{example}   \label{example: Sp}
Let $\Sp^\infty$ be the infinite symmetric product functor from pointed spaces to pointed spaces. Let $\finiteset{m}=\{0,1,\ldots, m\}$ denote the pointed set with basepoint~$0$ and the first $m$ positive integers as its non-basepoint elements.
There is a canonical isomorphism $\Sp^\infty\left(\finiteset{m}\right)\cong \naturals^m$,
where $\naturals$ is the set of non-negative integers. The map
$\Sp^\infty\left(\finiteset{m}\vee\finiteset{n}\right)
\to\Sp^\infty\left(\finiteset{m}\right)\times\Sp^\infty\left(\finiteset{n}\right)$
can be identified with a bijection $\naturals^{m+n}\xlongrightarrow{\cong}\naturals^m\times \naturals^n$. In particular, this map is a homotopy equivalence, so $\Sp^\infty$ is a special $\Gamma$-space.

When $X$ is a pointed space, it is convenient to have a topological description of $\Sp^\infty(X)$ as well. In this case there is a homeomorphism
\begin{equation*}
\Sp^\infty(X)\cong\underset{m\to \infty}{\colim} \Sp^m(X)=\underset{m\to \infty}{\colim} X^m/_{\Sigma_m},
\end{equation*}
where the maps in the colimit are basepoint inclusions.
The Dold-Thom theorem
tells us that the infinite delooping of the functor $\Sp^\infty$ is the Eilenberg-\MacLane spectrum~$H\integers$.
\end{example}

\begin{definition}   \label{definition: augmented}
An augmentation of a $\Gamma$-space $\Fcal$ is a natural transformation $\epsilon\colon \Fcal\to \Sp^\infty$ with the property that for every $\finiteset{n}$ the preimage of the basepoint under the augmentation map $\epsilon\left(\finiteset{n}\right)\colon \Fcal\left(\finiteset{n}\right)\to \Sp^\infty\left(\finiteset{n}\right)$ consists of just the basepoint of~$\Fcal\left(\finiteset{n}\right)$.
\end{definition}

Obviously the $\Gamma$-space $\Sp^\infty$ itself is augmented, via the identity transformation.
We have two other principal examples of augmented special $\Gamma$-spaces in mind,
denoted $\Ccal$ and $\Fcal_U$. The first one plays a central role in this paper, while the second one will be important in a sequel.

\begin{example}    \label{example: C}
We recall from~\cite{May-Geometry} the construction of the monad associated with an $E_\infty$-operad~$\Ocal$. The underlying symmetric sequence of $\Ocal$ consists of a sequence of contractible free $\Sigma_n$-spaces $\{E\Sigma_n\mid n=0, 1, \ldots, \}$, and we assume that $E\Sigma_0=*$. We denote the resulting $\Gamma$-space by~$\Ccal$.

Let $\Ical$ be the category of finite unpointed sets and injective functions, which contains ``standard" objects
$\finitesetUnpointed{n}=\{1,2,\ldots,n\}$. The operad structure of~$\Ocal$ endows the set of spaces $\left\{E\Sigma_n\right\}$ with the structure of a contravariant functor from $\Ical$ to topological spaces, $\finitesetUnpointed{n}\mapsto E\Sigma_n$
\cite[Notations~2.3]{May-Geometry}. On the other hand, a pointed space $X$ gives rise to a {\it covariant} functor from $\Ical$ to topological spaces given by
$\finitesetUnpointed{n}\mapsto X^{n}$. Here the functoriality is defined using basepoint inclusions. Define the functor $X\mapsto \Ccal(X)$ to be the strict coend of these two functors on~$\Ical$:
\begin{equation} \label{eq: coend formula}
\Ccal(X)\definedas E\Sigma_n\otimes_{n\in \Ical} X^n.
\end{equation}
Note that $\Ccal$ commutes with geometric realization.
This implies that $\Ccal$ is equivalent to the homotopy left Kan extension of its restriction to~$\Gamma$, and so
$\Ccal$ is a $\Gamma$-space. In fact, $\Ccal$ is actually a special $\Gamma$-space, and the Barratt-Priddy-Quillen Theorem says that
the infinite delooping of~$\Ccal$ is the sphere spectrum~$\Sphere$.

To define the augmentation on~$\Ccal$, we note that the functor
$\Sp^\infty=\underset{n\to \infty}{\colim} \Sp^n$
has a similar formula
to~\eqref{eq: coend formula}, where each $E\Sigma_n$ is replaced with a point:
\[
\Sp^\infty(X)\cong *\otimes_{n\in \Ical} X^n.
\]
The augmentation of $\Ccal$ is induced by the maps $E\Sigma_n\to *$, which induce a natural transformation
\[
\xymatrix{
\Ccal(X)=E\Sigma_n\otimes_{n\in\Ical} X^{n}  \ar[d]_{\epsilon}
\\
\Sp^\infty(X)= \ast \otimes_{n\in\Ical} X^n.
}
\]

A second point of view on the augmentation is that there is a natural homotopy equivalence
\begin{equation} \label{eq: CX as hocolim}
\Ccal(X)\simeq \underset{n\in \Ical}{\hocolim} X^n
\end{equation}
and a natural homeomorphism
\[
\Sp^\infty(X) \cong \underset{n\in \Ical}{\colim} X^n.
\]
In these terms the augmentation $\epsilon\colon\Ccal\to \Sp^\infty$ is the natural map from the homotopy colimit to the strict colimit.
\end{example}
\medskip

\begin{example}    \label{example: U}
Another example of an augmented $\Gamma$-space, which will play an important role in our intended future application, is the special $\Gamma$-space representing connective complex $K$-theory. We denote it by $\Fcal_U$, and for concreteness we describe an explicit model for it.

Let $\Inj(V,W)$ denote the space of unitary linear transformations of complex vector spaces
$V\rightarrow W$.
On an object $\finiteset{m}$ of~$\Gamma$, we define $\Fcal_{U}$ by a disjoint union indexed by $m$-tuples of non-negative integers:
\[
\Fcal_U\left(\strut\finiteset{m}\right)
    =\coprod_{(k_1,\ldots, k_m)\in \naturals^m}
     \Inj\left(\complexes^{k_1+\cdots+k_m}, \complexes^\infty\right)
          /_{U(k_1)\times\cdots\times U(k_m)}.
\]
Note that the component corresponding to $(0,\ldots,0)$ is a single point,
which is by definition the basepoint
of~$\Fcal_U\left(\strut\finiteset{m}\right)$.

We must also define $\Fcal_U$ on morphisms. First, the intuition. A point in
$\Fcal_U\left(\strut\finiteset{m}\right)$ is essentially an ordered $m$-tuple of pairwise orthogonal subspaces of~$\complexes^\infty$. Given a pointed map $\alpha\colon \finiteset{m}\to \finiteset{n}$, we must take an ordered $m$-tuple of subspaces of~$\complexes^{\infty}$, and get from it an ordered $n$-tuple of subspaces. To do this, we take all subspaces in
the given $m$-tuple whose indices have the same (non-basepoint) image in~$\finiteset{n}$,
and we add them together. If some elements of $\finiteset{m}$ are sent by $\alpha$ to the basepoint, then the corresponding subspaces are dropped.

Now we describe the action of $\Fcal_U$ on morphisms more precisely.
Given a pointed map
$\alpha\colon \finiteset{m}\to \finiteset{n}$,
we must define
$\Fcal(\alpha)\colon
      \Fcal\left(\strut\finiteset{m}\right)
      \to
      \Fcal\left(\strut\finiteset{n}\right)$.
We begin with the induced map on the indexing sets. Given an $m$-tuple $(k_1, \ldots, k_m)$,
define the $n$-tuple $\alpha(k_1, \ldots, k_m)$ by summing over preimages of elements of~$\finiteset{n}$:
\begin{align*}
\alpha(k_1, \ldots, k_m)  &=\left(l_1, \ldots, l_n\right)\\
\mbox{ where }l_i                       &=\Sigma_{j\in\alpha^{-1}(i)} k_j  \mbox{\quad for $i=1, \ldots, n$.}
\end{align*}
Note that if some non-basepoint element $i\in\finiteset{m}$ goes to the basepoint, i.e.
$\alpha(i)=0\in\finiteset{n}$, and if in addition $k_{i}>0$, then
$\sum_{\finiteset{n}}l_{i} < \sum_{\finiteset{m}}k_{j}$.

Finally, for each $i$, we choose an isomorphism
$
\complexes^{l_i}\xlongrightarrow{\cong} \complexes^{\Sigma_{j\in\alpha^{-1}(i)} k_j}.
$
 We add them together to give a unitary morphism
\begin{equation}   \label{eq: sum of maps}
\complexes^{l_1+\cdots+l_n}\longrightarrow \complexes^{k_1+\cdots+k_m},
\end{equation}
which in turn induces a well-defined map
\[
\Inj\left(\complexes^{k_1+\cdots+k_m}, \complexes^\infty\right)
          /_{U(k_1)\times\cdots\times U(k_m)}
\longrightarrow
\Inj\left(\complexes^{l_1+\cdots+l_n}, \complexes^\infty\right)
          /_{U(l_1)\times\cdots\times U(l_n)}.
\]
We obtain a well-defined map
$\Fcal_{U}\left(\finiteset{m}\right)
   \xrightarrow{\Fcal_{U}(\alpha)}
   \Fcal_{U}\left(\strut\finiteset{n}\right)$
by carrying out this procedure on each component of the domain.

The augmentation $\Fcal_U\left(\strut\finiteset{m}\right)\to \Sp^\infty\left(\strut\finiteset{m}\right)$ is the map that collapses each connected component to a single point:
\begin{equation}    \label{eq: define augmentation FU}
\begin{gathered}
\xymatrix{
\coprod\limits_{(k_1,\ldots, k_m)\in \naturals^m}
     \Inj\left(\complexes^{k_1+\cdots+k_m}, \complexes^\infty\right)
     /_{U(k_1)\times\cdots\times U(k_m)}
     \ar[d]^-{\epsilon}\\
\coprod\limits_{(k_1,\ldots, k_m)\in {\naturals}^{m}} \hspace{-1em} *
\quad \cong \naturals^m.
}
\end{gathered}
\end{equation}
Notice that $\Fcal_U(\finiteset{1})\cong \coprod_{k=0}^\infty BU(k)$,
and $\Fcal_U\left(\strut\finiteset{m}\right)\simeq \Fcal_U(\finiteset{1})^m$, meaning that
$\Fcal_U$ is actually a special $\Gamma$-space.
It is well known that the infinite delooping of ${\Fcal_U}$ is the connective
$K$-theory spectrum~$bu$. For more details see~\cite{Segal-Categories}.
\end{example}

\medskip

We return to generalities, this time regarding filtered spaces.
We continue to assume that $\Fcal$ is an augmented $\Gamma$-space, i.e., a $\Gamma$-space equipped with an augmentation $\epsilon\colon\Fcal\to \Sp^\infty$. As described in the introduction, the filtration of $\Sp^\infty$ by~$\Sp^m$ induces a filtration of $\Fcal$ by subfunctors~$\filter{m}\Fcal$, where $\filter{m}\Fcal$ is defined as the $\Gamma$-space that is the strict pullback of the diagram
\[
\Fcal\xrightarrow{\ \epsilon\ } \Sp^\infty \longleftarrow \Sp^m.
\]

We need to understand the filtration above not only when it is applied to an augmented
functor~$\Fcal$, but also when it is applied, for example, to a composite functor
$\Ccal\Fcal$ or a higher iterate~$\Ccal^{k}\Fcal$.
For this purpose we temporarily drop $\Fcal$ from the picture, and we consider
the general problem of filtering the functor $\Ccal$ when it is
applied to a \emph{filtered} space (such as an output of an augmented $\Gamma$-space~$\Fcal$).
The basic idea is as follows. The functor $\Ccal(X)$ is built out of powers
of~$X$, i.e., spaces of the form $X^n$, where maps between them are induced by basepoint inclusions. Suppose $X$ is itself a filtered space, with the basepoint having filtration zero.
Then powers of $X$ are equipped with a product filtration \eqref{eq: product filtration}, and basepoint inclusions are filtration-preserving maps. In this way the product filtration of the spaces $X, X^2, \ldots, X^n, \ldots$ induce a natural filtration of $\Ccal(X)$. (It is, in general, different from the ``usual" one. See Example~\ref{example: filtration one}.) The remainder of the section formalizes this notion.

\begin{definition}         \label{definition: Topfi}
A filtered space is a pointed topological space, together with a filtration by subspaces
\begin{equation}\label{eq: filtered X}
*=X_0 \hookrightarrow X_1 \hookrightarrow \cdots \hookrightarrow X_m \hookrightarrow \cdots X=\bigcup_m X_m.
\end{equation}
As indicated in~\eqref{eq: filtered X}, we require that $X_0$ is the basepoint of~$X$.
\end{definition}

Let $\filtered$ be the category of filtered pointed spaces. We can make use of the previously-defined notation $\filter{m}$ in this context.
In \eqref{eq: unstable rank filtration}, we used $\filter{m}$ for the $m$-th unstable rank filtration of an augmented $\Gamma$-space.
In the context of~$\filtered$, we use it
to denote the operation of taking the $m$-th filtration of a filtered space.

\begin{definition}    \label{definition: filtered space}
Let $\filter{m}\colon\filtered\to \filtered$ be the functor that truncates a filtered space at stage~$m$. So if $X$ is a filtered space as in~\eqref{eq: filtered X}, then $\filter{m} X$ is the filtered space
\[
*=X_0 \hookrightarrow X_1 \hookrightarrow \cdots \hookrightarrow X_m \xrightarrow{=} X_m \xrightarrow{=} \cdots X_m.
\]
\end{definition}

One wants to have wedges and products in $\filtered$, so we need to define appropriate filtrations on those constructions.
Suppose $X$ and $Y$ are filtered spaces. The \defining{wedge filtration} on $X\vee Y$ will play a role in the proof of Lemma~\ref{lemma: filtered Freudenthal}. It is defined by the evident formula:
\begin{equation}    \label{eq: wedge filtration}
\filter{m}{(X\vee Y)}=\filter{m}X \vee \filter{m}Y.
\end{equation}
On the other hand, the \defining{product filtration} on $X\times Y$ is defined by the formula
\begin{equation}   \label{eq: define product filtration}
\filter{m}{\left(X\times Y\right)}
=\bigcup_{m_1+m_2\le m}\filter{m_1}X\times \filter{m_2}Y.
\end{equation}
Inductively, a filtered space $X$ has a product filtration on $X^n$ for $n\geq 0$:
\begin{equation}     \label{eq: product filtration}
\filter{m}\left(X^n\right)
     = \bigcup_{m_1+\cdots+m_n\le m}
          \filter{m_1}X\times\cdots\times\filter{m_n}X.
\end{equation}

Note that the product filtration of $X^n$ is invariant under the action of~$\Sigma_n$,
and we have assumed that the basepoint has filtration zero. Recall that $\Ical$ denotes the category of unpointed finite sets and injections. It follows that if $X$ is an object of~$\filtered$, then the
functor $\Ical\to \Top$ defined by $\finitesetUnpointed{n}\mapsto X^{n}$
is actually a functor from $\Ical$ to~$\filtered$. Hence we can construct filtrations of $\Ccal(X)$ (Example~\ref{example: C}) and $\Sp^{\infty}(X)$ that take the filtration of $X$ into account.

\begin{definition}     \label{definition: define filtration CX}
Suppose $X$ is a filtered space. We define natural filtrations on $\Ccal(X)$ and $\Sp^{\infty}(X)$ by the formulas
\begin{align*}   \label{eq: define filtration CX}
\filter{m}\Ccal(X)
   &\definedas E\Sigma_n \otimes_{n\in\Ical} \filter{m} \left(X^n\right)\\
\filter{m}\Sp^{\infty}(X)
   &\definedas \quad *\   \otimes_{n\in\Ical} \filter{m} \left(X^n\right).
\end{align*}
\end{definition}

Informally speaking, a point in $\filter{m}\Sp^{\infty}(X)$ consists of a finite subset of~$X$, say $\{x_1, \ldots, x_n\}$, subject to the condition that the filtrations of $x_1, \ldots, x_n$ add to at most~$m$. A~point in $\filter{m}\Ccal(X)$ consists of that same data, but ``decorated'' with a point in~$E\Sigma_n$.
With this definition, $\Ccal$ and $\Sp^{\infty}$ will be considered to be functors from $\filtered$ to $\filtered$ for the rest of the paper,
except when explicitly stated otherwise.

\begin{example}\label{example: filtration one}
\hfill
\begin{enumerate}
\item
Suppose that $X$ is a trivially filtered space, i.e. $X$~is concentrated in
filtration~$1$, so $X=\filter{1}X$. 
The filtrations of $\Ccal(X)$ and $\Sp^{\infty}(X)$ in this case are actually the ``usual" ones,
\begin{align*}
\filter{m}\Ccal(X)
    &\simeq
    \left(\coprod_{k=1}^{m}E\Sigma_{k}\times_{\Sigma_{k}}X^k\right)/\sim\\
\filter{m}\Sp^{\infty}(X)
    &=\Sp^{m}(X)
\end{align*}
where $\sim$ indicates the basepoint identifications made by the coend.

\item
Now suppose given a general filtered space~$X$; we compute $\filter{1}\Ccal(X)$. We need to know
$\filter{m}\left(X^n\right)$ when $m=1$, and in this case
there is a homeomorphism
$\filter{1}\left(X^n\right)\cong \left(\filter{1}X\right)^{\vee n}\largestrut $.
Therefore
\[
\filter{1}\Ccal(X) = E\Sigma_n \otimes_{n\in\Ical}
\left(\strut\filter{1}X\right)^{\vee n}.
\]
An easy calculation shows that there is a homeomorphism followed by a homotopy equivalence
\[
\filter{1}\Ccal(X) \cong {E\Sigma_1}_+\wedge \filter{1} X\simeq \filter{1}X.
\]
Observe that this is \emph{not} the same as the ``usual" filtration of~$\Ccal X$.
\end{enumerate}
\end{example}
\medskip

The filtration obtained by applying~$\Ccal$, and more generally iterates of~$\Ccal$, to a filtered space $X$ plays an important role in subsequent sections. In particular, we will be looking at a bar construction using the monad structure for~$\Ccal$, and we need to know that the monad structure maps preserve the filtrations of Definition~\ref{definition: define filtration CX}.
We first note that $\Sp^{\infty}$ has the desired property. Let
$\mu_{\Sp^{\infty}}(X)\colon\Sp^{\infty}\Sp^{\infty}(X)\to \Sp^{\infty}(X)$ and $\eta_{\Sp^{\infty}}(X)\colon X\to \Sp^{\infty}(X)$ be the monad structure maps for~$\Sp^{\infty}$.

\begin{lemma}   \label{lemma: Spinf structure maps}
The maps $\eta_{\Sp^{\infty}}(X)$ and $\mu_{\Sp^{\infty}}(X)$ are maps of filtered spaces.
\end{lemma}

\begin{proof}
The lemma is true for $\eta_{\Sp^\infty}(X)$ by inspection, since $X\cong X^1\subseteq\Sp^{\infty}X$ as a filtered space. For $\mu_{\Sp^{\infty}}(X)$, we need only observe that by~\eqref{eq: product filtration}, the iterated product filtration on the left side of
\[
X^{k_{1}}\times\ldots\times\X^{k_{n}}\rightarrow X^{k_{1}+\ldots+k_{n}}
\]
is isomorphic to the product filtration on the right side.
\end{proof}

\begin{remark}      \label{remark: pullback square}
It follows immediately from Definition~\ref{definition: define filtration CX}  that
the augmentation $\epsilon\colon\Ccal\to\Sp^{\infty}$ is filtration-preserving.
Further, there is a strict pullback diagram:
\[
\xymatrix{
\filter{m}\Ccal(X)
     \ar[d]\ar[r]
     &\Ccal(X)
      \ar[d]\\
\filter{m}\Sp^{\infty}(X)
     \ar[r]
     &\Sp^{\infty}(X).
}
\]
Observe that, unlike \eqref{eq: unstable rank filtration},
the notation $\filter{m}$ in this situation takes the filtration of $X$ into account.
That is, the diagram is not tautological because it expresses a pullback of endofunctors
of~$\filtered$.
\end{remark}

Now let  $\eta_{\Ccal}$ and $\mu_{\Ccal}$ be the monad unit and multiplication maps for~$\Ccal$.

\begin{lemma}\label{lemma: filtered monad}
Let $X$ be a filtered space. Then the structure maps $\mu_{\Ccal}(X)\colon\Ccal\Ccal(X)\to \Ccal(X)$ and $\eta_{\Ccal}(X)\colon X\to \Ccal X$ for the monad~$\Ccal$ preserve the filtration of Definition~\ref{definition: define filtration CX}.
\end{lemma}

\begin{proof}
Consider the following cube, with one missing arrow:
\[
\xymatrix@=1em
{
& \filter{m}(\Ccal\Ccal X)\, \ar[rr]\ar'[d][dd]
& & \Ccal\Ccal X \ar[dl]^-{\mu_{\Ccal}}
\ar[dd]^-{\epsilon\circ\epsilon}
\\
\filter{m}(\Ccal X)\, 
                    \ar[rr]\ar[dd]
& & \Ccal X \ar[dd]^(.35){\epsilon}
\\
& \filter{m}(\Sp^{\infty}\Sp^{\infty} X)\, \ar'[r][rr]\ar[dl]
& & \Sp^{\infty}\Sp^{\infty}X \ar[dl]^-{\mu_{\Sp^{\infty}}}
\\
\filter{m}(\Sp^{\infty}X)\, \ar[rr]
& & \Sp^{\infty}X 
}
\]
The bottom face commutes by Lemma~\ref{lemma: Spinf structure maps}.
The front and back faces are pullback squares,
as indicated in Remark~\ref{remark: pullback square}.
The right face commutes because $\epsilon\colon\Ccal\to\Sp^{\infty}$ is a map of monads.
Because the cube commutes and the front face is a strict pullback, we obtain a unique map $\filter{m}(\Ccal\Ccal X)\to\filter{m}(\Ccal X)$ that makes the entire diagram commute, which establishes that $\mu_{\Ccal}(X)$ is filtration preserving.
\end{proof}

For the unit map, observe that
$X\xrightarrow{\simeq} E\Sigma_{1}\times X\subseteq\Ccal X$, and
the filtration of $E\Sigma_{1}\times X$ comes entirely from the filtration of~$X$.

\begin{lemma}    \label{lemma: boring lemma}
Let $X$ and $Y$ be objects of~$\filtered$. Then the map
\[
\Ccal(X\times Y)\xrightarrow{\,\Ccal\pi_{1}\times\Ccal\pi_{2}\,}
      \Ccal X\times\Ccal Y
\]
preserves filtration.
\end{lemma}

\begin{proof}
We consider a similar diagram to the one in the previous lemma:
\[
\xymatrix@=1em
{
& \filter{m}(\Ccal (X\times Y)) \ar[rr]\ar'[d][dd]
& & \Ccal(X\times Y) \ar[dl]^(.4){\Ccal\pi_{1}\times\Ccal\pi_{2}  }\ar[dd]^-{\epsilon}
\\
\filter{m}(\Ccal X\times\Ccal Y) 
                    \ar[rr]\ar[dd]
& & \Ccal X\times\Ccal Y \ar[dd]^(.3){\epsilon\times\epsilon}
\\
& \filter{m}(\Sp^{\infty}(X\times Y)) \ar'[r][rr]\ar@{-->}[dl]
& & \Sp^{\infty}(X\times Y) \ar[dl]^(.4){\qquad\Sp^{\infty}\pi_{1}\times\Sp^{\infty}\pi_{2}}
\\
\filter{m}(\Sp^{\infty}X \times\Sp^{\infty}Y) \ar[rr]
& & \Sp^{\infty}X \times\Sp^{\infty}Y.
}
\]
Once we show that the dotted arrow exists,
the lemma follows exactly as in the proof of Lemma~\ref{lemma: filtered monad}.

A point in $ \filter{m}(\Sp^{\infty}(X\times Y))$ is a collection of pairs $(x_i, y_i)$ such that the sum of filtrations is at most $m$. Such a point is mapped
by $\Sp^{\infty}\pi_{1}\times\Sp^{\infty}\pi_{2}$ to the pair $(\{x_i\},\{y_i\})$ where the first term is a collection of points in $X$ the second is a collection of points in $Y$. The sum of the filtrations is still at most~$m$, which means that this point belongs to $\filter{m}(\Sp^{\infty}X \times\Sp^{\infty}Y)$.
As a result, the dotted arrow exists, and the lemma follows.
\end{proof}

\section{Operad actions on augmented special $\Gamma$-spaces}
\label{section: operad actions} 

In this section we review the connection between the operadic and $\Gamma$-space approaches to infinite loop spaces by recalling 
the action of an $E_\infty$-operad on a special $\Gamma$-space. We then show that if $\Fcal$ is an augmented special $\Gamma$-space, then {\it any} action of an $E_\infty$-operad on $\Fcal$ respects the augmentation.

\begin{remark*}
Strictly speaking, the general result of this section 
(Proposition~\ref{proposition: augmentation preserved}) is not needed
for the examples to which we want to apply the rest of the paper. In all of those examples
(such as Example~\ref{example: bu example operad action}), there is a specific operad acting explicitly on a given $\Gamma$-space, and it is easy to verify directly that the action preserves the augmentation. But it is reassuring to know that in fact the action must be augmentation-preserving for general reasons. We thank the referee for prompting us to think this through.
\end{remark*}

\noindent {\bf\smallstrut{Notation.}} Throughout this section, we sometimes label arrows in diagrams by the underlying map that induces them, leaving the reader to infer from context what natural transformations have been applied to the underlying map. In general, we use $\mu$~for an operad action, $\lambda$~for a ``linearization" map \eqref{eq: special equivalence}, $\epsilon$~for an augmentation, and $\alpha$~for the assembly map \eqref{eq: assembly}. \\

To set the stage, let  $\Ocal=\{E{\Sigma_k}\mid k\ge 0\}$ denote a generic $E_\infty$-operad, and let 
$\Fcal$ be a special $\Gamma$-space. An \defining{action} of~$\Ocal$ on $\Fcal$ consists of natural transformations (one for each $k\ge 0$)
\[
\mu\colon E\Sigma_k \times \Fcal^k\to \Fcal.
\]
The natural transformations are required to satisfy the usual symmetry, associativity, and unit relations. 
It is known that for every special $\Gamma$-space $\Fcal$ there is an $E_\infty$-operad acting on it. Segal outlined at the end of~\cite{Segal-Categories} a general construction that associates to a special $\Gamma$-space an $E_\infty$-operad acting on it. A variant of this statement was proved by Mandell in~\cite[Theorems 1.5 and 1.6]{Mandell-Inverse}. In Mandell's variant, he fixes a particular $E_\infty$-operad (the Barratt-Eccles operad) and shows that every special $\Gamma$-space is equivalent to one on which the Barrat-Eccles operad acts.



In addition to an action of an $E_\infty$-operad, one could consider an action of the commutative operad, whose spaces are all given by points. In this case, an action consists of natural transformations
\[
* \times \Fcal^k\to \Fcal,
\]
again with appropriate symmetry, associativity, and unit relations. That is, a $\Gamma$-space with an action of the commutative monoid takes values in strict commutative monoids; its stabilization is equivalent to a product of Eilenberg-\MacLane spectra.
The canonical example is the special $\Gamma$-space $\Sp^\infty$, which admits an action of the commutative operad via
\[
*\times \Sp^\infty(\whatever)^k 
     \xleftarrow{\ \cong\ }
     \Sp^\infty(\finiteset{k}\wedge \whatever)
     \longrightarrow 
     \Sp^\infty(\whatever),
\]
the second map being induced the fold map $\finiteset{k}\rightarrow\finiteset{1}= S^{0}$. 
It follows that $\Sp^\infty$ also admits a canonical action of any $E_\infty$-operad via the map $E\Sigma_k\to *$.

When evaluated on finite sets, the $\Gamma$-space 
$\Sp^\infty(\whatever)$ takes values in discrete spaces. 
It follows that a natural transformation
$\Fcal(\whatever)\rightarrow\Sp^\infty(\whatever)$ is completely determined by its induced transformation $\pi_{0}\Fcal(\whatever)\rightarrow\pi_{0}\Sp^\infty(\whatever)$.
This observation will be used a couple of times in this section. In the next lemma we use it to prove that the 
property of possessing an augmentation is 
a homotopy invariant of $\Gamma$-spaces.

\begin{lemma}\label{lemma: augmentation invariance}
Suppose $\Fcal$ is a $\Gamma$-space equipped with an augmentation $\Fcal\to \Sp^\infty$ and an equivalence to another $\Gamma$-space $\Fcal\xrightarrow{\simeq}\Gcal$. Then there is a unique augmentation $\Gcal\to \Sp^\infty$ that forms a commuting triangle.
\end{lemma}

\begin{proof}
For every finite set~$X$, the space $\Sp^\infty(X)$ is discrete. We have a diagram
\[
\Sp^\infty(X)\longleftarrow \Fcal(X) \xlongrightarrow{\simeq} \Gcal(X),
\]
and hence there is a {\it unique} map $\Gcal(X)\to \Sp^\infty(X)$ that completes this diagram to a commutative triangle. Because of the uniqueness, this map is natural in $X$ and is therefore a natural transformation 
$\Gcal\rightarrow\Sp^\infty$ of $\Gamma$-spaces.
\end{proof}

The following proposition is the main result of this section. It says that an action of an $E_\infty$-operad on an augmented $\Gamma$-space is automatically augmentation-preserving. 

\begin{proposition}\label{proposition: augmentation preserved}
Suppose $\Fcal$ is an augmented special $\Gamma$-space, and $\Ocal$ is an $E_\infty$-operad acting on $\Fcal$. 
Then the action of $\Ocal$ is augmentation-preserving, in the sense that
for every~$k$, the following diagram commutes:
\begin{equation*}
\begin{gathered}
\xymatrix{
E\Sigma_k \times \Fcal^k
    \ar[r]^-{\mu}\ar[d]_-{{*}\times\epsilon^k}   & \Fcal\ar[d]_-{\epsilon} \\
{*} \times \left(\Sp^{\infty}\right)^k
    \ar[r]^-{\mu} &\Sp^\infty.
}
\end{gathered}
\end{equation*}
\end{proposition}


Before we get to the proof of
Proposition~\ref{proposition: augmentation preserved}, we need to develop some tools, the first of which is the assembly map. Let $X_1, \ldots, X_k$ be pointed finite sets, 
and define $X=X_1\vee\cdots\vee X_k$. If $\Fcal$ is a $\Gamma$-space, 
the pinch maps $X\rightarrow X_{i}$ induce a ``linearization" map
\begin{equation} \label{eq: special equivalence} 
\lambda\colon \Fcal(X)\longrightarrow  \Fcal(X_1)\times \cdots \times \Fcal(X_k).
\end{equation} 
If $\Fcal$ is special, then \eqref{eq: special equivalence} is 
a weak equivalence, but there is not an immediately apparent homotopy inverse. 
However, because $\Fcal$ is 
special, it has an action of an $E_{\infty}$-operad. 
Our next lemma says that {\emph{any}} such action allows us to construct a
kind of inverse for \eqref{eq: special equivalence}, as follows. 
For each~$i$ the inclusion 
$X_i\hookrightarrow X$ induces a map 
$\Fcal\left(X_i\right)\hookrightarrow \Fcal(X)$. Taking the product 
of these maps over $i$, and composing with the operad action on~$\Fcal$,
defines the \defining{assembly map,} 
\begin{equation}\label{eq: assembly}
\alpha\colon E\Sigma_k\times \Fcal(X_1)\times \cdots \times \Fcal(X_k) 
      \longrightarrow E\Sigma_k\times \Fcal(X)^k 
      \longrightarrow \Fcal(X).
\end{equation}

\begin{lemma}\label{lemma: reverse}
Let $\Fcal$ be a special $\Gamma$-space, and fix an action of an $E_\infty$-operad. Let $X=X_1\vee\cdots\vee X_k$ be a finite pointed set, as above. Then the composite 
$\lambda\circ\alpha$ is homotopic to the map that projects off~$E\Sigma_{k}$. 
\end{lemma}


\begin{proof}
Consider the following homotopy commutative diagram, 
where the bottom square's vertical equivalences come from the assumption that $\Fcal$ is special,
and the horizontal maps are given by the operad action
(in the bottom row, the product over $i$ of the action for each $\Fcal(X_{i})$): 
\begin{equation}\label{diagram: assembly}
\begin{gathered}
\xymatrix{
E\Sigma_k \times \Fcal(X_1)\times \cdots \times \Fcal(X_k) 
      \ar[d]\ar[rd]^-{\alpha} & \\ 
E\Sigma_k \times\left(\largestrut\Fcal(X)\right)^{k} 
      \ar[r]^{\mu}\ar[d]_{\lambda^k}^\simeq
    &\Fcal(X)
      \ar[d]_{\lambda}^\simeq\\
E\Sigma_k \times \left(\largestrut\Fcal(X_1)\times \cdots \times \Fcal(X_k)\right)^k 
     \ar[r]_-{\mu^k}
     &\Fcal(X_1)\times \cdots \times \Fcal(X_k).
}
\end{gathered}
\end{equation}
The composite of the left column sends each $\Fcal(X_i)$ by the identity map to the $i$-th copy of~$\Fcal(X_i)$ (the ``diagonal" copy) and to the basepoint of the non-diagonal copies
of~$\Fcal(X_i)$. However, the bottom row maps into $\Fcal(X_{i})$ in
the bottom right corner via the operad action 
$E\Sigma_{k}\times \Fcal(X_{i})^k\rightarrow\Fcal(X_i)$, and the basepoint
of $\Fcal(X_{i})$ acts as the identity. It follows that the counterclockwise composite map from top left to bottom right
simply projects off~$E\Sigma_{k}$. Hence the clockwise composition does likewise. 
\end{proof}

\begin{remark*} 
The ``linearization" map $\lambda$ of~\eqref{eq: special equivalence} is an 
equivalence of symmetric multi-functors when $\Fcal$ is special. 
In view of Lemma~\ref{lemma: reverse}, we like to think of the map 
$\alpha$ as a homotopy inverse to this equivalence. Notice however that the source of $\alpha$ is not the target 
of~\eqref{eq: special equivalence}, but instead plays the role 
typically played by a cofibrant replacement. 
\end{remark*}

The second tool we need for the proof of Proposition~\ref{proposition: augmentation preserved} 
is a connection between an operad action and the fold map.
On the one hand, the map $\finiteset{k}\rightarrow\finiteset{1}$ that takes $1,\ldots,k$ to $1$ induces the fold map 
$\bigvee_{k}X=\finiteset{k}\wedge X\rightarrow X$ and therefore a map 
\[
\Fcal(\finiteset{k}\wedge X) \xrightarrow{\Fcal(\rm{fold})} \Fcal(X).
\]
On the other hand, we have the linearization map \eqref{eq: special equivalence},
\[
\lambda\colon \Fcal(\finiteset{k}\wedge X)\rightarrow 
  \Fcal(\finiteset{1}\wedge X)^k=\Fcal(X)^k,
\]
the map whose components are induced by sending each non-basepoint element of 
$\finiteset{k}$ in turn to the 
non-basepoint element of $\finiteset{1}$ and which is an homotopy equivalence when $\Fcal$ is special. 

\begin{lemma}\label{lemma: inverse assembly}
Let $\Fcal$ be a special $\Gamma$-space, with a fixed action of an $E_{\infty}$-operad. The following diagram is homotopy commutative, where the horizontal maps are, respectively, projection and the operad action:
\[
\xymatrix{
E\Sigma_k \times \Fcal(\finiteset{k}\wedge X)
     \ar[r]^-{\pi}
     \ar[d]_-{\lambda}^-{\simeq}
&
\Fcal(\finiteset{k}\wedge X)\ar[d]^{\Fcal(\rm{fold})}
\\
  E\Sigma_k \times \Fcal(\finiteset{1}\wedge X)^k \ar[r]_-{\mu}
  & \Fcal(\finiteset{1}\wedge X).
}
\]
\end{lemma}

\begin{proof}
By Lemma~\ref{lemma: reverse}, the projection map $\pi$ (top row) factors as the composite 
of $\lambda$ (left column) and the assembly map~$\alpha$ 
\eqref{eq: assembly} for the case $X_1=\cdots =X_k=X$.
We reproduce the diagram of interest below, with $\alpha$ inserted as the upward diagonal composite. The upper triangle commutes: 
\[
\xymatrix{
E\Sigma_k \times \Fcal(\finiteset{k}\wedge X)
     \ar[dd]_-{\lambda}
     \ar[rr]^{\pi}
 && \Fcal(\finiteset{k}\wedge X)\ar[dd]^{{\rm{fold}}}\\ 
&E\Sigma_k \times \Fcal(\finiteset{k}\wedge X)^k 
     \ar[ur]_-{\mu}
     \ar@/^.5pc/@<.1ex>[dl]^-{\Fcal({\rm{fold}})^k}
      \\
E\Sigma_k \times \Fcal(\finiteset{1}\wedge X)^k 
     \ar[rr]_-{\mu}
     \ar@/^.5pc/@<-.1ex>[ur]
     &&\Fcal(\finiteset{1}\wedge X).
}
\]
The key point is that the first map in $\alpha$ (lower left corner to middle)  is actually a section of the map induced by the fold map; that is, going from the bottom left corner to the middle and then back again is the identity. Further, the lower right triangle commutes by naturality of the operad action, and the lemma follows. 
\end{proof}

Now we are ready to prove that any $E_{\infty}$-operad action on an augmented special $\Gamma$-space must preserve the augmentation.

\begin{proof}[Proof of Proposition~\ref{proposition: augmentation preserved}]
Recall that $\Fcal$ is a special $\Gamma$-space equipped with a natural transformation
$\epsilon\colon\Fcal\rightarrow\Sp^{\infty}$. Let $X$ be a pointed finite set. We must show that the right square in the following diagram (strictly) commutes: 
\[
\xymatrix{
E\Sigma_k \times \Fcal(\finiteset{k}\wedge X) 
       \ar[d]_{\ast\times\epsilon}\ar[r]_-\simeq^-{\lambda} 
    & E\Sigma_k \times \Fcal(X)^k \ar[r]^-{\mu}\ar[d]_{\ast\times\epsilon^k}  
    & \Fcal(X)\ar[d]_{\epsilon} \\
\ast\times\Sp^\infty(\finiteset{k}\wedge X) \ar[r]_-\cong 
    &\ast\times \Sp^{\infty}(X)^k \ar[r]_-\cong 
    &\Sp^\infty(X).
}
\]
The composite horizontal map in the bottom row is induced by the folding map (on the nose), and by Lemma~\ref{lemma: reverse} the same is true up to homotopy for the top row.
Hence the outer rectangle commutes up to homotopy. On the other hand,
the left square is strictly commutative. Hence the right square commutes up to
homotopy, and since $\Sp^\infty(X)$ is discrete, the right square commutes strictly.
\end{proof}

To close this section, we return to the special $\Gamma$-space of 
Example~\ref{example: U}. As we mentioned before, there is a general construction that associates to a special $\Gamma$-space an $E_\infty$-operad acting on it. But for a specific $\Fcal$ there may be a particular $E_\infty$-operad acting naturally on $\Fcal$. The following example will play a major role in future applications. 

\begin{example}  \label{example: bu example operad action} 
Recall that in Example~\ref{example: U} we defined
\[
\Fcal_U\left(\strut\finiteset{m}\right)
   =\coprod_{(k_1,\ldots, k_m)\in \naturals^m}
    \Inj\left(\complexes^{k_1+\cdots+k_m}, \complexes^\infty\right)
          /_{U(k_1)\times\cdots\times U(k_m)}.
\]
To streamline notation, we appeal to the concept of Grassman manifolds and write
\[
\Gr\left(k_1, \ldots, k_m\right)= \Inj(\complexes^{k_1+\cdots+k_m}, \complexes^\infty)/_{U(k_1)\times\cdots\times U(k_m)},
\]
so that we have
\[
\Fcal_U\left(\strut\finiteset{m}\right)
   =\coprod_{(k_1,\ldots, k_m)\in \naturals^m}
   \Gr\left(k_1, \ldots, k_m\right) .
\]

We want to use the linear isometries operad as the $E_\infty$-operad acting on~$\Fcal_U$.
Since its $n$-th space is
$\Inj\left(\Largestrut(\complexes^\infty)^n, \complexes^\infty\right)$,
we need to describe maps of the form
\begin{equation}\label{eq: isometries action}
\Inj\left(\Largestrut(\complexes^\infty)^n, \complexes^\infty\right)
\times
\Fcal_U\left(\strut\finiteset{m}\right)^n \to \Fcal_U\left(\strut\finiteset{m}\right).
\end{equation}
We proceed one component at a time. Suppose that $\Gr_{1}, \ldots,\Gr_{n}$ are components of
$\Fcal_U\left(\strut\finiteset{m}\right)$, say with
\begin{align*}
\Gr_{1}&=\Gr\left(i_{1,1}, \ldots, i_{1,m}\right)\\
\Gr_{2}&=\Gr\left(i_{2,1}, \ldots, i_{2,m}\right)\\
  \vdots \\
\Gr_{n}&=\Gr\left(i_{n,1}, \ldots, i_{n,m}\right).
\end{align*}
To find the component of $\Fcal_U\left(\strut\finiteset{m}\right)$ that will be the
target of~\eqref{eq: isometries action}, we add together the dimensions of corresponding subspaces
in $\Gr_{1},\ldots,\Gr_{n}$:
\[
k_{1}=\sum_{j=1}^{n} i_{j,1},\ \ldots\ ,k_{m}=\sum_{j=1}^{n} i_{j,m}.
\]
To define a map
\begin{equation}\label{eq: action map U}
\Inj\left(\Largestrut(\complexes^\infty)^n, \complexes^\infty\right)
\times
\Gr_{1}\times \cdots\times \Gr_{n}
\rightarrow
\Gr\left(k_{1},\ldots,k_{m}\right),
\end{equation}
let $(\alpha,f_{1},...,f_{n})$ represent an element of the domain. Hence
$\alpha\colon \left(\complexes^{\infty}\right)^{n}\rightarrow\complexes$, and
$f_{j}:\complexes^{i_{j,1}+...+i_{j,m}}\rightarrow\complexes^{\infty}$.
We define the action map \eqref{eq: action map U} by
\[
\left(\strut\alpha, [f_1], \ldots, [f_n]\right)
\mapsto
\left[\strut\alpha\circ(f_1\times\cdots\times f_n)\right].
\]

Let us check from first principles 
that the action of the linear isometries operad on 
$\Fcal_U$ is augmentation-preserving. We need to check commutativity of the diagram
\begin{equation*}
\begin{gathered}
\xymatrix{
\Inj\left(\largestrut(\complexes^\infty)^n, \complexes^\infty\right)
         \times \left(\Fcal_{U}\right)^n
    \ar[r]\ar[d]   & \Fcal_{U}\ar[d]^-{\epsilon} \\
\left(\Sp^{\infty}\right)^n
    \ar[r]  &\Sp^\infty.
}
\end{gathered}
\end{equation*}
Going around clockwise, we apply \eqref{eq: action map U}, and then the augmentation
\eqref{eq: define augmentation FU} to obtain $(k_{1},\ldots,k_{m})\in\naturals^{m}$.
Going around counterclockwise, we have
$\Gr_{1},\ldots,\Gr_{n}$ mapping, respectively, to $\left(i_{1,1}, \ldots, i_{1,m}\right)$, $\ldots$,
$\left(i_{n,1}, \ldots, i_{n,m}\right)$ in~$\naturals^{m}$. The monoid multiplication
$\left(\Sp^{\infty}\right)^n\rightarrow\Sp^\infty$ adds coordinate-wise to obtain
$(k_{1},\ldots,k_{m})\in\naturals^{m}$, and so the diagram commutes as required.
\end{example}

\section{The bar construction and the unstable rank filtration}
\label{section: bar}
In this section, we take an augmented special $\Gamma$-space~$\Fcal$ and set up a two-sided bar construction
\[
\Barr(\Ccal, \Ccal, \Fcal)\definedas\Ccal\Ccal^{\bullet} \Fcal.
\]
The bar construction is to be a simplicial object over {\emph{augmented}} $\Gamma$-spaces, so we use the results of Section~\ref{section: operad actions} to establish that actually any $E_{\infty}$-operad action on $\Fcal$ gives a corresponding action of $\Ccal$ that respects the augmentation. 
Finally, we exhibit a straightforward filtration of $\Barr(\Ccal, \Ccal, \Fcal)$
by simplicial subspaces denoted
$\filter{m}\Barr(\Ccal, \Ccal, \Fcal)$ or $\filter{m}\Ccal\Ccal^{\bullet}\Fcal$, whose geometric realizations are equivalent to
the \emph{unstable} rank filtration of~$\Fcal$~\eqref{eq: unstable rank filtration}. Our principal goal for this paper, however, is achieved in Section~\ref{section: stable filtration}, where we explain how to filter
$\Barr(\Ccal, \Ccal, \Fcal)$ to get the \emph{stable} rank filtration, which is a more subtle matter.

\bigskip

Having discussed the concept of an action of an $E_{\infty}$-operad~$\Ocal$ on an augmented special $\Gamma$-space~$\Fcal$ in Section~\ref{section: operad actions}, we return to the monad $\Ccal$ associated to~$\Ocal$ (see~\eqref{eq: coend formula}).
The action $\mu$ of the operad $\Ocal$ on $\Fcal$ factors through 
the quotient map to $\Ccal\Fcal$, thereby giving a left action
$\mu_{\Fcal}$ of the monad $\Ccal$ on~$\Fcal$ \cite[Proposition 2.8]{May-Geometry}:
\begin{equation}    \label{eq: Ccal action}
\begin{gathered}
\xymatrix{
\coprod_{n\in\naturals}^{}E\Sigma_{n}\times\Fcal^{n}
    \ar@{->>}[dr]
    \ar[rr]^-{\mu}
     && \Fcal\\
&\Ccal\Fcal \ar[ur]_-{\mu_{\Fcal}}
     }
\end{gathered}
\end{equation}
That is, there is a natural transformation
$\mu_{\Fcal}\colon\Ccal \Fcal\to \Fcal$, satisfying the usual associativity and unit relations.
Further, we saw in Proposition~\ref{proposition: augmentation preserved}
that an $E_\infty$-operad acting on a special $\Gamma$-space is automatically augmentation-preserving. It follows that the associated monad action is augmentation-preserving as well, in a sense that is made precise in the following lemma.

\begin{lemma}\label{lemma: augmentation}
For any action of an $E_{\infty}$-operad $\Ocal$ on a special $\Gamma$-space~$\Fcal$, the associated action of the monad $\Ccal$ on~$\Fcal$ preserves the augmentation,
in the sense that the following diagram commutes:
\begin{equation*}    
\begin{gathered}
\xymatrix{
\Ccal \Fcal
    \ar[d]^-{\epsilon\circ\epsilon}
    \ar[r]^-{\mu_{\Fcal}}
     & \Fcal \ar[d]^-{\epsilon}\\
\Sp^\infty \Sp^\infty \ar[r]_-{\mu_{\Sp^{\infty}}}
     &\Sp^\infty.
     }
\end{gathered}
\end{equation*}
\end{lemma}

\begin{proof}
We unwind diagram~\eqref{eq: Ccal action} to be the top row of the following diagram, with vertical maps given by augmentations: 
\begin{equation*}    
\begin{gathered}
\xymatrix{
\displaystyle\coprod_{n\in\naturals}^{}E\Sigma_{n}\times\Fcal^{n}
    \ar@{->>}[r]
    \ar[d]_{*\times\epsilon^n}
&\Ccal \Fcal
    \ar[d]^{\epsilon\circ\epsilon}
    \ar[r]^{\mu_{\Fcal}}
& \Fcal \ar[d]^{\epsilon}\\
\displaystyle\coprod_{n\in\naturals}^{} *\times\left(\Sp^{\infty}\right)^{n}
    \ar@{->>}[r]
&\Sp^\infty \Sp^\infty
    \ar[r]_-{\mu_{\Sp^{\infty}}}
&\Sp^\infty.
     }
\end{gathered}
\end{equation*}
By Proposition~\ref{proposition: augmentation preserved} the outer rectangle commutes, and we must prove that the right-hand square commutes.

However, the left square of the diagram commutes, by naturality of two kinds: (i)~the map of operads $\{E\Sigma_{n}\}$ to $\{*\}$ from the $E_{\infty}$-operad to the commutative operad, and (ii) the natural transformation $\Fcal\rightarrow\Sp^{\infty}$. Finally, the maps beginning the horizontal rows are epimorphisms, which forces the right-hand square to commute as well.
\end{proof}

We are ready for the question of filtering $\Barr(\Ccal, \Ccal, \Fcal)$ by simplicial subspaces in order to obtain the unstable rank filtration. 
Given an action of an $E_\infty$-operad~$\Ocal$ on an augmented special $\Gamma$-space~$\Fcal$, 
Lemma~\ref{lemma: augmentation} indicates that we have an augmentation-preserving action of the associated monad $\Ccal$ on~$\Fcal$.
As we said in the introduction, one may now form the two-sided bar construction,
 $\Barr(\Ccal, \Ccal, \Fcal)$, also denoted~$\Ccal\Ccal^\bullet\Fcal$.
It is a simplicial \emph{augmented} $\Gamma$-space, whose $k$-th simplicial dimension is the augmented $\Gamma$-space~$\Ccal^{k+1}\Fcal$, and
which is also simplicially augmented to~$\Fcal$:
\begin{equation}\label{eq: resolution of Fcal}
\Fcal
\longleftarrow
\left\{\Largestrut
\xymatrix{
\Ccal \Fcal\
   \ar[r]
&\ \Ccal^{2}\Fcal\
         \ar@<-.6ex>[l]
         \ar@<.6ex>[l]
         \ar@<-.5ex>[r]
         \ar@<.5ex>[r]
&\ \Ccal^{3}\Fcal\
         \ar@<-1.0ex>[l]
         \ar[l]
         \ar@<1.0ex>[l]
\ldots
\quad\Ccal^{k+1}\Fcal\quad
\ldots
}
\right\}.
\end{equation}
Denoting the monad multiplication of $\Ccal$ by  $\mu_{\Ccal}\colon\Ccal\Ccal\rightarrow\Ccal$, the face
and degeneracy maps are given by
\begin{alignat*}{3}
d_{i} &= \Ccal^{i}\,\mu_{\Ccal}\left(\Ccal^{k-i-1}\Fcal\right)
       &\colon \Ccal\Ccal^k \Fcal \to  \Ccal\Ccal^{k-1} \Fcal
       & \mbox{\quad  for } i=0, \ldots, k-1,
\\
d_{k} &= \qquad\Ccal^{k}\,\mu_{\Fcal}
       &\colon \Ccal\Ccal^k \Fcal \to  \Ccal\Ccal^{k-1} \Fcal,
       \\
s_{i} &= \Ccal\Ccal^{i}\,\eta_{\Ccal}(\Ccal^{k-i}\Fcal)
       &\colon \Ccal\Ccal^k \Fcal \to  \Ccal\Ccal^{k+1} \Fcal
       & \mbox{\quad  for } i=0, \ldots, k.
\end{alignat*}

It follows from Lemma~\ref{lemma: filtered monad}, coupled with Lemma~\ref{lemma: augmentation}, that the face and degeneracy maps respect the
natural filtration of $\Barr(\Ccal, \Ccal, \Fcal)$. By this we mean that for each value of $k, i,$ and $m$, the face map $d_i$ restricts to a map
\[
\filter{m}{\Ccal\Ccal^k \Fcal} \to  \filter{m}{\Ccal\Ccal^{k-1} \Fcal}
\]
and likewise for the degeneracy map $s_{i}$.
As a consequence, $\Barr(\Ccal, \Ccal, \Fcal)$ is an augmented simplicial object in the category of filtered $\Gamma$-spaces.

In the usual way, one thinks of $\Barr(\Ccal,\Ccal,\Fcal)$ as a \emph{resolution} of $\Fcal$ by free $\Ccal$-algebras. This is because the unit map
$\eta_{\Ccal}(\Fcal)\colon \Fcal\rightarrow\Ccal\Fcal$, provides an extra degeneracy for the simplicial object $\Barr(\Ccal, \Ccal, \Fcal)$, and hence the natural map $|\Barr(\Ccal,\Ccal,\Fcal)|\to \Fcal$ is an equivalence.
Proposition~\ref{proposition: tautology} below establishes that
in our setting, where we know that the action of $\Ccal$ on~$\Fcal$ is augmentation-preserving, the map $|\Barr(\Ccal,\Ccal,\Fcal)|\to \Fcal$ is actually an equivalence of
\emph{filtered} objects.

\begin{definition}
A map of filtered spaces $X\to Y$ is a \defining{filtered equivalence} if it restricts to a weak equivalence $\filter{m}X\xrightarrow{\ \simeq\ } \filter{m}Y$ for each $m\ge 0$.
\end{definition}

\begin{proposition}\label{proposition: tautology}
Let $\Fcal$ be an augmented special $\Gamma$-space with a chosen action of an $E_{\infty}$-operad. Then the action map $\mu_{\Fcal}\colon\Ccal\Fcal\to \Fcal$ induces a filtered equivalence $\realization{\Ccal\Ccal^\bullet \Fcal}\xrightarrow{\simeq} \Fcal$. The unit map $\eta_{\Ccal}\colon \Fcal\to \Ccal\Fcal$ induces the inverse filtered equivalence
$\Fcal\xrightarrow{\simeq} \realization{\Ccal\Ccal^\bullet \Fcal}$.
\end{proposition}

\begin{proof}
The unit map $\eta_{\Ccal}\colon\Fcal\to \Ccal\Fcal$ preserves filtration
(Lemma~\ref{lemma: filtered monad}),
and therefore endows $\Ccal\Ccal^\bullet \Fcal$ with a filtered extra degeneracy.
This means that for each~$m$, we get an augmented simplicial subobject of~\eqref{eq: resolution of Fcal}
that is equipped with an extra degeneracy:
\[
\xymatrix{
\filter{m}{\Fcal}
   {\bf\ar@{-->}[r]<-.8ex>}
&
\mbox{\huge $\{$} \filter{m}{\Ccal\Fcal}\
   \ar[r]
   \ar[l]
   \ar@{-->}[r]<-1.4ex>
&\ \filter{m}{\Ccal^{2}\Fcal}\
         \ar@<-.6ex>[l]
         \ar@<.6ex>[l]
         \ar@<-.5ex>[r]
         \ar@<.5ex>[r]
         \ar@{-->}[r]<-1.9ex>
&\ \filter{m}{\Ccal^{3}\Fcal}\
         \ar@<-1.0ex>[l]
         \ar[l]
         \ar@<1.0ex>[l]
\ldots
\quad\filter{m}{\Ccal^{k+1}\Fcal}\quad
\ldots \mbox{\huge $\}$}.
}
\]
It follows that an equivalence is induced by the augmentation
$\filter{m}{\Ccal\Fcal} \to \filter{m}{\Fcal}$,
giving
$\realization{\filter{m}\Ccal\Ccal^\bullet \Fcal}
\xrightarrow{\,\simeq\,}  \filter{m}\Fcal$,
with the inverse equivalence given by the unit map. See~\cite[Section 4.5]{Riehl-Categorical} for more detail on simplicial objects with an extra degeneracy.
\end{proof}

\section{The stable rank filtration of the bar construction}
\label{section: stable filtration}

We continue to use $\Fcal$ to denote an augmented special $\Gamma$-space,
which is equipped with an action of~$\Ccal$. By Lemma~\ref{lemma: augmentation}, the action preserves the augmentation. In Section~\ref{section: bar}, we showed a straightforward way to filter $\Ccal\Ccal^{\bullet}\Fcal$ by simplicial $\Gamma$-subspaces: the $m$-th filtration of $\Ccal\Ccal^{\bullet}\Fcal$ has $\filter{m}\Ccal\Ccal^{k}\Fcal$ in the $k$-th
simplicial dimension. The face and degeneracy maps are respected for straightforward reasons, and so we obtain simplicial $\Gamma$-spaces $\filter{m}\Ccal\Ccal^{\bullet}\Fcal$, and the resulting realizations give the
\emph{unstable} rank filtration of~$\Fcal$.

In the current section we discuss another, less obvious way to filter $\Ccal\Ccal^{\bullet}\Fcal$. We filter \emph{inside} the outermost factor of~$\Ccal$ in the double bar construction to define subobjects~$\Ccal\filter{m}\Ccal^{\bullet}\Fcal$. The goal is to show that this new filtration produces special simplicial $\Gamma$-spaces, and that their infinite deloopings are the \emph{stable} rank filtration of~$\Fcal$~\eqref{eq: rank}, which is the main goal of this paper (Theorem~\ref{theorem: bar filtration}).
There are three key adjectives here: in order of logical dependence, ``simplicial," ``special," and ``stable."

We outline the section by describing the three steps required in order to accomplish the goal of the previous paragraph. First, the adjective ``simplicial": it is not immediately obvious that for a fixed~$m$ the $\Gamma$-subspaces $\Ccal\filter{m}\Ccal^{k}\Fcal$
actually assemble over~$k$ into a simplicial $\Gamma$-space. Hence the first part of this section is devoted to proving the following proposition. (The proof of Proposition~\ref{proposition: structure maps respect filtration} begins with the paragraph preceding
Lemma~\ref{lemma: switch}.)

\begin{proposition}  \label{proposition: structure maps respect filtration}
The face and degeneracy maps of
$\Barr\left(\Ccal,\Ccal,\Fcal\right)=\Ccal\Ccal^\bullet\Fcal$ respect
the filtration of $\Ccal\Ccal^{\bullet}\Fcal$ by $\FilteredBK{m}{\Fcal}$.
That is, $\FilteredBK{m}{\Fcal}$ is a simplicial $\Gamma$-subspace of $\Barr\left(\Ccal,\Ccal,\Fcal\right)$.
\end{proposition}

Second, the adjective ``special."
The individual $\Gamma$-spaces $\Ccal\filter{m}\Ccal^{k}\Fcal$ are not special for most values of $m$ and~$k$, so it
is somewhat surprising that, when we fix~$m$, assemble over~$k$,
and apply geometric realization, we obtain a special $\Gamma$-space.
We were not able to see an obvious reason for this to happen, but it does, and this is the major technical result of this paper.

\begin{theorem}\label{theorem: linearity theorem}
\LinearityTheoremText
\end{theorem}

The summarized proof of Theorem~\ref{theorem: linearity theorem} is given in Section~\ref{section: strategy}.
The details occupy Sections~\ref{section: two structure maps}--\ref{section: maps out of E}.

Given Theorem~\ref{theorem: linearity theorem}, we can bring in the third adjective, ``stable."
First we recall some concepts from the introduction.
The \defining{stabilization} of~$\Fcal$ is the spectrum
\begin{equation*}   
\stabilization{\Fcal}\definedas
\left\{
\Fcal(S^0), \Fcal\left(S^1\right), \ldots,
          \Fcal\left(S^n\right), \ldots
\right\}.
\end{equation*}
Given that $\Fcal$ is a $\Gamma$-space
with an augmentation-preserving action of the monad~$\Ccal$,
the \defining{unstable rank filtration} of~$\Fcal$ is the sequence of (not necessarily special) $\Gamma$-spaces
\begin{equation*}
\filter{0}\Fcal
     \rightarrow  \filter{1}\Fcal
     \rightarrow \cdots
     \rightarrow  \filter{m}\Fcal
     \rightarrow \cdots
     \rightarrow  \Fcal,
\end{equation*}
and the \defining{stable rank filtration} of~$\Fcal$ is the sequence of \emph{spectra}
\begin{equation*}
\stabilization{\!\left(\filter{0}\Fcal\right)}
     \rightarrow \stabilization{\!\left(\filter{1}\Fcal\right)}
     \rightarrow \cdots
     \rightarrow \stabilization{\!\left(\filter{m}\Fcal\right)}
     \rightarrow \cdots
     \rightarrow \stabilization{\Fcal}.
\end{equation*}
Further, recall that if $\Gcal$ is a special $\Gamma$-space, then it has an \defining{infinite delooping} $\Ktheory\Gcal$, a spectrum whose infinite loop space is equivalent to the group completion of~$\Gcal$.
The last part of this section proves the main theorem of our paper, which is the following.

\begin{theorem}\label{theorem: bar filtration}
\BarFiltrationTheoremText
\end{theorem}

The discussion leading up to the proof of  Theorem~\ref{theorem: bar filtration} begins with
Definition~\ref{definition: stable equivalence}, and the actual proof appears at the end of the section.

\bigskip


Now that we have described the results, we begin the first task, showing that for a fixed~$m$, the object $\FilteredBK{m}{\Fcal}$ is a simplicial subobject
of~$\Barr(\Ccal, \Ccal, \Fcal)=\Ccal\Ccal^\bullet \Fcal$. We must check that the face and degeneracy maps of $\Ccal\Ccal^\bullet \Fcal$ respect~$\FilteredBK{m}{\Fcal}$. We begin this verification with the
following elementary lemma.
Recall
from Definitions~\ref{definition: filtered space} and~\ref{definition: define filtration CX}
that $\Ccal$ and $\filter{m}$ are to be considered endofunctors of~$\filtered$, the category of filtered pointed spaces. An immediate consequence of
Definition~\ref{definition: filtered space} is that there is a natural transformation $\filter{m}\rightarrow\Id$ of endofunctors of~$\filtered$.

\begin{lemma}\label{lemma: switch}
The natural inclusion
$\filter{m}\circ\,\Ccal
\hookrightarrow
        \Id\circ\,\Ccal=\Ccal$,
induced by the natural transformation $\filter{m}\rightarrow\Id$ applied to~$\Ccal$,
factors uniquely as a composition
\[
\filter{m}\circ\,\Ccal\hookrightarrow \Ccal\circ\filter{m}\hookrightarrow \Ccal\circ\,\Id=\Ccal,
\]
where the second map is $\Ccal$ applied to $\filter{m}\rightarrow\Id$.
\end{lemma}

\begin{proof}
Let $X$ be an object of~$\filtered$.
The product filtration on $X^n$, given in~\eqref{eq: product filtration}, tells us that
there is a natural factoring
\[
\xymatrix{
\strut\filter{m}\left(X^n\right)\,
    \ar@{^{(}->}[r]
    \ar@{^{(}->}[d]
    &\LARGEstrut\left(\filter{m} X\right)^n
    \ar@{_{(}->}[dl]
        \\
X^n
.
}
\]
We take the coend $E\Sigma_{n}\otimes_{n\in\Ical}\mbox{---}$ at each corner to get
\begin{equation}   \label{eq: take coends}
\begin{gathered}
\xymatrix{
\strut E\Sigma_n \otimes_{n\in \Ical}\filter{m}\left(X^n\right)\,
    \ar[r]
    \ar[d]
&
E\Sigma_n \otimes_{n\in \Ical}\left(\filter{m} X\right)^n
    \ar[dl]
        \\
E\Sigma_n \otimes_{n\in \Ical}X^n
.
}
\end{gathered}
\end{equation}
Substituting for the spaces in diagram~\eqref{eq: take coends} using
the definition of~$\Ccal X$ and $\filter{m}\Ccal X$ given, respectively,
in \eqref{eq: coend formula}
and Definition~\ref{definition: define filtration CX}, we obtain
\begin{equation*}  
\begin{gathered}
\xymatrix{
\filter{m}\Ccal X\
    \ar[r]
    \ar[d]
    &\LARGEstrut\Ccal\filter{m} X
    \ar[dl]
        \\
\Ccal X,
}
\end{gathered}
\end{equation*}
which establishes the existence of a factoring of the vertical map
through~$\Ccal\filter{m}X$, as required.

To show that the factorization is unique,
it is sufficient to show that the diagonal arrow in
\eqref{eq: take coends} is a monomorphism of $\Gamma$-spaces.  Recall that the coends in the source and the target are quotient spaces as indicated below:
\[
\xymatrix{
\coprod_{n} E\Sigma_n \times_{\Sigma_n} \left(\filter{m}X\right)^n
    \ar@{^{(}->}[d]
   \ar@{->>}[r]
& E\Sigma_n \otimes_{n\in \Ical}\left(\filter{m}X\right)^n
   \ar[d]
\\
\coprod_{n} E\Sigma_n \times_{\Sigma_n} X^n
   \ar@{->>}[r]
&
E\Sigma_n \otimes_{n\in \Ical}X^n.
}
\]
The map $\filter{m}X\rightarrow X$ is an inclusion, so
the left vertical arrow is a monomorphism.
The equivalence relations defining horizontal arrows are the same basepoint identifications in the top row and the bottom row. Hence the right column is an inclusion also.
\end{proof}

Showing that the face maps in the bar construction
$\Ccal\Ccal^\bullet \Fcal$
respect the subobjects $\Ccal\filter{m}\Ccal^\bullet \Fcal$
mostly comes from naturality; the
face maps in the bar construction come from the multiplication $\mu_{\Ccal}\colon\Ccal\Ccal\to\Ccal$, and all but one of them take place ``inside" of the filtration functor~$\filter{m}$. However, the outermost one (i.e. $d_{0}$) requires ``multiplying across $\filter{m}$." The following lemma identifies the target of that map.

\begin{lemma}\label{lemma: filtration of composition, Greg's version}
The natural transformation $\mu_{\Ccal}\colon\Ccal\Ccal\to \Ccal$ restricts, for every $m$, to a unique natural transformation
\[
\Ccal\filter{m}\Ccal \to \Ccal\filter{m}.
\]
\end{lemma}

\begin{proof}
It follows from Lemma~\ref{lemma: switch} that the map
$\Ccal \filter{m}\Ccal \to \Ccal\Ccal$ factors as
$\Ccal \filter{m}\Ccal  \to \Ccal\Ccal\filter{m}\to \Ccal\Ccal$.
We can expand to a commutative diagram by applying
$\mu_{\Ccal}\colon \Ccal\Ccal\to \Ccal$ to the natural transformation $\filter{m}\hookrightarrow\Id$, as follows:
\begin{equation*}
\begin{gathered}
\xymatrix{
\Ccal \filter{m}\Ccal\
\ar@{^{(}->}[r]
& \Ccal\Ccal\filter{m}\
\ar@{^{(}->}[r]
\ar[d]_{\mu_{\Ccal}(\filter{m})}
& \Ccal\Ccal
\ar[d]^{\mu_{\Ccal}(\Id)}
   \\
&
\Ccal\filter{m}\
\ar@{^{(}->}[r]
&
\Ccal.
}
\end{gathered}
\end{equation*}
The lemma follows by composing
$\Ccal \filter{m}\Ccal  \hookrightarrow \Ccal\Ccal\filter{m}$ with
just the left vertical arrow of the square.
\end{proof}

\medskip

With these tools, we can show that for each fixed $m\in\integers_{+}$, the set
of $\Gamma$-spaces $\left\{\Ccal\filter{m}\Ccal^{k}\Fcal\right\}_{k}$  assembles into a simplicial subobject of $\Barr\left(\Ccal,\Ccal,\Fcal\right)$,
proving Proposition~\ref{proposition: structure maps respect filtration}.

\begin{proof}
[Proof of Proposition~\ref{proposition: structure maps respect filtration}]
As before, let $\mu_{\Ccal}\colon \Ccal^2\to \Ccal$ and $\eta_{\Ccal}\colon\Id\to \Ccal$ denote the multiplication and the unit of the monad~$\Ccal$.
Recall that the face and degeneracy maps of
$\Barr\left(\Ccal,\Ccal,\Fcal\right)=\Ccal\Ccal^{\bullet}\Fcal$
in simplicial dimension $k$
are given by the formulas
\begin{alignat}{3}
d_{i} &= \Ccal^{i}\,\mu_{\Ccal}\left(\Ccal^{k-i-1}\Fcal\right)
       &\colon \Ccal\Ccal^k \Fcal \to  \Ccal\Ccal^{k-1} \Fcal
       & \mbox{\quad  for } i=0, \ldots, k-1,
       \nonumber\\
d_{k} &= \qquad\Ccal^{k}\,\mu_{\Fcal}
       &\colon \Ccal\Ccal^k \Fcal \to  \Ccal\Ccal^{k-1} \Fcal,
       \label{eq: face map i}
       \\
s_{i} &= \Ccal\Ccal^{i}\,\eta_{\Ccal}(\Ccal^{k-i}\Fcal)
       &\colon \Ccal\Ccal^k \Fcal \to  \Ccal\Ccal^{k+1} \Fcal
       & \mbox{\quad  for } i=0, \ldots, k.
       \nonumber
\end{alignat}

First we tackle the face maps. To show that $d_{i}$ respects the filtration, we must establish that the restriction of
$d_{i}$ to $\Ccal\filter{m}\Ccal^{k}\Fcal$ factors through
$\Ccal\filter{m}\Ccal^{k-1}\Fcal$. That is, we need a map of
$\Gamma$-spaces
$\dwiggle_{i}\colon \Ccal\filter{m}\Ccal^{k}\Fcal\to \Ccal\filter{m}\Ccal^{k-1}\Fcal$ that makes the following diagram commute:
\begin{equation}   \label{eq: filtered faces}
\begin{gathered}
\xymatrix{
\Largestrut\Ccal\filter{m}
                  \Ccal^{k}\Fcal
      \ar@{-->}[r]^-{\dwiggle_{i}}
      \ar@{^{(}->}[d]
      &
\Largestrut\Ccal\filter{m}
                             \Ccal^{k-1}\Fcal
      \ar@{^{(}->}[d]
\\
\Ccal\Ccal^{k}\Fcal\ar[r]_-{d_{i}}
     &
     \Ccal\Ccal^{k-1}\Fcal.
}
\end{gathered}
\end{equation}
Note that if $\dwiggle_{i}$ exists, then it is necessarily unique, because the vertical maps are monomorphisms of $\Gamma$-spaces. Likewise, it follows from the injectivity of the vertical maps that once the maps $\dwiggle_{i}$ are proven to exist, as well as the degeneracy maps $\swiggle_{i}$ defined below, they must satisfy the simplicial identities.

For $i\geq 1$ in~\eqref{eq: filtered faces}, the face map
$d_i\colon \Ccal\Ccal^{k}\Fcal \to \Ccal\Ccal^{k-1}\Fcal$ of the bar construction $\Ccal\Ccal^{\bullet}\Fcal$ is obtained, according to~\eqref{eq: face map i}, by applying $\Ccal$ to the map
\begin{equation}\label{eq: business part of d_i}
\Ccal^{i-1}\,\Ccal^2 \left(\Ccal^{k-i-1}\Fcal\right)
\xrightarrow{\ \Ccal^{i-1}\,\mu_{\Ccal}\left(\Ccal^{k-i-1}\Fcal\right)\ }
\Ccal^{i-1}\,\Ccal\left(\Ccal^{k-i-1}\Fcal\right).
\end{equation}
Hence for $i\geq 1$, we can define the map
$\dwiggle_{i}\colon \Ccal\filter{m}\Ccal^{k} \to \Ccal\filter{m}\Ccal^{k-1}$ by applying $\Ccal\filter{m}$ to the map~\eqref{eq: business part of d_i}.
For $i=0$, however, the face map $d_{0}$ in the bar construction is not $\Ccal$ applied to another map, but rather multiplies together the outermost factors of~$\Ccal$ in~$\Ccal\Ccal^{k}\Fcal$. In this case the map $\dwiggle_0$ is given by taking the map $\Ccal\filter{m}\Ccal\to \Ccal\filter{m}$ that is given to us by Lemma~\ref{lemma: filtration of composition, Greg's version} and applying it to~$\Ccal^{k-1}\Fcal$.

Similarly, for the degeneracies we require that for each $i$ there exists a (necessarily unique) map
$\swiggle_{i}\colon
\Ccal\filter{m}\Ccal^{k}\Fcal\to \Ccal\filter{m}\Ccal^{k+1}\Fcal$
that makes the following diagram commute:
\begin{equation*}
\begin{gathered}
\xymatrix{
\Largestrut\Ccal\filter{m}\Ccal^{k}\Fcal
      \ar@{-->}[r]^-{\swiggle_{i}}
      \ar@{^{(}->}[d]
      &
\Largestrut\Ccal\filter{m}\Ccal^{k+1}\Fcal
      \ar@{^{(}->}[d]
\\
\Ccal\Ccal^{k}\Fcal\ar[r]_-{s_{i}}
     &
     \Ccal\Ccal^{k+1}\Fcal.
}
\end{gathered}
\end{equation*}
In this case, for all $i\geq 0$ the desired map is obtained by applying $\Ccal\filter{m}$ to the map
\[
\Ccal^{i}\,\eta_{\Ccal}(\Ccal^{k-i}\Fcal)
\colon
\Ccal^{k}\Fcal\to \Ccal^{k+1}\Fcal.
\]
\end{proof}

With Proposition~\ref{proposition: structure maps respect filtration} in hand,
we obtain an increasing sequence of simplicial subobjects  of $\Barr(\Ccal, \Ccal, \Fcal)$,
as follows.

\begin{diag}     \label{diagram: filtered simplicial}
\[
\xymatrix@C=2em{
\stackrel{\vdots}{\Ccal\filter{m-1} \Fcal}\
   \ar[r]
   \ar@{}[d]|*=0[@]{\subseteq}
&\ \stackrel{\vdots}{\Ccal\filter{m-1}(\Ccal\Fcal)}\
   \ar@{}[d]|*=0[@]{\subseteq}
         \ar@<-.6ex>[l]
         \ar@<.6ex>[l]
         \ar@<-.5ex>[r]
         \ar@<.5ex>[r]
&\ \stackrel{\quad\vdots}{\Ccal\filter{m-1}\left(\Ccal^2\Fcal\right)}\
   \ar@{}[d]|*=0[@]{\subseteq}
         \ar@<-1.0ex>[l]
         \ar[l]
         \ar@<1.0ex>[l]
\ldots
&\ \stackrel{\quad\vdots}{\Ccal\filter{m-1}\left(\Ccal^k\Fcal\right)}\
   \ar@{}[d]|*=0[@]{\subseteq}
\ldots
\\
\Ccal\filter{m} \Fcal\quad
   \ar[r]
      \ar@{}[d]|*=0[@]{\subseteq}
&\quad \Ccal\filter{m}(\Ccal\Fcal)\quad
      \ar@{}[d]|*=0[@]{\subseteq}
         \ar@<-.6ex>[l]
         \ar@<.6ex>[l]
         \ar@<-.5ex>[r]
         \ar@<.5ex>[r]
&\quad \Ccal\filter{m}\left(\Ccal^2\Fcal\right)\
      \ar@{}[d]|*=0[@]{\subseteq}
         \ar@<-1.0ex>[l]
         \ar[l]
         \ar@<1.0ex>[l]
\ldots
&\ \Ccal\filter{m}\left(\Ccal^k\Fcal\right)\
      \ar@{}[d]|*=0[@]{\subseteq}
\ldots
\\
\quad\stackrel{\vdots}{\Ccal\Fcal}\qquad
   \ar[r]
&\qquad \stackrel{\vdots}{\Ccal(\Ccal\Fcal)}\qquad
         \ar@<-.6ex>[l]
         \ar@<.6ex>[l]
         \ar@<-.5ex>[r]
         \ar@<.5ex>[r]
&\quad \stackrel{\quad\vdots}{\Ccal\left(\Ccal^2\Fcal\right)}\quad
         \ar@<-1.0ex>[l]
         \ar[l]
         \ar@<1.0ex>[l]
\ldots
&\ \stackrel{\quad\vdots}{\Ccal\left(\Ccal^k\Fcal\right)}\
\ldots
}
\]
\end{diag}
As stated in
Theorem~\ref{theorem: linearity theorem}, toward the beginning of the section, the geometric realization of each row
of Diagram~\ref{diagram: filtered simplicial} actually gives a
\emph{special} simplicial $\Gamma$-space, i.e.
takes wedges to products up to homotopy, and we assume this for the remainder of the section.
The importance is that for special $\Gamma$-spaces, infinite delooping and stabilization are equivalent processes \cite[Theorem~4.2]{Bousfield-Friedlander}.
The proof  of Theorem~\ref{theorem: linearity theorem}
is deferred to Sections~\ref{section: strategy}--\ref{section: maps out of E}.
In the remainder of this section,
we prove the principal result of our paper, Theorem~\ref{theorem: bar filtration},
assuming Theorem~\ref{theorem: linearity theorem}.

Recall that the goal is to prove that the spectrum given by the infinite delooping of the $m$-th row of Diagram~\ref{diagram: filtered simplicial}, the $\Gamma$-space
$\realization{\Ccal\filter{m}\Ccal^{\bullet}\Fcal}$,
is equivalent to
the stabilization of~$\filter{m}\Fcal$,
\begin{equation*}   
\stabilization{(\filter{m}\Fcal)}\definedas
\left\{
\filter{m}\Fcal(S^0), \filter{m}\Fcal\left(S^1\right), \ldots,
          \filter{m}\Fcal\left(S^n\right), \ldots
\right\},
\end{equation*}
which is the $m$-th spectrum in the stable rank filtration of~$\Fcal$.

Our approach is to relate $\filter{m}\Fcal$ and
$\realization{\Ccal\filter{m}\Ccal^{\bullet}\Fcal}$ through the $\Gamma$-space,
$\realization{\filter{m}{(\Ccal\Ccal^\bullet \Fcal)}}$, which was studied in Section~\ref{section: bar}, culminating in Proposition~\ref{proposition: tautology}. Lemma~\ref{lemma: switch} provides us with a map
\begin{equation}   \label{eq: marker}
\realization{\filter{m}{(\Ccal\Ccal^\bullet \Fcal)}}
\longrightarrow
\realization{\Ccal\filter{m}{(\Ccal^\bullet \Fcal)}},
\end{equation}
and a key point in the proof of Theorem~\ref{theorem: bar filtration} is that this map induces an equivalence of spectra between the stablizations.
To study this point, we recall the term ``stable equivalence," as a useful criterion for a natural transformation of $\Gamma$-spaces to induce an equivalence of stabilizations.

\begin{definition}  \label{definition: stable equivalence}
A natural transformation between $\Gamma$-spaces $\alpha\colon \Fcal \to \Gcal$ is a \defining{stable equivalence} if there exists a constant $c$ such that, whenever $X$ is an $(n-1)$-connected pointed space, the map
$\alpha(X)\colon \Fcal(X) \to \Gcal(X)$ is $(2n-c)$-connected.
\end{definition}

\begin{lemma}
If $\alpha\colon \Fcal \to \Gcal$ is a stable equivalence of $\Gamma$-spaces, then $\alpha$ induces an equivalence of spectra on the stabilizations
$\stabilization{\alpha}\colon\stabilization{\Fcal}
\xrightarrow{\, \simeq\, }
\stabilization{\Gcal}$.
\end{lemma}

\begin{proof}
Recall from \eqref{eq: stabilize a Gamma space} that the $n$-th spaces in
$\stabilization{\Fcal}$ and $\stabilization{\Gcal}$  are
$\Fcal(S^n)$ and~$\Gcal(S^n)$, respectively. If $c$ is the number given in
Definition~\ref{definition: stable equivalence}, then
$\Fcal(S^n)\rightarrow \Gcal(S^n)$ is $(2n-c)$-connected. It follows that the map of spectra $\stabilization{\alpha}\colon\stabilization{\Fcal}
\to\stabilization{\Gcal}$ induces an isomorphism of homotopy groups, and the lemma follows.
\end{proof}

\begin{example}    \label{example: stable equivalence}
The unit $\eta_{\Ccal}\colon \Id\rightarrow\Ccal$ is a stable equivalence. This follows easily from analyzing the filtration of~$\Ccal$.
\end{example}

We will now recall a few elementary properties of stable equivalences. The following lemma says, informally speaking, that stable equivalences form a kind of ``ideal'' in the category of $\Gamma$-spaces.
\begin{lemma}
\label{lemma: compose stable equivalence}
Suppose $\alpha\colon \Gcal_1\to \Gcal_2$ is a stable equivalence, and $\Fcal$ is any $\Gamma$-space. Then the induced natural transformations $\Fcal\alpha\colon \Fcal\Gcal_1\to \Fcal\Gcal_2$ and $\alpha\Fcal\colon \Gcal_1\Fcal\to \Gcal_2\Fcal$ are stable equivalences.
\end{lemma}

\begin{proof}
The result follows immediately from~\cite[Proposition 5.20]{Lydakis-Gamma}, which says that any $\Gamma$-space preserves connectivity of maps.
\end{proof}

\begin{example}     \label{example: unit}
Lemma~\ref{lemma: compose stable equivalence} tells us that
if $\Fcal$ is a $\Gamma$-space, then the
natural transformation $\eta_{\Ccal}(\Fcal)\colon\Fcal\to\Ccal\Fcal$
induced by the map $\eta_{\Ccal}$ is a stable equivalence.
\end{example}

Recall from~\eqref{eq: marker} that we would like to show that
\[
\realization{\filter{m}{(\Ccal\Ccal^\bullet \Fcal)}}
\longrightarrow
\realization{\Ccal\filter{m}{(\Ccal^\bullet \Fcal)}}
\]
is a stable equivalence.
For this purpose,
we actually need a more refined version of
Example~\ref{example: unit}.
Rather than only knowing that
$\eta_{\Ccal}\colon\Fcal\to\Ccal\Fcal$ is a stable equivalence,
we need a filtered version of this statement to handle the filtered functors involved.

\begin{lemma}\label{lemma: filtered Freudenthal}
Suppose $\Fcal$ is a filtered $\Gamma$-space. Then the map $\eta_{\Ccal}\colon\Fcal\to \Ccal\Fcal$ induces a stable equivalence
$\filter{m}{\Fcal}\to \filter{m}{\Ccal\Fcal}$ for every $m$.
\end{lemma}

\begin{proof}
Recall that by Definition~\ref{definition: define filtration CX},
\[
\filter{m}{\Ccal\Fcal}=E\Sigma_n \otimes_{n\in \Ical}\filter{m}\left(\Fcal^{\times n}\right).
\]
Let $\Ical_{1}$ be the subcategory of $\Ical$ consisting of objects of cardinality $1$ or~$0$.
We organize the relevant spaces into the following commutative diagram:
\[
\xymatrix{
\filter{m}{\Fcal}
   \ar[r]^{\filter{m}\left(\eta_{\Ccal}(\Fcal)\right)}
   \ar[d]_-{\simeq}
& \filter{m}{\Ccal\Fcal}
   \ar[d]_-{=}^-{\rm{defn}}\\
E\Sigma_n \otimes_{n\in\Ical_{1}}
   \filter{m}\left(\Fcal^{\times n}\right)
   \ar[r]
   \ar[d]_-{\cong}
&
E\Sigma_n\otimes_{n\in \Ical}
   \filter{m}\left(\Fcal^{\times n}\right)
   \ar@{=}[d]\\
E\Sigma_n\otimes_{n\in \Ical}
   \bigvee_n\filter{m}{\Fcal}
   \ar[r]
&
E\Sigma_n\otimes_{n\in \Ical}
   \filter{m}\left(\Fcal^{\times n}\right).
}
\]
The homeomorphism in the lower left column uses that the left Kan extension from $\Ical_{1}$ to $\Ical$ of the functor $n\mapsto \filter{m}(\Fcal^{\times n})$ is the functor $n\mapsto \bigvee_n \filter{m}{\Fcal}$.

The map
$\bigvee_{n}\filter{m}{\Fcal}
\to
\filter{m}
\left(\Fcal^{\times n}
\right)$
is a stable equivalence for each $m$ and $n$
by~\cite[Lemma 3.13]{Arone-Lesh-Fundamenta}. It follows that the bottom map in the diagram is a stable equivalence, which proves the lemma.
\end{proof}

Since \eqref{eq: marker} is induced by Lemma~\ref{lemma: switch}, we
apply Lemma~\ref{lemma: filtered Freudenthal} to show how Lemma~\ref{lemma: switch}
interacts with filtration.

\begin{lemma}\label{lemma: stable equivalence}
Let $\Gcal$ be a filtered $\Gamma$-space. The natural map $\filter{m}\Ccal\Gcal \to \Ccal\filter{m}\Gcal$ provided by Lemma~\ref{lemma: switch} is a stable equivalence.
\end{lemma}

\begin{proof}
Consider the maps
\[
\filter{m}\Gcal\to \filter{m}\Ccal\Gcal\to \Ccal\filter{m}\Gcal.
\]
The first map is a stable equivalence by Lemma~\ref{lemma: filtered Freudenthal}, and the composed map is a stable equivalence because the map from the identity functor to $\Ccal$ is a stable equivalence (Example~\ref{example: stable equivalence} and Lemma~\ref{lemma: filtered Freudenthal}). It follows that the second map is a stable equivalence, which is what we wanted to prove.
\end{proof}

Now, finally, comes the payoff. We have assembled the tools needed to prove the main theorem.

\begin{proof}[Proof of Theorem~\ref{theorem: bar filtration}]
We are given an augmented special $\Gamma$-space~$\Fcal$. By Lemma~\ref{lemma: augmentation}, the action of the monad~$\Ccal$ on~$\Fcal$ is augmentation-preserving. We want to prove that there is an
equivalence of spectra
\[
\stabilization{\!\left(\filter{m}\Fcal\right)}
\simeq
\Ktheory{\realization{\FilteredBK{m}{\Fcal}}}.
\]
Because $\realization{\Ccal\filter{m}{(\Ccal^\bullet \Fcal)}}$ is a \emph{special}
$\Gamma$-space (Theorem~\ref{theorem: linearity theorem}), we actually know that
$\Ktheory{\realization{\Ccal\filter{m}{(\Ccal^\bullet \Fcal)}}}
\simeq
\stabilization{\realization{\Ccal\filter{m}{(\Ccal^\bullet \Fcal)}}}$
\cite[Theorem~4.2]{Bousfield-Friedlander}, so we want to prove that
\[
\stabilization{\!\left(\filter{m}\Fcal\right)}
\simeq
\stabilization{\realization{\Ccal\filter{m}{(\Ccal^\bullet \Fcal)}}}.
\]

There is a natural map
\begin{equation}    \label{eq: what we want}
\filter{m}\Fcal
    \to \realization{\Ccal\filter{m}\Ccal^{\bullet}\Fcal},
\end{equation}
induced by the unit map $\filter{m}\Fcal\rightarrow\Ccal\filter{m}\Fcal$ followed by inclusion
into the $0$-th simplicial level. By Lemma~\ref{lemma: switch}, there is a factoring of the unit map as
\[
\filter{m}\Fcal
     \rightarrow\filter{m}\Ccal\Fcal
     \rightarrow\Ccal\filter{m}\Fcal
\]
and so \eqref{eq: what we want} factors as a composite
\begin{equation}     \label{eq: factoring}
\filter{m}{\Fcal}
\xrightarrow{\ \simeq\ }
\realization{\filter{m}{(\Ccal\Ccal^\bullet \Fcal)}}
\longrightarrow
\realization{\Ccal\filter{m}{(\Ccal^\bullet \Fcal)}},
\end{equation}
where the first map of the composite is an equivalence by Proposition~\ref{proposition: tautology}. So we only need to prove that the second map induces an equivalence of stabilizations.
But the second map is a stable equivalence in each simplicial degree by Lemma~\ref{lemma: stable equivalence}, and therefore induces a stable equivalence of geometric realizations.
Therefore
$\filter{m}{\Fcal}
\longrightarrow
\realization{\Ccal\filter{m}{(\Ccal^\bullet \Fcal)}}$
is a stable equivalence, yielding an equivalence of spectra
$\stabilization{\filter{m}{\Fcal}}
\xrightarrow{\simeq}
\stabilization{\realization{\Ccal\filter{m}
   (\Ccal^\bullet \Fcal)}}$
as desired.

All these equivalences are compatible with the inclusion maps from $\filter{m}{(-)}$ to $\filter{m+1}{(-)}$. This proves the theorem.
\end{proof}


\section{Strategy and proof of linearity}
\label{section: strategy}

We now begin the proof of Theorem~\ref{theorem: linearity theorem}, which occupies the rest of the paper. Although the proof is relatively simple in structure, it
is surprisingly involved in its details. In this section we define all of the objects in the proof, state several intermediate results, and prove
Theorem~\ref{theorem: linearity theorem} assuming these claims.
The subsequent sections are devoted to the details of proving the intermediate results.
For the remainder of the paper, we assume that $\Fcal$ is an augmented special $\Gamma$-space. The monad $\Ccal$ associated to an $E_\infty$-operad
has a left action on $\Fcal$ denoted~$\mu_{\Fcal}:\Ccal\Fcal\rightarrow\Fcal$, whichx is augmentation-preserving (see Lemma~\ref{lemma: augmentation}).

Let $X$ and $Y$ be finite pointed sets, and let $\collapsemap{X}$ denote the collapse map $X\vee Y\rightarrow X$ that identifies all of $Y$ to the basepoint. To establish
that $\realization{\FilteredBK{m}{\Fcal}}$ is a special $\Gamma$-space, we must show that
for all finite pointed finite sets $X$ and~$Y$, the natural map
\begin{equation}     \label{eq: desired equivalance of SS}
f: \underbrace{\Hugestrut
              \FilteredBK{m}{\Fcal}(X\vee Y)
              }_{\text{Denoted $A_{\bullet}$}
              }
\xrightarrow{\qquad}
   \underbrace{\Hugestrut
              \FilteredBK{m}{\Fcal}X\times \FilteredBK{m}{\Fcal}Y
              }_{\text{Denoted $B_{\bullet}$}
              },
\end{equation}
given by
$\FilteredBK{m}{\Fcal}(\collapsemap{X})
    \times\FilteredBK{m}{\Fcal}(\collapsemap{Y})$,
induces a weak equivalence of realizations.

We point out that
\eqref{eq: desired equivalance of SS} is not actually a levelwise weak equivalence. That is, if we fix a value of~$k$, the map
\begin{equation} \label{eq: not special}
\Ccal\filter{m}\Ccal^k\Fcal(X\vee Y)
\to
\Ccal\filter{m}\Ccal^k\Fcal X \times \Ccal\filter{m}\Ccal^k\Fcal Y
\end{equation}
is usually not a weak equivalence.
For example, in simplicial dimension~$k=0$ when $m=2$, we would be trying to show that
\begin{equation}     \label{eq: nonexample}
\Ccal\filter{2}\Fcal(X\vee Y)
\longrightarrow
\Ccal\filter{2}\Fcal X
\times
\Ccal\filter{2}\Fcal Y
\end{equation}
is a homotopy equivalence. However, an explicit counterexample can be given by using $\Fcal=\Sp^{\infty}$ and $X=Y=S^{0}$. The left side turns out to be
$\left(\Ccal S^{0}\right)^{\times 5}$, while the right side turns out to be $\left(\Ccal S^{0}\right)^{\times 4}$.

Nevertheless, Theorem~\ref{theorem: linearity theorem} asserts that \eqref{eq: desired equivalance of SS} induces a weak equivalence of geometric realizations. The proof strategy is to ``fatten up" the spaces involved.
Let $A_{\bullet}$ and $B_{\bullet}$ be the domain and codomain, respectively, of~\eqref{eq: desired equivalance of SS}.
We embed both $A_{\bullet}$ and $B_{\bullet}$ in much larger simplicial spaces $\Ecal_{\bullet}$ and $\Dcal_{\bullet}$, respectively, and exhibit a map $\Ecal_{\bullet}\rightarrow\Dcal_{\bullet}$ that actually is a levelwise weak equivalence. It then turns out (see diagram~\eqref{diagram: square of auxiliaries}) that $A_{\bullet}\rightarrow B_{\bullet}$ is a retract of
$\Ecal_{\bullet}\rightarrow\Dcal_{\bullet}$ up to a simplicial homotopy, 
and $B_{\bullet}$ is a deformation retract of~$\Dcal_{\bullet}$,
so in the end all four simplicial spaces have weakly equivalent realizations.

To implement the strategy, we first expand $B_{\bullet}$
to a larger simplicial space $\Dcal_{\bullet}$ by gathering  the different ways that
$\Ccal\filter{m}\left(\Ccal^{k}\Fcal\right)$ could turn the wedge
in $A_{\bullet}$ into a product.
For $0\leq i\leq k$, define
\begin{align*}
D_{k}^{i}
   &=\Ccal
   \filter{m}\Ccal^{k-i}
   \left(\Largestrut\Ccal^{i}\Fcal X\times\Ccal^{i}\Fcal Y\right).
\\
\intertext{At one extreme, when $i=0$, we have transformed minimally,
using only $\Fcal$ to make a wedge-to-product transformation
on the very inside of the composite:}
D_{k}^{0}
   &=\Ccal
   \filter{m}\Ccal^{k}
         \left(\Largestrut\Fcal X\times\Fcal Y\right).
   \\
\intertext{At the other extreme, we define one additional space for $i=k+1$, when we have $B_{\bullet}$ itself,}
D_{k}^{k+1}
=
B_{k} &=
\Ccal
     \filter{m}\Ccal^{k}\Fcal X
\times
\Ccal
     \filter{m}\Ccal^{k}\Fcal Y.
\end{align*}
We define $\Dcal_{\bullet}$ by
\begin{equation}  \label{eq: define D}
\Dcal_{k}=\bigvee_{0\leq i\leq k+1}D_{k}^{i}.
\end{equation}

\begin{proposition}    \label{proposition: D is a simplicial space}
$\Dcal_{\bullet}$ is a simplicial space.
\end{proposition}

Proposition~\ref{proposition: D is a simplicial space} is proved as
Proposition~\ref{proposition: proving D is a simplicial space}
in Section~\ref{section: D is a simplicial space}, where we define the face and degeneracy maps for $\Dcal_{\bullet}$ and verify the simplicial identities.
Almost all of the simplicial structure maps are defined in a straightforward way using the monad structure of $\Ccal$ and the module structure of $\Fcal$ over~$\Ccal$. But in each simplicial dimension, there is one face map that needs more detailed attention to define.

As we said above, $D_{k}^{k+1}=B_{k}$. When we define the simplicial structure maps in Section~\ref{section: D is a simplicial space}, we will see that they are respected by the inclusion $B_{\bullet}\subset\Dcal_{\bullet}$. However, much more is true: $B_{\bullet}$~is not just a simplicial subspace, but in fact a simplicial deformation retract of~$\Dcal_{\bullet}$, as established by the following proposition.

\begin{proposition}    \label{proposition: existence of homotopy}
There is an inclusion of simplicial spaces $j: B_{\bullet}\hookrightarrow\Dcal_{\bullet}$
and a simplicial retraction $q: \Dcal_{\bullet}\rightarrow B_{\bullet}$. Further, there is a simplicial homotopy $h$ between the composite
\[
\Dcal_{\bullet}\xrightarrow{\ q\ } B_{\bullet} \xrightarrow{\ j\ }\Dcal_{\bullet}
\]
and the identity map on~$\Dcal_{\bullet}$, which restricts to the identity on the image of $B_{\bullet}$ in~$\Dcal_{\bullet}$.
\end{proposition}

The proof of Proposition~\ref{proposition: existence of homotopy} is given
in Section~\ref{section: homotopy of Ds} as
Proposition~\ref{proposition: simplicial homotopy}.

The second collection of auxiliary spaces is defined (as was the case with~$\Dcal_{\bullet}$) for $0\leq i\leq k+1$. It gathers the different ways that $\Ccal\filter{m}\Ccal^{k}\Fcal$
can distribute over a wedge. For $1\leq i\leq k+1$, we define
\begin{align}
E_{k}^{i}    \label{eq: definition of Eki}
   &=\Ccal\filter{m}\Ccal^{k-i+1}
         \left(\Largestrut\Ccal^{i-1}\Fcal X\vee\Ccal^{i-1}\Fcal Y\right),
\\
\intertext{and we have one additional space }
E_{k}^{0}&=A_{k}=        \label{eq: ident of Ek0 as Ak}
          \Ccal\filter{m}
          \Ccal^{k}\Fcal\left(\largestrut X\vee Y\right)
   \\
\intertext{while at the other end}
E_{k}^{k+1}
   &=\Ccal\filter{m}
         \left(
         \Largestrut\Ccal^{k}\Fcal X\vee\Ccal^{k}\Fcal Y
         \right).
\end{align}
Just as in the case of~$\Dcal_{\bullet}$, we define
\begin{equation}  \label{eq: define E}
\Ecal_{k}=\bigvee_{0\leq i\leq k+1}E_{k}^{i},
\end{equation}
we define face and degeneracy maps for~$\Ecal_{\bullet}$, and we have $A_{\bullet}$ sitting inside the large simplicial space as a retract.

\begin{proposition}    \label{proposition: E is a simplicial space}
$\Ecal_{\bullet}$ is a simplicial space, and there are
simplicial inclusion and retraction maps
$i\colon A_{\bullet}\hookrightarrow\Ecal_{\bullet}$
and $r\colon\Ecal_{\bullet}\rightarrow A_{\bullet}$.
\end{proposition}

The first part of Proposition~\ref{proposition: E is a simplicial space} is proved in Section~\ref{section: E is a simplicial space} as Proposition~\ref{proposition: prove Ecal is a simplicial space}. The second part is proved in
Section~\ref{section: maps out of E} as
Lemma~\ref{lemma: E to A retraction}.
Indeed, one could prove directly that $i\colon A_{\bullet}\hookrightarrow\Ecal_{\bullet}$ is actually a \emph{deformation} retraction, but we do not need this statement and it would require writing down (and verifying) yet another simplicial homotopy.

Lastly, having related $A_{\bullet}$ to $\Ecal_{\bullet}$ and $B_{\bullet}$ to~$\Dcal_{\bullet}$, we need to relate $\Ecal_{\bullet}$ to $\Dcal_{\bullet}$.
There are weak equivalences $E_{k}^{i}\rightarrow D_{k}^{i}$ for each $i$ and~$k$, taking the innermost wedge of $E_{k}^{i}$ to a product using the special $\Gamma$-space~$\Ccal$ (or
$\Fcal(X\vee Y)\xrightarrow{\simeq}\Fcal X\times\Fcal Y$, when $i=0$):
\begin{equation}    \label{diagram: map of E to D}
\begin{gathered}
\xymatrix{
E_{k}^{i}
   =\Ccal\filter{m}\Ccal^{k-i}\Ccal
         \left(\Largestrut\Ccal^{i-1}\Fcal X\vee\Ccal^{i-1}\Fcal Y\right)
   \ar[d]^-{\simeq}
\\
D_{k}^{i}
   =\Ccal\filter{m}\Ccal^{k-i}
         \left(\Largestrut\Ccal^{i}\Fcal X\times\Ccal^{i}\Fcal Y
         \right).
}
\end{gathered}
\end{equation}
(Note that $E_{k}^{i}$ has a slightly different appearance from \eqref{eq: definition of Eki} because one factor of $\Ccal$ has been split off.) 

\begin{proposition}    \label{proposition: map E to D}
The map $p\colon \Ecal_{\bullet}\rightarrow\Dcal_{\bullet}$
defined by~\eqref{diagram: map of E to D} is a simplicial map and a levelwise weak equivalence.
\end{proposition}

Proposition~\ref{proposition: map E to D} is proved in
Section~\ref{section: E is a simplicial space}
as Proposition~\ref{proposition: prove map E to D}.

The simplicial spaces and maps between them are summarized in the following diagram:
\begin{equation}   \label{diagram: square of auxiliaries}
\thatdiagram
\end{equation}

\begin{lemma}          \label{lemma: commuting squares}
\CommutingSquareLemmaText
\end{lemma}

The proof of Lemma~\ref{lemma: commuting squares} is provided at the end of
Section~\ref{section: maps out of E}.

With these statements in hand, and most of the detail swept under the rug, we are able to prove Theorem~\ref{theorem: linearity theorem}.

\begin{proof}[Proof of Theorem~\ref{theorem: linearity theorem}]
We must show that
\[f:
\underbrace{\Hugestrut\FilteredBK{m}{\Fcal}(X\vee Y)}_{A_{\bullet}}
\xrightarrow{\qquad}
\underbrace{\Hugestrut\FilteredBK{m}{\Fcal} X\times \FilteredBK{m}{\Fcal}Y}_{B_{\bullet}}
\]
is a weak equivalence on realizations for all finite pointed sets $X$ and~$Y$. However,
$f=q\circ j\circ f
 \simeq q\circ p\circ i$
 by Lemma~\ref{lemma: commuting squares}. Hence, in the homotopy category
the map $f$ is a retract of~$p$. Since $p$ is a weak equivalence on realizations by Proposition~\ref{proposition: map E to D},
the map $f$ is as well.
\end{proof}

\begin{remark*}
It follows from Proposition~\ref{proposition: existence of homotopy} that the inclusion $j: B_{\bullet}\hookrightarrow\Dcal_{\bullet}$ is a weak homotopy equivalence. Since
$p: \Ecal_{\bullet}\rightarrow\Dcal_{\bullet}$ is a weak equivalence,
and $f:A_{\bullet}\rightarrow B_{\bullet}$ is also a weak equivalence, we actually
find out after the fact that all of the maps in
\eqref{diagram: square of auxiliaries} are weak homotopy equivalences (on realizations).
\end{remark*}

\section{The two special maps}      \label{section: two structure maps}

To flesh out the details of the argument of Section~\ref{section: strategy} and define
the structure maps of $\Dcal_{\bullet}$ and $\Ecal_{\bullet}$, we need two maps that relate the monad structure of $\Ccal$ (and its action on $\Ccal$-modules) to wedges and products.
Suppose that $\Gcal$ is a $\Gamma$-space taking values in~$\filtered$, with a filtration-preserving action of~$\Ccal$. (In later sections, $\Gcal$
will be $\Ccal^{i}\Fcal$ for some nonnegative integer~$i$ and an augmented $\Ccal$-module $\Fcal$, as in Theorem~\ref{theorem: linearity theorem}.)
We define the following composites, natural in sets $A$ and~$B$, where $\pi$ indicates a projection, $\mu$ denotes multiplication for a monad or algebra structure, and $\iota$ indicates an inclusion:
\begin{align} \label{eq: define beta}
\beta\colon
&        \Ccal\left(\strut\Gcal A \times\Gcal B\right)
\xrightarrow{\ \Ccal\pi_{1}\times\Ccal\pi_{2}\ }
        \Ccal\Gcal A \times \Ccal\Gcal B
\xrightarrow{\mu_{\Gcal}\times\mu_{\Gcal}}
       \Gcal A \times \Gcal B
\\
\label{eq: define gamma}
\gamma\colon
&        \Ccal\left(\strut\Gcal A \vee \Gcal B\right)
\xrightarrow{\Ccal(\Gcal\iota_{A}\vee\Gcal\iota_{B})}
        \Ccal\Gcal\left(\strut A\vee B\right)
\xrightarrow{\quad \mu_{\Gcal}\quad }
        \Gcal(A\vee B).
\end{align}
When necessary, we will use a subscript to clarify the functor~$\Gcal$, writing $\beta_{\Gcal}$ or~$\gamma_{\Gcal}$.

\begin{remark}         \label{remark: beta and gamma preserve filtration}
The maps $\beta$ and $\gamma$ are both maps of filtered objects. In each case, the second map in the composition respects filtration because the action of $\Ccal$ is assumed to respect filtration.
Further, Lemma~\ref{lemma: boring lemma} establishes that if $\Gcal$ is filtered, then $\Ccal\pi_{1}\times\Ccal\pi_{2}$ respects the filtration,
so $\beta$ is a map of filtered objects. For $\gamma$, we need only note that
$\Gcal A\xrightarrow{} \Gcal(A\vee B)$ is filtration preserving by naturality, since
$A\rightarrow A\vee B$ is filtration-preserving by definition~\eqref{eq: wedge filtration}.
\end{remark}

We need the following well-known lemma, which spells out the intuition that $\Ccal A$ is a free $\Ccal$-module. Let $\eta$ denote the unit map for the monad~$\Ccal$, i.e. $\eta(A):A\rightarrow \Ccal A$.
The first part of the lemma gives a criterion for a map out of $\Ccal A$ to be a map of $\Ccal$-modules. The second part of the lemma identifies the image of $\eta(A)$ as the
``generators" of~$\Ccal A$, in the sense that module maps $\Ccal A\rightarrow X$ are determined by that they do when restricted from $\Ccal A$ to~$A$. Both are analogous to properties of free modules over a ring.

\begin{lemma}    \label{lemma: check on generators}
\mbox{}\hfill
\begin{enumerate}
\item \label{item: composites}
Suppose that $X$ is a $\Ccal$-module with structure map $\mu_{X}\colon\Ccal X\rightarrow X$.
Then a map $f:\Ccal A \rightarrow X$ is a map of $\Ccal$-modules if and only if
\[
f=\mu_{X}\circ\Ccal(f\circ\eta(A)).
\]
\item
\label{item: former corollary}
Two $\Ccal$-module maps $f, \fwiggle:\Ccal A\rightarrow X$ are equal if and only if $f\circ\eta(A)=\fwiggle\circ\eta(A)$.
\end{enumerate}
\end{lemma}

\begin{proof}
For \eqref{item: composites}, consider the following diagram:
\begin{equation}     \label{diagram: C-modules}
\begin{gathered}
\xymatrix{
\Ccal A             \ar[rr]^-{\Ccal(\eta(A))}
                    \ar[rrd]_-{=}
   && \Ccal(\Ccal A) \ar[rr]^-{\Ccal(f)}
                     \ar[d]^-{\mu_{\Ccal}(A)}
   && \Ccal X \ar[d]^-{\mu_{X}}\\
&&\Ccal A\ar[rr]_-{f}&& X.
}
\end{gathered}
\end{equation}
If $f$ is a map of $\Ccal$-modules, then the rectangle commutes. Since the triangle necessarily commutes by the unital property of~$\Ccal$, we find that the perimeter commutes, that is,  $f=\mu_{X}\circ\Ccal(f\circ\eta(A))$.

Next, suppose that $f=\mu_{X}\circ\Ccal(f\circ\eta(A))$; we want to show that $f$ is a map of $\Ccal$-modules. However, associativity of a $\Ccal$-module action tells us that $\mu_{X}$ is a map of $\Ccal$-modules, and $\Ccal(f\circ\eta(A))$ is one likewise, simply because it is $\Ccal$ applied to a map. Hence $f$ is a composition of maps of $\Ccal$-modules, and the result follows.

For~\eqref{item: former corollary}, observe that if $f$ and $\fwiggle$ are $\Ccal$-module maps, then diagram~\eqref{diagram: C-modules} commutes for each of them.
Consider the top row. By assumption $f\circ\eta(A)=\fwiggle\circ\eta(A)$,
and the top row is $\Ccal$ applied to this composition in each case. Hence
$f$ and $\fwiggle$ give the same map across the top row of diagram~\eqref{diagram: C-modules}. Therefore they give the same bottom row as well, that is, $f=\fwiggle$.
\end{proof}

The remainder of the section shows how $\beta$ and $\gamma$, defined in \eqref{eq: define beta} and~\eqref{eq: define gamma}, interact with the $\Ccal$-module action maps, which are used to define face maps in $\Dcal_{\bullet}$ and~$\Ecal_{\bullet}$.
First, the following elementary lemma. In this lemma, $X$ and $Y$ can be
objects in $\filtered$, or $\Gamma$-spaces taking values in~$\filtered$.

\begin{lemma}     \label{lemma: product of modules}
Suppose that $X$ and $Y$ are filtered $\Ccal$-modules with filtration-preserving structure maps $\mu_{X}$ and~$\mu_{Y}$, respectively.
\begin{enumerate}
\item $X\times Y$ is endowed with a (two-variable) filtration-preserving $\Ccal$-module structure via the structure maps
\[
\beta\colon
        \Ccal\left(X\times Y\right)
\xrightarrow{\ \Ccal\pi_{1}\times\Ccal\pi_{2}\ }
        \Ccal X \times \Ccal Y
\xrightarrow{\mu_{\X}\times\mu_{Y}}
       X \times Y
\]
\[
\eta\colon X\times Y  
\xrightarrow{\ \eta_{\Ccal}\ }
 \Ccal(X\times Y).
\]
\item \label{item: map into product}
If $f\colon X\rightarrow W$ and $g\colon Y\rightarrow Z$ are maps of filtered $\Ccal$-modules, then $f\times g\colon X\times Y\rightarrow W\times Z$ is likewise a map of filtered $\Ccal$-modules.
\end{enumerate}
\end{lemma}

\begin{proof}
The content of the lemma in the statements about $\Ccal$-modules, because preservation of filtration has already been established.

The outside square in~\eqref{diagram: module map}, below, must commute to establish the associativity for~$\beta$. It does so because both ways around the square factor through the middle, and the parallel diagonal arrows are actually the same map:

\begin{equation}    \label{diagram: module map}
\begin{gathered}
\xymatrix{
   \Ccal\left(\Largestrut X\times Y
                \right)\qquad\quad
   \ar@<-5.0ex>[dd]_-{\beta}
&& \quad
\Ccal\Ccal\left(\Largestrut X\times Y
              \right)
   \ar[ll]_-{\Ccal(\beta)}
   \ar[dd]^-{\mu_{\Ccal}}
   \ar[dl]^{\Ccal\Ccal\pi_{1}\times\Ccal\Ccal\pi_{2}}
\\
&
\Ccal\Ccal X\times\Ccal\Ccal Y
   \ar@<.5ex>[dl]^-{
   \quad(\mu_{X}\times\mu_{Y})\circ(\mu_{\Ccal}\times\mu_{\Ccal})
   }
   \ar@<-1.0ex>[dl]_-{
   (\mu_{X}\times\mu_{Y})\circ(\mu_{X}\times\mu_{Y})\quad
   }
 &\\
X\times Y \qquad\qquad
&& \quad
\Ccal\left(\Largestrut X\times Y\right).
   \ar@<.5ex>[ll]^{\beta}
}
\end{gathered}
\end{equation}
\noindent The unit condition for $\beta$ comes from the diagram below,
in which each of the smaller triangles commutes:
\[
\xymatrix{
X\times Y
    \ar[d]_-{\eta}
    \ar[drr]^{\eta\times\eta}
    \ar@<.8ex>[drrrr]^{\id\times\id}
\\
\Ccal\left( X\times Y\right)
    \ar@/^-1.0pc/[rrrr]_{\beta}
    \ar[rr]^{\Ccal\pi_{1}\times\Ccal\pi_{2}}
&& \Ccal X\times\Ccal Y
    \ar[rr]^{\mu_{\Gcal}\times\mu_{\Gcal}}
&& X\times Y.
}
\]

Now suppose that $f$ and $g$ are as indicated; we must show that $f\times g$ is a map of $\Ccal$-modules. This is established by the following diagram, in which the outer rectangle commutes because the two squares do so:
\[\xymatrix{
\Ccal(W\times Z)\,
    \ar[d]_-{\Ccal\pi_{1}\times\Ccal\pi_{2}} &
\,\Ccal(X\times Y)
     \ar[l]^-{\Ccal(f\times g)}
     \ar[d]^-{\Ccal\pi_{1}\times\Ccal\pi_{2}}\\
\Ccal W\times\Ccal Z
     \ar[d]_-{\mu_W\times\mu_Z}
&\Ccal X\times\Ccal Y
     \ar[l]_-{\Ccal f\times\Ccal g}
     \ar[d]^-{\mu_X\times\mu_Y}
\\
W\times Z
     &X\times Y.
     \ar[l]_-{f\times g}
}
\]
\end{proof}

The resulting corollary is what we need to show that $\Dcal_{\bullet}$
\eqref{eq: define D} is a simplicial space, since the face maps make use of~$\beta$ \eqref{eq: define beta}.

\begin{corollary}
\label{corollary: product of multiplications is module map}
Let $\Gcal$ be a filtered $\Ccal$-module. Then the map $\beta$ of \eqref{eq: define beta}
endows $\Gcal A \times \Gcal B$ with the structure of a (two-variable) filtered $\Ccal$-module.
The map
$\mu_{\Gcal}\times\mu_{\Gcal}\colon
     \Ccal\Gcal A\times\Ccal\Gcal B\rightarrow\Gcal A\times\Gcal B
$
is a map of filtered $\Ccal$-modules.
\end{corollary}

Next we work on $\gamma$ \eqref{eq: define gamma}, which appears in simplicial structure maps for~$\Ecal_{\bullet}$ \eqref{eq: define E}.

\begin{lemma} Let $\Gcal$ be a $\Ccal$-module.
     \label{lemma: structure maps give C-modules}
Then
$\gamma\colon
     \Ccal\left[\largestrut\Gcal A\vee\Gcal B\right]
     \rightarrow\Gcal(A\vee B)$ is a filtration-preserving map of (two-variable) $\Ccal$-modules.
\end{lemma}

\begin{proof}
The assertion that $\gamma$ is filtration preserving was addressed in Remark~\ref{remark: beta and gamma preserve filtration}.
We must show that the outside of the following diagram commutes to establish the multiplication condition:
\[
\xymatrix{
\Ccal\Gcal(A\vee B)\quad
    \ar[d]_-{\mu_{\Gcal}(A\vee B)}
&&\Ccal\Ccal\Gcal( A\vee B)
    \ar[ll]_-{\Ccal(\mu_{\Gcal}(A\vee B))}
    \ar[d]^-{\mu_{\Ccal}(\Gcal(A\vee B))}
&&
\Ccal\Ccal\left[\Largestrut\Gcal A\vee \Gcal B\right]
    \ar[ll]_-{\Ccal\Ccal(\Gcal\iota_{A}\vee\Gcal\iota_{B})}
    \ar[d]^-{\mu_{\Ccal}(\Gcal A\vee\Gcal B)}
    \ar@/_1.8pc/[llll]_-{\Ccal(\gamma)}
    \\
\Gcal (A\vee B)
&& \Ccal\Gcal (A\vee B)
     \ar[ll]_-{\mu_{\Gcal}(A\vee B)}
&& \Ccal\left[\Largestrut\Gcal A\vee \Gcal B\right].
     \ar[ll]_-{\Ccal(\Gcal\iota_{A}\vee\Gcal\iota_{B})}
     \ar@/^1.3pc/[llll]^-{\gamma}
}
\]
The right-hand square commutes by naturality of~$\mu_{\Ccal}$. The left-hand square commutes by the associativity of the action of $\Ccal$ on~$\Gcal$.
\end{proof}

The maps $\beta$ and $\gamma$ are involved in defining the simplicial structure of $\Dcal_{\bullet}$ and~$\Ecal_{\bullet}$ (Propositions~\ref{proposition: D is a simplicial space} and~\ref{proposition: E is a simplicial space}). We also need a map
$\Ecal_{\bullet}\rightarrow\Dcal_{\bullet}$ (see~\eqref{diagram: map of E to D}), for which the following lemma is essential. Recall that $\collapsemap{A}: A\vee B\rightarrow A$ denotes the map that collapses $B$ to a point.

\begin{lemma}   \label{lemma: special is a C-module}
The map
$\Ccal(A\vee B)
\xrightarrow{\Ccal\collapsemap{A}\times \Ccal\collapsemap{B}}
\Ccal A\times\Ccal B$
is a map of filtered $\Ccal$-modules.
\end{lemma}

\begin{proof}
The $\Ccal$-module structure on the target is provided by~$\beta$.
Here is the necessary commutative diagram to show that
$\Ccal(A\vee B)\rightarrow \Ccal A\times\Ccal B $ is a map of $\Ccal$-modules:
\[\xymatrix{
\Ccal(A\vee B)
      \ar[dd]_-{\Ccal\collapsemap{A}\times\Ccal\collapsemap{B}}
&&
&& \Ccal\Ccal(A\vee B)
      \ar[llll]^-{\mu_{\Ccal}}
      \ar[lld]^-{\Ccal\Ccal\collapsemap{A}\times\Ccal\Ccal\collapsemap{B}}
      \ar[dd]^-{\Ccal(\Ccal\collapsemap{A}\times\Ccal\collapsemap{B})}\\
&& \Ccal\Ccal A\times\Ccal\Ccal B
      \ar[lld]_-{\mu_{\Ccal}\times\mu_{\Ccal}}
      \\
\Ccal A\times\Ccal B && && \Ccal(\Ccal A\times\Ccal B).
      \ar[llu]_-{\Ccal\pi_{1}\times\Ccal\pi_{2}}
      \ar[llll]_-{\beta}
}
\]
The lower triangle commutes by definition of~$\beta$. The upper quadrilateral commutes because the map $\Ccal\Ccal\to \Ccal$ is a natural transformation.
For commutativity of the right triangle, we consider the maps into~$\Ccal\Ccal A$,
the other factor being similar:
\[
\Ccal(\pi_{1})\circ\Ccal(\Ccal\collapsemap{A}\times\Ccal\collapsemap{B})
       =\Ccal\left(\pi_{1}\circ (\Ccal\collapsemap{A}\times\Ccal\collapsemap{B})\right)
       =\Ccal(\Ccal\collapsemap{A}).
\]

To establish that
$\Ccal\collapsemap{A}\times \Ccal\collapsemap{B}\colon\Ccal(A\vee B)
\xrightarrow{}
\Ccal A\times\Ccal B $
preserves filtration, we apply Lemma~\ref{lemma: check on generators}.
Let $g$ denote the restriction of
$\Ccal\collapsemap{A}\times \Ccal\collapsemap{B}$
to $A\vee B\subseteq\Ccal(A\vee B)$. We can express $g$ as the composite
\[
A\vee B  \xrightarrow{\,\eta\vee\eta\,} \Ccal A \vee\Ccal B
         \xrightarrow{\quad} \Ccal A \times\Ccal B,
\]
where the second map is the standard inclusion of a wedge into a product.
Observe that $g$ preserves filtration because
the unit map for $\Ccal$ is filtration-preserving (Lemma~\ref{lemma: filtered monad}),
and so is the wedge-to-product inclusion of filtered spaces
(see \eqref{eq: wedge filtration} and~\eqref{eq: define product filtration}).
Now consider the composite
\[
\Ccal(A\vee B)\xrightarrow{\ \Ccal g\ }\Ccal(\Ccal A \times\Ccal B)
              \xrightarrow{\ \beta\ } \Ccal A \times\Ccal B.
\]
Since $g$ preserves filtration, so does $\Ccal g$ by naturality. Further, $\beta$ is filtration-preserving by Lemma~\ref{lemma: product of modules}.
But by Lemma~\ref{lemma: check on generators},
$\beta\circ\Ccal g=\Ccal\collapsemap{A}\times \Ccal\collapsemap{B}$, and this finishes the proof.
\end{proof}


\section{The simplicial space $\Dcal_{\bullet}$}
    \label{section: D is a simplicial space}

In this section, we show that
$\Dcal_{\bullet}$ (see equation~\eqref{eq: define D}) is a simplicial space.
We need to define the face and degeneracy maps. A feature to notice is that not all face and degeneracy maps from a fixed summand $D_{k}^{i}$ have the same target.
The face and degeneracy maps that originate from the summand
\begin{equation}   \label{eq: Dki}
D_{k}^{i}=\Ccal\left[\LARGEstrut\filter{m}
   \Ccal^{k-i}\left(
     \Largestrut\Ccal^{i}\Fcal X\times\Ccal^{i}\Fcal Y
              \right)
                 \right]
\end{equation}
come in four general types, plus one special case, the face map~$d_{k-i}$. The diagram after the list below gives a visualization of the organization of the structure maps, but note that to save space, the degeneracy maps originating from $D_{k}^{i}$ are not shown and instead the full set of degeneracies originating from $D_{k-2}^{i-2}$ is
exhibited. We make frequent use of Lemma~\ref{lemma: filtration of composition, Greg's version},
which tells us that $\mu_{\Ccal}:\Ccal\Ccal\rightarrow\Ccal$ restricts to a map
$\Ccal(\filter{m}\Ccal)\rightarrow\Ccal\filter{m}$.
\begin{enumerate}
\item (Horizontal degeneracies) The maps $s_{0},...,s_{k-i}:D_{k}^{i}\rightarrow D_{k+1}^{i}$
    come from all the ways of inserting a $\Ccal$ into $\Ccal^{k-i}$ in~\eqref{eq: Dki}, which is $k-i+1$ ways altogether. These maps are given exactly the same way as
    the degeneracy maps of the same names in $\FilteredBK{m}{\Fcal}$ (see Proposition~\ref{proposition: structure maps respect filtration}).
\item (Vertical degeneracies)
\label{item: vertical degeneracies D}
The maps $s_{k-i+1},\dots,s_{k}: D_{k}^{i}\rightarrow D_{k+1}^{i+1}$ come from the ways of inserting $\Ccal$ to the {\emph{right}} of a $\Ccal$ in $\Ccal^{i}\Fcal$, which is $i$~ways altogether.
\item (Horizontal faces) The maps $d_{0},\ldots,d_{k-i-1}:D_{k}^{i}\rightarrow D_{k-1}^{i}$ come from applying the monad multiplication $\mu_{\Ccal}:\Ccal^{2}\rightarrow\Ccal$ to
    neighboring factors of $\Ccal$ in $\Ccal\left[\filter{m}\Ccal^{k-i}\right]$. There are $k-i$ such pairs to multiply, and the faces correspond to the face maps with the same names in $\FilteredBK{m}{\Fcal}$. (See Proposition~\ref{proposition: structure maps respect filtration}. Note the definition of $d_{0}$ relies on Lemma~\ref{lemma: filtration of composition, Greg's version}.)
\item (``Borderline" face, horizontal) The face map
$d_{k-i}:D_{k}^{i}\rightarrow D_{k-1}^{i}$ requires special attention. It multiplies the closest factor of $\Ccal$ exterior to the product onto each factor in the product. More precisely, $d_{k-i}$ is induced by applying
$\Ccal\filter{m}\Ccal^{k-i-1}$ to the map $\beta$
of \eqref{eq: define beta} with the filtered spaces
$A=\Ccal^{i-1}\Fcal X$ and $B=\Ccal^{i-1}\Fcal Y$.
To know that this is possible, we have to know that $\beta$ is filtration-preserving, which is true by Remark~\ref{remark: beta and gamma preserve filtration}.
\item (Vertical faces)
\label{item: vertical faces D}
The maps
$d_{k-i+1},\ldots,d_{k}:D_{k}^{i}\rightarrow D_{k-1}^{i-1}$
come from applying the monad multiplication $\mu_{\Ccal}\colon\Ccal^{2}\rightarrow\Ccal$
(or the module multiplication $\mu_{\Fcal}\colon\Ccal\Fcal\rightarrow\Ccal$ in the case of~$d_{k}$) to neighboring pairs $\Ccal\Ccal$ in both factors of $\Ccal^{i}\Fcal X\times\Ccal^{i}\Fcal Y$ (using the same neighboring pairs in each factor).
\end{enumerate}

The following diagram shows how the faces and degeneracies are organized, with the borderline face map indicated in bold font:
\[
\xymatrix{
{\overbrace{\Ccal\left[\LARGEstrut\filter{m}
   \Ccal^{k-i-1}\left(\Largestrut\Ccal^{i}\Fcal X\times\Ccal^{i}\Fcal Y
                \right)
                 \right]}^{D_{k-1}^{i}}}\quad
   \ar@<-1.5ex>[r]
   \ar@<-2.0ex>[r]
   \ar@<-2.5ex>[r]_-{s_{0},\ldots, s_{k-i-1}}
   \ar@<-6ex>[dd]_-{d_{k-i},\ldots,d_{k-1}}
   \ar@<-5ex>[dd]
   \ar@<-4ex>[dd]
   \ar@<-3ex>[dd]
&\Hugestrut\quad
{\overbrace{\Ccal\left[\LARGEstrut\filter{m}
   \Ccal^{k-i}\left(\Largestrut\Ccal^{i}\Fcal X\times\Ccal^{i}\Fcal Y
              \right)
                 \right]}^{D_{k}^{i}}}
   \ar@<-6ex>[dd]_-{d_{k-i+1},\ldots,d_{k}}
   \ar@<-5ex>[dd]
   \ar@<-4ex>[dd]
   \ar@<-3ex>[dd]
         \ar@<-4.0ex>[l]_-{d_{0},\ldots,d_{k-i-1}}
         \ar@<-3.5ex>[l]
         \ar@<-3.0ex>[l]
         \ar@<-2.0ex>[l]^-{{\mathbf{d_{k-i}}}}
\\
& \\
{\underbrace{\Ccal\left[
   \LARGEstrut\filter{m}
   \Ccal^{k-i-1}\left(\Largestrut\Ccal^{i-1}\Fcal X\times\Ccal^{i-1}\Fcal Y\right)
   \right]}_{D_{k-2}^{i-1}}}\quad
   \ar@<-4ex>[uu]
   \ar@<-5ex>[uu]
   \ar@<-6ex>[uu]_-{s_{k-i},\ldots,s_{k-2}}
   \ar@<-1.5ex>[r]
   \ar@<-2.0ex>[r]
   \ar@<-2.5ex>[r]_-{s_{0},\ldots, s_{k-i-1}}
& \quad{\underbrace{\Ccal\left[
   \LARGEstrut\filter{m}
   \Ccal^{k-i}\left(\Largestrut\Ccal^{i-1}\Fcal X\times\Ccal^{i-1}\Fcal Y\right)
   \right]}_{D_{k-1}^{i-1}}}.
   \ar@<-4ex>[uu]
   \ar@<-5ex>[uu]
   \ar@<-6ex>[uu]_-{s_{k-i+1},\ldots,s_{k-1}}
         \ar@<-4.0ex>[l]_-{d_{0},\ldots,d_{k-i-1}}
         \ar@<-3.5ex>[l]
         \ar@<-3.0ex>[l]
         \ar@<-2.0ex>[l]^-{{\mathbf{d_{k-i}}}}
}
\]

\begin{proposition}
\label{proposition: proving D is a simplicial space}
$\Dcal_{\bullet}$ is a simplicial space.
\end{proposition}

\begin{proof}
We need to verify the simplicial identities,
\begin{enumerate}
\item $d_{i}d_{j} = d_{j-1}d_{i}$ if $ i<j$
\item $s_{i}s_{j}=s_{j+1}s_{i}$ if $i\leq j$
\item $d_{i}s_{j}=s_{j-1}d_{i}$ if $i<j$
\item $d_{j}s_{j}=\id=d_{j+1}s_{j}$
\item $d_{i}s_{j}=s_{j}d_{i-1}$ if $i>j+1$.
\end{enumerate}
First, consider a simplicial identity that does not involve a borderline face map.
\begin{itemize}
\item If the identity involves only vertical or only horizontal maps (but not the borderline face), the identity is satisfied in a straightforward way because $\Ccal$ is a monad, and (in the case of vertical maps) because $\Fcal$ is a left module over~$\Ccal$.
\item If one side of a simplicial identity involves both a vertical map and a horizontal map, then so does the other.
    The vertical map is applying the monad or module structure for $\Ccal$ and $\Fcal$ \emph{inside} the product $\Ccal^{i}\Fcal X\times\Ccal^{i}\Fcal Y$.
    The horizontal map is applying the monad structure for $\Ccal$
    \emph{outside} the product.
    Provided that the borderline face map is not involved, these operations commute by functoriality. We call this argument ``commuting operations."
\end{itemize}

%
The tricky situations are those where a simplicial identity involves a borderline face map. Most of the required identities are true by ``commuting operations" even when a borderline face map is
involved\footnote{It is critical here that we do {\bf{not}} have a degeneracy that takes
$\Ccal^{i}\Fcal X\times\Ccal^{i}\Fcal Y$ to
$\Ccal\Ccal^{i}\Fcal X\times\Ccal\Ccal^{i}\Fcal Y$,
i.e. no degeneracy puts a $\Ccal$ immediately to the left of the functors $\Ccal^{i}$ and inside the parentheses in
$
\Ccal\left[\strut\filter{m}
   \Ccal^{k-i}\left(
   \strut\Ccal^{i}\Fcal X \times\Ccal^{i}\Fcal Y
              \right)
                 \right]
$.},
but there are three situations involving
$d_{k-i}:D_{k}^{i}\rightarrow D_{k-1}^{i}$
that require individual checking. As in the diagram above, a borderline face map is indicated in a bold font.
\begin{itemize}
\item (Going from $D_{k}^{i}$ to $D_{k-2}^{i}$)
Two borderline faces in a row:
\begin{equation}   \label{eq: D borderline face identity}
\begin{gathered}
{\mathbf{d_{k-i-1}}}{\mathbf{d_{k-i}}}={\mathbf{d_{k-i-1}}}d_{k-i-1}
\qquad\qquad
\xymatrix{
D_{k-1}^{i}\ \ar[d]_-{{\mathbf{d_{k-i-1}}}}
       & D_{k}^{i}
            \ar[l]_-{{\mathbf{d_{k-i}}}}
            \ar[d]^-{d_{k-i-1}}\\
D_{k-2}^{i}\        & D_{k-1}^{i}\ar[l]^-{{\mathbf{d_{k-i-1}}}}
}
\end{gathered}
\end{equation}
The business part of the diagram is given by
\[
\xymatrix{
   \Ccal\left(\Largestrut\Ccal^{i}\Fcal X\times\Ccal^{i}\Fcal Y
                \right)
   \ar[d]_-{\beta}
&
   \Ccal\Ccal\left(\Largestrut\Ccal^{i}\Fcal X\times\Ccal^{i}\Fcal Y
              \right)
   \ar[l]_{\Ccal(\beta)}
   \ar[d]_-{\mu_{\Ccal}}
\\
   \Ccal^{i}\Fcal X\times\Ccal^{i}\Fcal Y
&
   \Ccal\left(\Largestrut\Ccal^{i}\Fcal X\times\Ccal^{i}\Fcal Y\right),
   \ar[l]_-{\beta}
}
\]
which commutes because $\beta$ defines a $\Ccal$-module
structure on the product $\Ccal^{i}\Fcal X\times\Ccal^{i}\Fcal Y$
(Lemma~\ref{lemma: structure maps give C-modules}).
\\
\item (Going from $D_{k}^{i}$ to $D_{k-2}^{i-1}$)
Vertical face $d_{k-i+1}$, then the borderline face map:
\begin{equation}  \label{eq: D borderline and other face}
\begin{gathered}
{\mathbf{d_{k-i}}}d_{k-i+1} = d_{k-i}{\mathbf{d_{k-i}}}
\qquad\qquad
\xymatrix{
D_{k-1}^{i}\ \ar[d]_-{d_{k-i}}
       & D_{k}^{i}
            \ar[l]_-{{\mathbf{d_{k-i}}}}
            \ar[d]^-{d_{k-i+1}}\\
D_{k-2}^{i-1}\ 
      & D_{k-1}^{i-1}
            \ar[l]^-{{\mathbf{d_{k-i}}}}
}
\end{gathered}
\end{equation}
The business part of the diagram, below, commutes by applying
Lemma~\ref{lemma: check on generators} (all the maps
are maps of $\Ccal$-modules), and observing that
both ways around the square are the same map on
$\Ccal\Ccal^{i-1}\Fcal X\times\Ccal\Ccal^{i-1}\Fcal Y$:
\[
\xymatrix{
\Ccal\Ccal^{i-1}\Fcal X\times\Ccal\Ccal^{i-1}\Fcal Y
   \ar[d]_-{\mu_{\Ccal^{i-1}\Fcal}\times\mu_{\Ccal^{i-1}\Fcal}}
&\Ccal\left(\Largestrut\Ccal\Ccal^{i-1}\Fcal X\times\Ccal\Ccal^{i-1}\Fcal Y
              \right)
   \ar[l]_-{\beta}
   \ar[d]^-{\Ccal\left(\mu_{\Ccal^{i-1}\Fcal}\times\mu_{\Ccal^{i-1}\Fcal}\right)}
\\
\Ccal^{i-1}\Fcal X\times\Ccal^{i-1}\Fcal Y
& \Ccal\left(\Largestrut\Ccal^{i-1}\Fcal X\times\Ccal^{i-1}\Fcal Y\right).
   \ar[l]_-{\beta}
}
\]
\item (Going from $D_{k-1}^{i}$ to itself)
Horizontal degeneracy $s_{k-i-1}$, then the borderline face map:
\begin{equation}   \label{eq: D degeneracy identity}
{\mathbf{d_{k-i}}}s_{k-i-1}=\id
\qquad\qquad
\xymatrix{
D_{k-1}^{i}\ar[r]^{s_{k-i-1}}
           \ar[dr]_-{\id}
       & D_{k}^{i}
            \ar[d]^{{\mathbf{d_{k-i}}}}\\
       & D_{k-1}^{i}
}
\end{equation}
The business part of the diagram, below, commutes by
Lemma~\ref{lemma: structure maps give C-modules}:
\[
\xymatrix{
\Ccal^{i}\Fcal X\times\Ccal^{i}\Fcal Y
      \ar[r]^-{\eta}
      \ar[dr]_(.4){\id}
&\Ccal\left(\Largestrut\Ccal^{i}\Fcal X\times\Ccal^{i}\Fcal Y
              \right)
      \ar[d]^{\beta}
\\
%
& \left(\Largestrut\Ccal^{i}\Fcal X\times\Ccal^{i}\Fcal Y\right).
}
\]
\end{itemize}

\end{proof}

\section{Deformation retraction
$\Bcal_{\bullet}\hookrightarrow\Dcal_{\bullet}$}
\label{section: homotopy of Ds}

In this section, we turn our attention to the right-hand column of
diagram~\eqref{diagram: square of auxiliaries}, which we reproduce here for reference:
\begin{equation}   \tag{\ref{diagram: square of auxiliaries}}
\thatdiagram
\end{equation}
There
is an inclusion $j: B_{\bullet}\hookrightarrow\Dcal_{\bullet}$ given by
the fact that $B_{k}$ is one wedge summand of~$\Dcal_{k}$:
\begin{equation}   \label{eq: defn j}
j_{k}\colon B_{k}=\Ccal\left[\LARGEstrut\filter{m}
   \Largestrut\Ccal^{k}\Fcal X
                 \right]
\times
\Ccal\left[\LARGEstrut\filter{m}
   \Largestrut\Ccal^{k}\Fcal Y
                 \right]
   =D_{k}^{k+1}\hookrightarrow\Dcal_{k}.
\end{equation}
The face and degeneracy maps are defined the same way on $B_{k}$
and~$D_{k}^{k+1}$
(see Proposition~\ref{proposition: structure maps respect filtration} for~$B_{k}$, and items
\eqref{item: vertical degeneracies D} and \eqref{item: vertical faces D}
after \eqref{eq: Dki} for~$D_{k}^{k+1}$), and as a result we have the following statement.

\begin{lemma}
$j:B_{\bullet}\hookrightarrow\Dcal_{\bullet}$ is the inclusion of a simplicial subspace.
\end{lemma}

Our goal in this section is to produce the retraction $q\colon\Dcal_{\bullet}\rightarrow B_{\bullet}$ on the right side of \eqref{diagram: square of auxiliaries},
and a simplicial homotopy $h$ between the composite
$\Dcal_{\bullet}\xrightarrow{\ q\ }B_{\bullet}\xrightarrow{\ j\ }\Dcal_{\bullet}$
and the identity map of $\Dcal_{\bullet}$.
(This will prove Proposition~\ref{proposition: existence of homotopy}.)

Producing the retraction map $q$ is not difficult. Define
$q:\Dcal_{\bullet}\rightarrow\Dcal_{\bullet}$ on the summands $D_{k}^{i}$ of $\Dcal_{\bullet}$ as
the restriction of
$
q_{k}=\Ccal^{k-i+1}(\pi_1)\times\Ccal^{k-i+1}(\pi_2)
$,
taking
\[
\underbrace{\Ccal\left[\LARGEstrut\filter{m}
   \Ccal^{k-i}\left(\Largestrut\Ccal^{i}\Fcal X\times\Ccal^{i}\Fcal Y
              \right)\right]}_{D_{k}^{i}}
\xlongrightarrow{\quad}
\underbrace{\Ccal\left[\LARGEstrut\filter{m}
   \Largestrut\Ccal^{k}\Fcal X
                 \right]
\times
\Ccal\left[\LARGEstrut\filter{m}
   \Largestrut\Ccal^{k}\Fcal Y
                 \right]}_{D_{k}^{k+1}=B_{k}}.
\]
While defining the simplicial homotopy $j\circ q\simeq \id_{\Dcal_{\bullet}}$ is necessarily more involved than defining~$q$, a straightforward candidate presents itself and turns out to work. Thus the principal result of this section is the following.

\begin{proposition}  \label{proposition: simplicial homotopy}
The map $q:\Dcal_{\bullet}\rightarrow B_{\bullet}$ is a simplicial retraction of the map~$j$. Further, there is a simplicial homotopy~$h$, defined in~\eqref{eq: define simplicial homotopy}, between the identity map of~$\Dcal_{\bullet}$ and $j\circ q$, and $h$ is the identity on~$\left\{D_{k}^{k+1}\right\}$.
\end{proposition}

Proposition~\ref{proposition: simplicial homotopy} follows from
Lemma~\ref{lemma: homotopy compatible with degeneracies} and
Lemma~\ref{lemma: homotopy compatible with faces} below, which
state that a particular set of maps we define below satisfies the necessary requirements to be a homotopy. (As a consequence, we find for free that $q$ is in fact a simplicial map.)
The data required to specify a simplicial homotopy
in simplicial dimension $k$ is a sequence of maps
\[
 h^{0}_{k},...,h^{k+1}_{k}:\Dcal_{k}\longrightarrow \Dcal_{k},
\]
which have to satisfy the following conditions.
\begin{itemize}
\item Compatibility with degeneracies
\begin{equation}
\label{eq: compatibility with degeneracies}
s_{j}\left[h^{i}_{k}(a)\right]=
\begin{cases}
h^{i}_{k+1}\left[s_{j}(a)\right]
           &\text{for $0\leq j\leq k-i$}\\
\LARGEstrut h^{i+1}_{k+1}\left[s_{j}(a)\right]
           &\text{for $k-i< j\leq k$.}
\end{cases}
\end{equation}
\item Compatibility with faces
\begin{equation}   \label{eq: compatibility with faces}
d_{j}\left[h^{i}_{k}(a)\right]=
\begin{cases}
h^{i}_{k-1}\left[d_{j}(a)\right]
           &\text{for $0\leq j\leq k-i$}\\
\LARGEstrut h^{i-1}_{k-1}\left[d_{j}(a)\right]
           &\text{for $k-i< j\leq k$.}
\end{cases}
\end{equation}
\end{itemize}
Given such data, the collection $\left\{h_{k}^{i}\right\}_{0\leq i\leq k}$ defines a simplicial homotopy between $\left\{h_{k}^{0}\right\}_{k\in\integers_{\geq 0}}$ and $\left\{h_{k}^{k+1}\right\}_{k\in\integers_{\geq 0}}$, which are, by the required identities, necessarily maps of simplicial spaces.


\bigskip

The motivation for the simplicial homotopy we define is to interpolate between the source and the target of the map defined in Proposition~\ref{proposition: simplicial homotopy} by ``moving one $\Ccal$ at a time across the parentheses."
If $0\leq n\leq i$, there is an evident map~$\delta_{k}^{i}$,
\[
\begin{gathered}
\xymatrix
{
D_{k}^{n}
 =\Ccal\left[
   \LARGEstrut\filter{m}
   \Ccal^{k-i}\Ccal^{i-n}\left(\Largestrut\Ccal^{n}\Fcal X \times\Ccal^{n}\Fcal Y \right)
   \right]
   \ar[d]^{\delta_{k}^{i}}\\
D_{k}^{i}
 =\Ccal\left[
   \LARGEstrut\filter{m}
   \Ccal^{k-i}\left(\Largestrut\Ccal^{i}\Fcal X \times\Ccal^{i}\Fcal Y \right)
   \right],
}
\end{gathered}
\]
induced by using $\beta$ (see \eqref{eq: define beta}) to push $\Ccal^{i-n}$ across the parentheses. For each $i=0,...,k+1$,
we define the corresponding map in the proposed simplicial homotopy,  $h_{k}^{i}:\Dcal_{k}\rightarrow\Dcal_{k}$, by
\begin{equation}   \label{eq: define simplicial homotopy}
h_{k}^{i}\left(D_{k}^{n}\right)=
\begin{cases}
\id_{D_{k}^{n}}
        & \text{if $i\leq n$}\\
\delta_{k}^{i}:D_{k}^{n}\rightarrow D_{k}^{i}
        & \text{if $0\leq n\leq i$.}
\end{cases}
\end{equation}
To help with visualizing the homotopy, we show its organization below.
The rows in the diagram consist of $D_{k}^{i}$ for a fixed value of~$i$,
and the diagonals are the spaces wedged together to give $\Dcal_{k}$ for various values of~$k$.
Degeneracies go upward or to the right, and faces go downward or to the left
(cf. diagram preceding Proposition~\ref{proposition: proving D is a simplicial space}).
\[\xymatrix{
D_{2}^{3} & D_{3}^{3} & D_{4}^{3} & D_{5}^{3} & D_{6}^{3}\\
D_{1}^{2} & D_{2}^{2} & D_{3}^{2} & D_{4}^{2} & D_{5}^{2}\\
D_{0}^{1} & D_{1}^{1} & D_{2}^{1} & D_{3}^{1} & D_{4}^{1}\\
          & D_{0}^{0}\ar@<.7ex>[ul]^{h_{0}^{1}\!\!\!}
          & D_{1}^{0}\ar@<.7ex>[ul]^{h_{1}^{1}\!\!\!} \ar@<0ex>[uull]^(.8){h^{2}_{1}\!\!\!}
          & D_{2}^{0}\ar@<.7ex>[ul]^{h_{2}^{1}\!\!\!}
          \ar@<0ex>[uull]^(.8){h^{2}_{2}\!\!\!}
%
          & D_{3}^{0}\ar@<.7ex>[ul]^{h_{3}^{1}\!\!\!}
          \ar@<0ex>[uull]^(.8){h^{2}_{3}\!\!\!}
}
\]
All of the maps $h_{k}^{i}$ push diagonally upward and to the left.
For a fixed value of~$i$, the maps $h_{k}^{i}$ push the $0$-th through the $(i-1)$-st rows up to the $i$-th row, and leave the $i$-th row and anything above it alone. The last map in each simplicial dimension, $h_{k}^{k+1}$ in dimension~$k$, pushes everything along the diagonal to the leftmost column, which is actually $B_{\bullet}$ embedded in~$\Dcal_{\bullet}$ via $B_{k}=D_{k}^{k+1}$.

We tackle proving the compatibility requirements in two stages.

\begin{lemma}      \label{lemma: homotopy compatible with degeneracies}
$\left\{h_{k}^{i}\right\}$ is compatible with degeneracies, as
in~\eqref{eq: compatibility with degeneracies}.
\end{lemma}

\begin{proof}
Checking that \eqref{eq: compatibility with degeneracies} holds on $D_{k}^{n}$ depends on the relationship of $n$
to~$i$. For $n\geq i$, we have defined $h^{i}_{k}$ to be the identity on $D_{k}^{n}$.
On the other hand, $s_{j}$ either takes $D_{k}^{n}$ to $D_{k+1}^{n}$ or $D_{k+1}^{n+1}$, on which $h^{i}_{k+1}$ is also the identity, so \eqref{eq: compatibility with degeneracies} is certainly satisfied in these cases.

Next suppose that $n< i$ (so $k-i< k-n$). There are three possible cases, depending on the relationship of $j$, $k-i$, and $k-n$ (both degeneracies horizontal, one horizontal and one vertical, or both vertical). The required diagrams all turn out to commute by straightforward naturality. For the lower degeneracies, $s_{j}$ for $j\leq k-i$, we need to look at
\[\xymatrix{
D_{k}^{i}\ 
     \ar[rr]^-{s_{j\leq k-i}}
&& D_{k+1}^{i}
        \\
&&D_{k}^{n}
     \ar[ull]^{h_{k}^{i}}
     \ar[rr]_-{s_{j\leq k-i}}
&&
D_{k+1}^{n}\ .
     \ar[ull]_{h_{k+1}^{i}}
}
\]
Writing out the diagram shows the required condition is satisfied by ``commuting operations," because the degeneracy is happening on the outside and the homotopy is working on the inside:
\[
{
{
\xymatrix{
\Ccal\left[
   \LARGEstrut\filter{m}
   \Ccal^{k-i}\left(\Largestrut\Ccal^{i}\Fcal X \times\Ccal^{i}\Fcal Y \right)
   \right]
   \ar[r]^-{s_{j\leq k-i}}
&
\Ccal\left[
   \LARGEstrut\filter{m}
   \Ccal^{k-i+1}\left(\Largestrut\Ccal^{i}\Fcal X \times\Ccal^{i}\Fcal Y \right)
   \right]
\\
\Ccal\left[
   \LARGEstrut\filter{m}
   \Ccal^{k-i}\Ccal^{i-n}\left(\Largestrut\Ccal^{n}\Fcal X \times\Ccal^{n}\Fcal Y \right)
   \right]
   \ar[u]^{h_{k}^{i}}
   \ar[r]^-{s_{j\leq k-i}}
&
\ \Ccal\left[
   \LARGEstrut\filter{m}
   \Ccal^{k-i+1}\Ccal^{i-n}\left(\Largestrut\Ccal^{n}\Fcal X \times\Ccal^{n}\Fcal Y \right)
   \right].
   \ar[u]^{h_{k+1}^{i}}
}}     
}      
\]

The case $k-i<j\leq k-n$ involves one horizontal and one vertical degeneracy, while $k-n<j$ involves two vertical degeneracies:
\[
\xymatrix{
D_{k+1}^{i+1}
\\
D_{k}^{i}
   \ar[u]^-{s_{k-i<j}}
\\
               && D_{k}^{n}
                  \ar[rr]_-{s_{j\leq k-n}}
                  \ar[ull]^{h_{k}^{i}}
                    && D_{k+1}^{n}
                       \ar[uullll]_{h_{k+1}^{i+1}}
}
\xymatrix{
D_{k+1}^{i+1}
\\
D_{k}^{i}
   \ar[u]^-{s_{j>k-n}}
               && D_{k+1}^{n+1}
                   \ar[ull]_{h_{k+1}^{i+1}}
\\
               && D_{k}^{n}.
                  \ar[u]_-{s_{j>k-n}}
                  \ar[ull]^{h_{k}^{i}}
}
\]

Both of the necessary diagrams again commute by straightforward  naturality considerations, and we leave the details to the reader.
\end{proof}

\begin{lemma}    \label{lemma: homotopy compatible with faces}
$\{h_{k}^{i}\}$ is compatible with faces.
\end{lemma}

\begin{proof}
As in the case of the degeneracies
(Lemma~\ref{lemma: homotopy compatible with degeneracies}),
we get a lot for free: when
$i<n$ we know that $h_{k}^{i}$ is the identity on $D_{k}^{n}$,
and compatibility is satisfied because $d_{j}$ either takes $D_{k}^{n}$ to $D_{k-1}^{n}$ or $D_{k-1}^{n-1}$, where $h_{k}^{i}$ is also the identity map
(note the need for $i<n$ rather than $i\leq n$). Hence we only need to consider $i\geq n$ (so $k-i\leq k-n$).

For low-index face maps, the diagram that needs to commute is
\[\xymatrix{
D_{k-1}^{i}
&& D_{k}^{i}
     \ar[ll]_-{d_{j\leq k-i}}
        \\
&&D_{k-1}^{n}
     \ar[ull]^-{h_{k-1}^{i}}
&&
D_{k}^{n}.
     \ar[ll]^-{d_{j\leq k-i}}
     \ar[ull]_-{h_{k}^{i}}
}
\]
This becomes the following diagram, which is tautological for $i=n$, and for $i>n$ commutes for all $j\leq k-i$ by naturality of the monad multiplication of~$\Ccal$:
\[
\xymatrix{
\Ccal\left[
   \LARGEstrut\filter{m}
   \Ccal^{k-i-1}\left(\Largestrut\Ccal^{i}\Fcal X \times\Ccal^{i}\Fcal Y \right)
   \right]
&
\Ccal\left[
   \LARGEstrut\filter{m}
   \Ccal^{k-i}\left(\Largestrut\Ccal^{i}\Fcal X \times\Ccal^{i}\Fcal Y \right)
   \right]
      \ar[l]_-{d_{j\leq k-i}}
\\
\Ccal\left[
   \LARGEstrut\filter{m}
   \Ccal^{k-i-1}\Ccal^{i-n}
       \left(\Largestrut\Ccal^{n}\Fcal X \times\Ccal^{n}\Fcal Y
       \right)
   \right]
   \ar[u]^-{h_{k-1}^{i}}
&
\ \Ccal\left[
   \LARGEstrut\filter{m}
   \Ccal^{k-i}\Ccal^{i-n}
        \left(\Largestrut\Ccal^{n}\Fcal X \times\Ccal^{n}\Fcal Y
        \right)
   \right].
   \ar[u]^-{h_{k}^{i}}
   \ar[l]^(.48){d_{j\leq k-i}}
}      
\]

As with the degeneracies, the case $k-i<j\leq k-n$ involves one horizontal and one vertical face map, while $k-n<j$ involves two vertical face maps:
\[
\xymatrix{
D_{k}^{i}
   \ar[d]_-{d_{j>k-i}}
\\
D_{k-1}^{i-1}
\\
               && D_{k-1}^{n}
                  \ar[ull]^-{h_{k-1}^{i-1}}
                    && D_{k}^{n}
                       \ar[ll]^-{d_{j\leq k-n}}
                       \ar[uullll]_-{h_{k}^{i}}
}
\xymatrix{
D_{k}^{i}
   \ar[d]_{d_{j>k-n}}
\\
D_{k-1}^{i-1}
               && D_{k}^{n}
                   \ar[ull]_-{h_{k}^{i}}
                   \ar[d]^{d_{j>k-n}}
\\
               && D_{k-1}^{n-1}.
                  \ar[ull]^-{h_{k-1}^{i-1}}
}
\]
Both of the necessary diagrams again commute for straightforward reasons of naturality, and we once again leave the details to the reader.
\end{proof}


\section{The simplicial space $\Ecal_{\bullet}$}
       \label{section: E is a simplicial space}

In this section, we study the auxiliary simplicial space~$\Ecal_{\bullet}$
(see \eqref{eq: define E}), which is very similar to $\Dcal_{\bullet}$, but with wedges instead of products. The goal is to prove the first part of
Proposition~\ref{proposition: E is a simplicial space},
which states that $\Ecal_{\bullet}$ is a simplicial space
(Proposition~\ref{proposition: prove Ecal is a simplicial space} below),
as well as Proposition~\ref{proposition: map E to D}, which says that the natural map $\Ecal_{\bullet}\rightarrow\Dcal_{\bullet}$ is a simplicial map and a
weak equivalence.

The faces and degeneracies of $\Ecal_{\bullet}$ are organized as in~$\Dcal_{\bullet}$,
although this time it is a vertical face map that is singled out as a ``borderline" map:
\[
\xymatrix{
{\overbrace{\Ccal\left[\LARGEstrut\filter{m}
   \Ccal^{k-i-1}\Ccal\left(\Largestrut\Ccal^{i-1}\Fcal X\vee\Ccal^{i-1}\Fcal Y
                \right)
                 \right]}^{E_{k-1}^{i}}}\quad
   \ar@<-.5ex>[r]
   \ar@<-1.0ex>[r]
   \ar@<-1.5ex>[r]_-{s_{0},\ldots, s_{k-i-1}}
   \ar@<-14ex>[dd]_-{{\mathbf{d_{k-i}}}}
   \ar@<-11ex>[dd]
   \ar@<-10ex>[dd]
   \ar@<-9ex>[dd]^-{{d_{k-i+1},\ldots,d_{k-1}}}
&\Hugestrut\quad
{\overbrace{\Ccal\left[\LARGEstrut\filter{m}
   \Ccal^{k-i}\Ccal\left(\Largestrut\Ccal^{i-1}\Fcal X\vee\Ccal^{i-1}\Fcal Y
              \right)
                 \right]}^{E_{k}^{i}}}
   \ar@<-12ex>[dd]_{{\mathbf{d_{k-i+1}}}}
   \ar@<-9ex>[dd]
   \ar@<-8ex>[dd]
   \ar@<-7ex>[dd]^-{{d_{k-i+2},\ldots,d_{k}}}
         \ar@<-3.0ex>[l]_{d_{0},\ldots,d_{k-i}}
         \ar@<-2.5ex>[l]
         \ar@<-2.0ex>[l]
         \ar@<-1.5ex>[l]
\\
& \\
{\underbrace{\Ccal\left[
   \LARGEstrut\filter{m}
   \Ccal^{k-i-1}\Ccal\left(\Largestrut\Ccal^{i-2}\Fcal X\vee\Ccal^{i-2}\Fcal Y\right)
   \right]}_{E_{k-2}^{i-1}}}\quad
   \ar@<-6ex>[uu]
   \ar@<-7ex>[uu]
   \ar@<-8ex>[uu]_-{s_{k-i},\ldots,s_{k-2}}
   \ar@<-2.5ex>[r]
   \ar@<-3.0ex>[r]
   \ar@<-3.5ex>[r]_-{s_{0},\ldots, s_{k-i-1}}
& \quad{\underbrace{\Ccal\left[
   \LARGEstrut\filter{m}
   \Ccal^{k-i}\Ccal\left(\Largestrut\Ccal^{i-2}\Fcal X\vee\Ccal^{i-2}\Fcal Y\right)
   \right]}_{E_{k-1}^{i-1}}}
   \ar@<-6ex>[uu]
   \ar@<-7ex>[uu]
   \ar@<-8ex>[uu]_-{s_{k-i+1},\ldots,s_{k-1}}
         \ar@<-1.0ex>[l]_-{d_{0},\ldots,d_{k-i}}
         \ar@<-.5ex>[l]
         \ar@<-.0ex>[l]
         \ar@<+.5ex>[l]
}
\]

As in the case of $\Dcal_{\bullet}$, the faces and degeneracies
originating from $E_{k}^{i}$ come in two classes each, plus the special ``borderline" face map.
\begin{enumerate}
\item (Horizontal degeneracies) The maps $s_{0},...,s_{k-i}:E_{k}^{i}\rightarrow E_{k+1}^{i}$
    come from all the ways of inserting a $\Ccal$ in $\Ccal^{k-i}$.
\item (Vertical degeneracies) The maps $s_{k-i+1},\dots,s_{k}: E_{k}^{i}\rightarrow E_{k+1}^{i+1}$ come from the ways of inserting $\Ccal$ in $\Ccal^{i-1}\Fcal$.
\item (Horizontal faces) The maps $d_{0},\ldots,d_{k-i-1}:E_{k}^{i}\rightarrow E_{k-1}^{i}$ come from multiplying neighboring factors of $\Ccal$ in $\Ccal\filter{m}\Ccal^{k-i}$. The last horizontal face, $d_{k-i}$, multiplies the rightmost $\Ccal$ of $\Ccal^{k-i}$ with the remaining $\Ccal$ outside the wedge.
\item (``Borderline" face, vertical) The face map
$d_{k-i+1}:E_{k}^{i}\rightarrow E_{k-1}^{i}$ is induced by applying
    $\Ccal\filter{m}\Ccal^{k-i}$ to the natural map
\[
\gamma: \Ccal(\Ccal A \vee \Ccal B)
       \xrightarrow{\Ccal(\Ccal\iota_{A}\vee\Ccal\iota_{B})} \Ccal\Ccal(A\vee B)
       \xrightarrow{\ \mu_{\Ccal}\ } \Ccal(A\vee B)
\]
where $A=\Ccal^{i-2}\Fcal X$ and $B=\Ccal^{i-2}\Fcal Y$. For this we need the preservation of filtration noted in Remark~\ref{remark: beta and gamma preserve filtration}.
\item (Vertical faces) The maps $d_{k-i+2},\ldots,d_{k}:E_{k}^{i}\rightarrow E_{k-1}^{i-1}$ come from applying
    the monad multiplication $\Ccal^{2}\rightarrow\Ccal$
(or the module multiplication $\Ccal\Fcal\rightarrow\Ccal$)
to corresponding neighboring pairs $\Ccal^2$ (or $\Ccal\Fcal$) in both summands of  $\Ccal^{i-1}\Fcal X\vee\Ccal^{i-1}\Fcal Y$.
\end{enumerate}

\begin{proposition}      \label{proposition: prove Ecal is a simplicial space}
$E_{\bullet}$ is a simplicial space.
\end{proposition}

\begin{proof}
The proof is very similar to that of
Proposition~\ref{proposition: proving D is a simplicial space}.
The key matter to consider is that $E_{k}^{i}$, like $D_{k}^{i}$, has one borderline face, but this time it is $d_{k-i+1}$ (instead of $d_{k-i}$) and goes vertically (instead of horizontally). This means that there are slightly different diagrams that require special focus.
We need to check the following identities, which are the ones that do not follow straightforwardly from a ``commuting operations" argument.

\begin{itemize}
\item (From $E_{k}^{i}$ to $E_{k-2}^{i-2}$) Two borderline faces in a row:
\begin{equation}   \label{eq: E borderline face identity}
\begin{gathered}
{\mathbf{d_{k-i+1}}}d_{k-i+2}={\mathbf{d_{k-i+1}d_{k-i+1}}}
\qquad\qquad
\xymatrix{
E_{k-1}^{i-1}\ar[d]_-{{\mathbf{d_{k-i+1}}}}
       & E_{k}^{i}
            \ar[l]_-{d_{k-i+2}}
            \ar[d]^-{{\mathbf{d_{k-i+1}}}}\\
E_{k-2}^{i-2}  & E_{k-1}^{i-1}
                \ar[l]^-{{\mathbf{d_{k-i+1}}}}
}
\end{gathered}
\end{equation}
We write $A=\Ccal^{i-3}\Fcal X$ and
$B=\Ccal^{i-3}\Fcal Y$. Also, for the purpose of distinguishing a particular $\Ccal$ in~$E_{k}^{i}$, let $\Gcal=\Ccal$, and let $\mu_{\Ccal}$ indicate the transformation $\Ccal^2\rightarrow\Ccal$, while $\mu_{\Gcal}$ indicates the transformation $\Ccal\Gcal\rightarrow\Gcal$. The business part of the diagram comes down to
\[
\xymatrix{
   \Ccal\left(\Largestrut\Gcal A\vee\Gcal B\right)
   \ar[d]_-{\gamma_{\Gcal}}
&
   \Ccal\left(\Largestrut\Ccal\Gcal A\vee\Ccal\Gcal B
              \right)
   \ar[l]_-{\gamma_{\Ccal}}
   \ar[d]^-{\Ccal\left(\mu_{\Gcal}(A)\vee\mu_{\Gcal}(B)\right)}
\\
   \Gcal\left(\Largestrut A\vee B\right)
&
   \Ccal\left(\Largestrut\Gcal A\vee\Gcal B\right).
   \ar[l]_-{\gamma_{\Gcal}}
}
\]
All the maps in the diagram are maps of $\Ccal$-modules, so by Lemma~\ref{lemma: check on generators} we only need to verify commutativity restricted to
$\Ccal\Gcal A\vee\Ccal\Gcal B$ in the upper right-hand corner. Commutativity is easily verified on the two summands separately using naturality of $\Ccal\Gcal\rightarrow\Gcal$.\\

%

\item (From $E_{k}^{i}$ to $E_{k-2}^{i-1}$)

\begin{equation} \label{eq: E borderline and other face}
\begin{gathered}
d_{k-i}{\mathbf{d_{k-i+1}}}={\mathbf{d_{k-i}}}d_{k-i}
\qquad\qquad
\xymatrix{
E_{k-1}^{i}
\ar[d]_-{{\mathbf{d_{k-i}}}}
       & E_{k}^{i}
            \ar[l]_-{d_{k-i}}
            \ar[d]^-{{\mathbf{d_{k-i+1}}}} \\
E_{k-2}^{i-1}
      & E_{k-1}^{i-1}
          \ar[l]^-{d_{k-i}}
}
\end{gathered}
\end{equation}
Here the required diagram commutes because $\gamma$ is a map of $\Ccal$-modules, which was proved in Lemma~\ref{lemma: structure maps give C-modules}.
\\
\item (From $E_{k-1}^{i-1}$ to itself)
\begin{equation}     \label{eq: E degeneracy identity}
{\mathbf{d_{k-i+1}}}s_{k-i+1}=\id
\end{equation}

\noindent In this case what we need commutativity of
\[
\xymatrix{
\Ccal\left(\Largestrut A\vee B
                \right)
   \ar[rr]^{\Ccal(\eta\vee\eta)}
   \ar[drr]_{\id}
&&\Hugestrut\quad
\Ccal\left(\Largestrut\Ccal A\vee\Ccal B
              \right)
   \ar[d]^{\gamma}
\\
%
&& \Ccal\left(\Largestrut A\vee B
                \right),
}
\]
where $A=\Ccal^{i-2}\Fcal X$ and $B=\Ccal^{i-2}\Fcal Y$.
Since all the maps are $\Ccal$-module maps, by
Lemma~\ref{lemma: check on generators} we only have to check commutativity
restricted to $A\vee B$, where it is true by the unital property of the
monad~$\Ccal$.
\end{itemize}
\end{proof}

The last thing we need for this section is to relate $\Ecal_{\bullet}$
to~$\Dcal_{\bullet}$.
\begin{proposition}     \label{proposition: prove map E to D}
The map
\[
\xymatrix{
E_{k}^{i}
   =\Ccal\filter{m}\Ccal^{k-i}
   \left[
   \LARGEstrut\Ccal
         \left(\Largestrut\Ccal^{i-1}\Fcal X\vee\Ccal^{i-1}\Fcal Y\right)
   \right]
   \ar[d]^{\simeq}
\\
D_{k}^{i}
   =\Ccal\filter{m}
   \Ccal^{k-i}
      \left[
         \LARGEstrut\Ccal^{i}\Fcal X\times\Ccal^{i}\Fcal Y
   \right]
}
\]
defines a map of simplicial spaces $\Ecal_{\bullet}\rightarrow\Dcal_{\bullet}$.
\end{proposition}

\begin{proof}
As usual, ``commuting operations" (in this case, in the form of naturality of the wedge-to-product map) take care of all but the compatibility involving the borderline face maps. Since the borderline face maps in $\Ecal_{\bullet}$ and $\Dcal_{\bullet}$ are different, this leaves us with two conditions to check.

To check that
$E_{k}^{i}\rightarrow D_{k}^{i}$ is compatible with $d_{k-i}$ (borderline face map for~$E_{k}^{i}$),
\[\xymatrix{
E_{k-1}^{i}\ar[d]_{\simeq}
        & E_{k}^{i}\ar[l]_-{d_{k-i}}
                   \ar[d]^{\simeq}\\
D_{k-1}^{i} & D_{k}^{i},\ar[l]^-{\mathbf{d_{k-i}}}
}
\]
%
the business part of the diagram is
\[
\xymatrix{
   \Ccal
         \left(\Largestrut\Ccal^{i-1}\Fcal X\vee\Ccal^{i-1}\Fcal Y\right)
   \ar[d]^{\simeq}
&
   \Ccal\Ccal
         \left(\Largestrut\Ccal^{i-1}\Fcal X\vee\Ccal^{i-1}\Fcal Y\right)
   \ar[d]^{\simeq}
   \ar[l]^{\mu_{\Ccal}(\whatever)}
\\
\Largestrut\Ccal^{i}\Fcal X\times\Ccal^{i}\Fcal Y
&
\Ccal \left(\Largestrut\Ccal^{i}\Fcal X\times\Ccal^{i}\Fcal Y
      \right).
   \ar[l]_-{\beta}
}
\]
The diagram commutes because the
wedge-to-product map is a map of $\Ccal$-modules
(Lemma~\ref{lemma: special is a C-module}).

Next we check $d_{k-i+1}$ (borderline face map for~$D_{k}^{i}$):
\[\xymatrix{
E_{k-1}^{i-1}\ar[d]_{\simeq}
        & E_{k}^{i}\ar[l]_(.42){\mathbf{d_{k-i+1}}}
                   \ar[d]^{\simeq}\\
D_{k-1}^{i-1} & D_{k}^{i}.\ar[l]_(.42){d_{k-i+1}}
}
\]
%
%
In this case, the business part of the diagram is the square below, where $A=\Ccal^{i-2}\Fcal X$ and $B=\Ccal^{i-2}\Fcal Y$:
\[
\xymatrix{
   \Ccal(A\vee B)
   \ar[d]^{\simeq}_{\Ccal\collapsemap{A}\times\Ccal\collapsemap{B}}
&&
   \Ccal
         \left(
         \Largestrut\Ccal A\vee\Ccal B
         \right)
    \ar[d]_{\simeq}
        ^{\Ccal\collapsemap{\Ccal\! A}\times\Ccal\collapsemap{\Ccal\! B}}
   \ar[ll]^{\gamma}
\\
\Ccal A\times\Ccal B
&&
\Ccal\Ccal A\times \Ccal\Ccal B.
   \ar[ll]_{\mu_{\Ccal}\times\mu_{\Ccal}}
}
\]
All of the maps in the diagram are $\Ccal$-module maps by
Lemma~\ref{lemma: special is a C-module} for the vertical maps,
Lemma~\ref{lemma: structure maps give C-modules} for the top horizontal map, and
Corollary~\ref{corollary: product of multiplications is module map}
for the bottom horizontal map. Hence by
Lemma~\ref{lemma: check on generators}, we only need to check commutativity on the parenthesized $\Ccal A\vee\Ccal B$ in the upper righthand corner. Both ways around the square are the standard inclusion of the wedge
$\Ccal A\vee\Ccal B$
into the product $\Ccal A\times\Ccal B$,
 by using the unital property of~$\Ccal$ for the clockwise direction.
\end{proof}

\bigskip
\section{Maps from $\Ecal_{\bullet}$ to $\Dcal_{\bullet}$ and $A_{\bullet}$}
\label{section: maps out of E}

We continue to assume that $\Fcal$ is an augmented, special $\Gamma$-space with an augmentation-preserving action of~$\Ccal$.
The last several sections have studied parts of
diagram~\eqref{diagram: square of auxiliaries}, which we reproduce here for the convenience of the reader:
\begin{equation}   
\tag{\ref{diagram: square of auxiliaries}}
\thatdiagram
\end{equation}

\noindent In this section, we finish the proof of Lemma~\ref{lemma: commuting squares},
which we also restate here for convenience.

\begin{CommutingSquareLemma}
\CommutingSquareLemmaText
\end{CommutingSquareLemma}

At this point we have dealt with the top row,
\begin{itemize}
\item $f:A_{\bullet}\rightarrow B_{\bullet}$ is a map of simplicial spaces (see~\eqref{eq: desired equivalance of SS}),
\end{itemize}
the lower row,
\begin{itemize}
\item $p:\Ecal_{\bullet}\rightarrow\Dcal_{\bullet}$
is a map of simplicial spaces and a weak equivalence (Propositions~\ref{proposition: proving D is a simplicial space}, \ref{proposition: prove Ecal is a simplicial space},
and~\ref{proposition: prove map E to D}),
\end{itemize}
and the right-hand column,
\begin{itemize}
\item $j$ is a simplicial inclusion (equation~\eqref{eq: defn j}), with simplicial
deformation retraction $q$ (Proposition~\ref{proposition: simplicial homotopy}).
\end{itemize}
To complete the proof of Lemma~\ref{lemma: commuting squares}, we must handle the left-hand column (i.e., exhibit simplicial inclusion and retraction $A_{\bullet}\hookrightarrow \Ecal_{\bullet}\rightarrow A_{\bullet}$), and we must establish the square's two
commutativity properties.

Recall $\Ecal_{k}=\bigvee_{i=0}^{k+1}E_{k}^{i}$, and from \eqref{eq: definition of Eki}  and~\eqref{eq: ident of Ek0 as Ak}, the definition of $E_{k}^{i}$ and the special cases $E_{k}^{0}$ and $\Largestrut E_{k}^{k+1}$:
\begin{align*}
E_{k}^{i}
   &=\Ccal\filter{m}\Ccal^{k-i+1}
         \left(\Largestrut\Ccal^{i-1}\Fcal X\vee\Ccal^{i-1}\Fcal Y\right)
\\
E_{k}^{0}&=A_{k}=
          \Ccal\filter{m}
          \Ccal^{k}\Fcal\left(\largestrut X\vee Y\right)
   \\
E_{k}^{k+1}
   &=\Ccal\filter{m}
         \left(
         \Largestrut\Ccal^{k}\Fcal X\vee\Ccal^{k}\Fcal Y
         \right).
\end{align*}
At the beginning of Section~\ref{section: E is a simplicial space}, we exhibited the simplicial structure maps for $\Ecal_{\bullet}$. Observe that for $i=0$, all of the face and degeneracy maps of $E_{k}^{0}$ go to $E_{k-1}^{0}$ and $E_{k+1}^{0}$, respectively. Hence $\left\{E_{k}^{0}\right\}$ is a simplicial subobject of~$\Ecal_{\bullet}$, that is $A_{\bullet}$ is a simplicial subspace of~$\Ecal_{\bullet}$. We denote the inclusion by~$i$, as in~\eqref{diagram: square of auxiliaries}.

We need to prove that
$A_{\bullet}$ is actually a retract of $\Ecal_{\bullet}$.
We define $r_{k}:\Ecal_{k}\rightarrow A_{k}$ on each wedge summand of the domain by
\[
\begin{gathered}
\xymatrix{
E_{k}^{i}=\Ccal\filter{m}
   \Ccal^{k-i+1}
   \left(\Largestrut\Ccal^{i-1}\Fcal X\vee\Ccal^{i-1}\Fcal Y
              \right)
   \ar[d]^{r_{k}}\\
A_{k}=E_{k}^{0}=\Ccal\filter{m}
   \Ccal^{k}
   \Fcal\left(\strut X\vee Y\right),
}   
\end{gathered}
\]
induced by $\Ccal\filter{m}\Ccal^{k-i+1}$ applied to the natural map
\[
\Ccal^{i-1}\Fcal X\vee\Ccal^{i-1}\Fcal Y
   \xrightarrow{\ \Ccal^{i-1}\Fcal(\iota_{X})\vee\Ccal^{i-1}\Fcal(\iota_{Y})
   \ }
   \Ccal^{i-1}\Fcal(X\vee Y).
\]

\begin{lemma}     \label{lemma: E to A retraction}
$r_{\bullet}\colon \Ecal_{\bullet}\rightarrow A_{\bullet}$ is a simplicial retraction of $i\colon A_{\bullet}\hookrightarrow\Ecal_{\bullet}$.
\end{lemma}

\begin{proof}
The inclusion $A_{\bullet}\hookrightarrow\Ecal_{\bullet}$ takes
$A_{k}$ identically to~$E_{k}^{0}$, and
$r_{k}:\Ecal_{k}\rightarrow A_{k}$ is the identity on the summand~$E_{k}^{0}$.
So we just need to check that the collection of maps $r_{k}$ is
compatible with the face and degeneracy maps of~$\Ecal_{\bullet}$.

All of the diagrams except the one with the borderline face map, $d_{k-i+1}$,  commute for straightforward naturality reasons. For $d_{k-i+1}$, we
need commutativity of the following diagram:
\begin{equation*}
\begin{gathered}
\xymatrix{
{\overbrace{\Ccal\filter{m}
   \Ccal^{k-i+1}
   \left(\Largestrut\Ccal^{i-1}\Fcal X\vee\Ccal^{i-1}\Fcal Y
              \right)}^{E_{k}^{i}}} \quad
   \ar[d]^{r_{k}}
   \ar[r]^-{d_{k-i+1}}
&
\quad{\overbrace{\Ccal\left[
   \LARGEstrut\filter{m}
   \Ccal^{k-i+1}\left(\Largestrut\Ccal^{i-2}\Fcal X\vee\Ccal^{i-2}\Fcal Y\right)
   \right]}^{E_{k-1}^{i-1}}}
    \ar[d]^{r_{k-1}}
              \\
{\underbrace{\Ccal\filter{m}
   \Ccal^{k}
   \Fcal\left(\strut X\vee Y\right)}_{A_{k}=E_{k}^{0}}}
      \ar[r]^{d_{k-i+1}}
&
{\underbrace{\Ccal\filter{m}
   \Ccal^{k-1}
   \Fcal\left(\strut X\vee Y\right)}_{A_{k-1}=E_{k-1}^{0}}}.
}   
\end{gathered}
\end{equation*}
Recall from Section~\ref{section: E is a simplicial space}
that the borderline face map
$d_{k-i+1}$ of $\Ecal_{k}^{i}$ is induced by the map $\gamma_{\Gcal}$ below
(see also~\eqref{eq: define gamma}), where the $\Ccal$-module $\Gcal$ is $\Ccal$ itself, and the spaces
$A$ and $B$ are $\Ccal^{i-2}\Fcal X$ and $\Ccal^{i-2}\Fcal Y$, respectively:
\[
\gamma_{\Gcal}\colon
        \Ccal\left(\strut\Gcal A \vee \Gcal B\right)
\xrightarrow{\Ccal(\Gcal\iota_{A}\vee\Gcal\iota_{B})}
        \Ccal\Gcal\left(\strut A\vee B\right)
\xrightarrow{\quad \mu_{\Gcal}\quad }
        \Gcal(A\vee B).
\]
With this notation, the necessary commuting diagram above
is induced by:
\[
\begin{gathered}
\xymatrix{
   \Ccal
   \left(\Largestrut\Ccal\Ccal^{i-2}\Fcal X\vee\Ccal\Ccal^{i-2}\Fcal Y
              \right) \quad
   \ar[d]_{\Ccal\left(\Ccal^{i-1}\Fcal\iota_{X}\,\vee\,\Ccal^{i-1}\Fcal\iota_{Y}\right)}
   \ar[r]^{\gamma_{\Ccal}}
&
\quad   \Ccal
   \left(\Largestrut\Ccal^{i-2}\Fcal X\vee\Ccal^{i-2}\Fcal Y\right)
   \ar[d]^{\Ccal\left(\Ccal^{i-2}\Fcal\iota_{X}\,\vee\,\Ccal^{i-2}\Fcal\iota_{Y}\right)}
              \\
   \Ccal^{i}
   \Fcal\left(\strut X\vee Y\right)
      \ar[r]^-{\ \mu_{\Ccal}(\Ccal^{i-2}\Fcal) \ }
&
   \Ccal^{i-1}
   \Fcal\left(\strut X\vee Y\right).
}   
\end{gathered}
\]
By Lemma~\ref{lemma: check on generators} and Lemma~\ref{lemma: structure maps give C-modules}, it is sufficient to show commutativity restricting to the inside of the parentheses in the upper left-hand corner. Following $\Ccal\Ccal^{i-2}\Fcal X=\Ccal^{i-1}\Fcal X$ around the diagram clockwise, we find the composite to be the natural map
\[
\Ccal^{i-1}\Fcal X
     \rightarrow \Ccal^{i-1}\Fcal\left(\strut X\vee Y\right),
\]
and the same is true if we go around counter-clockwise, by using the unital property of~$\Ccal$.
\end{proof}

\begin{proof}[Proof of Lemma~\ref{lemma: commuting squares}]

Consider the diagram
\begin{equation}   
\tag{\ref{diagram: square of auxiliaries}}
\thatdiagram
\end{equation}
We must first show that the outer square commutes, that is, $f\circ r=q\circ p$. Below is the diagram for simplicial level $k$ and the wedge summands $E_{k}^{i}$ and~$D_{k}^{i}$ of $\Ecal_{k}$ and~$\Dcal_{k}$:
\[
\begin{gathered}
\xymatrix{
\Ccal\filter{m}\Ccal^{k}\Fcal\left(\strut X\vee Y\right)
\ar[r]
& \Ccal\filter{m}\Ccal^{k}\Fcal X
\times
\Ccal\filter{m}\Ccal^{k}\Fcal Y
\\
\Ccal\filter{m}\Ccal^{k-i+1}
   \left(\Largestrut\Ccal^{i-1}\Fcal X\vee\Ccal^{i-1}\Fcal Y
              \right)
   \ar[u]^{r_{k}}
   \ar[r]
&
\Ccal\filter{m}
   \Ccal^{k-i}\left(\Largestrut\Ccal^{i}\Fcal X \times\Ccal^{i}\Fcal Y \right).
   \ar[u]_{q_{k}}
}   
\end{gathered}
\]
We only need to check commutativity into each factor in the upper right-hand corner. Let $\Gcal=\Ccal^{i-1}\Fcal$. Considering only the first factor, the diagram is induced by
\[
\begin{gathered}
\xymatrix{
\Ccal\Gcal(X\vee Y)
   \ar[r]
& \Ccal\Gcal X\\
\Ccal(\Gcal X\vee\Gcal Y)
   \ar[r]
   \ar[u]^-{\Ccal(\Gcal\iota_{X}\vee\Gcal\iota Y)}
&
\Ccal\Gcal X \times \Ccal\Gcal Y,
   \ar[u]
}
\end{gathered}
\]
which obviously commutes since both ways around the diagram collapse $Y$ to a point (see Proposition~\ref{proposition: simplicial homotopy} for the definition of~$q_{k}$).

Lastly, we need to show that the inner square of ~\eqref{diagram: square of auxiliaries}
is homotopy commutative.
First we prove that $f=q\circ p\circ i$. To see this, recall that $A_{k}=E_{k}^{0}$, so we only have to use that summand of $\Ecal_{\bullet}$ in the lower left-hand corner:
\[
\begin{gathered}
\xymatrix{
A_{k}=\Ccal\filter{m}\Ccal^{k}\Fcal\left(\strut X\vee Y\right)
\ar[d]_{=}
\ar[r]^-{f_{k}}
& \, B_{k}=\Ccal\filter{m}\Ccal^{k}\Fcal X
\times
\Ccal\filter{m}\Ccal^{k}\Fcal Y
\\
E_{k}^{0}=\Ccal\filter{m}\Ccal^{k}\Fcal\left(\strut X\vee Y\right)
   \ar[r]_-{p_{k}}
&
\displaystyle
 D_{k}^{0}=  \Ccal\filter{m}\Ccal^{k}\left(\Largestrut\Fcal X \times\Fcal Y \right)
 \ar[u]_{q_{k}}
}   
\end{gathered}
\]
Both ways around the diagram are the same, because the map into each factor of $B_{k}$ is given by collapsing $X$ or $Y$ to a point. Postcomposing
$f=q\circ p\circ i$ with $j$ gives
\[
j\circ f=j\circ q\circ p\circ i.
\]
But by Proposition~\ref{proposition: simplicial homotopy}, $j\circ q$ is homotopic to the identity map of $\Dcal_{\bullet}$, and the lemma follows.

\end{proof}

%
%
%
%

\bibliographystyle{amsalpha}

\begin{thebibliography}{ADL20}

\bibitem[AD01]{Arone-Dwyer}
G.~Z. Arone and W.~G. Dwyer, \emph{Partition complexes, {T}its buildings and
  symmetric products}, Proc. London Math. Soc. (3) \textbf{82} (2001), no.~1,
  229--256. \MR{1794263 (2002d:55003)}

\bibitem[ADL16]{ADL2}
G.~Z. Arone, W.~G. Dwyer, and K.~Lesh, \emph{Bredon homology of partition
  complexes}, Doc. Math. \textbf{21} (2016), 1227--1268. \MR{3578208}

\bibitem[ADL20]{ADL3}
Gregory~Z. Arone, William~G. Dwyer, and Kathryn Lesh, \emph{p-toral
  approximations compute bredon homology}, arXiv:1901.07330v2 [math.AT].

\bibitem[AL07]{Arone-Lesh-Crelle}
Gregory~Z. Arone and Kathryn Lesh, \emph{Filtered spectra arising from
  permutative categories}, J. Reine Angew. Math. \textbf{604} (2007), 73--136.
  \MR{2320314 (2008c:55013)}

\bibitem[AL10]{Arone-Lesh-Fundamenta}
\bysame, \emph{Augmented {$\Gamma$}-spaces, the stable rank filtration, and a
  {$bu$} analogue of the {W}hitehead conjecture}, Fund. Math. \textbf{207}
  (2010), no.~1, 29--70. \MR{2576278}

\bibitem[BF78]{Bousfield-Friedlander}
A.~K. Bousfield and E.~M. Friedlander, \emph{Homotopy theory of {$\Gamma
  $}-spaces, spectra, and bisimplicial sets}, Geometric applications of
  homotopy theory ({P}roc. {C}onf., {E}vanston, {I}ll., 1977), {II}, Lecture
  Notes in Math., vol. 658, Springer, Berlin, 1978, pp.~80--130. \MR{513569}

\bibitem[Goo90]{Goodwillie-Calculus-I}
Thomas~G. Goodwillie, \emph{Calculus. {I}. {T}he first derivative of
  pseudoisotopy theory}, $K$-Theory \textbf{4} (1990), no.~1, 1--27.
  \MR{1076523}

\bibitem[KP85]{Kuhn-Priddy}
Nicholas~J. Kuhn and Stewart~B. Priddy, \emph{The transfer and {W}hitehead's
  conjecture}, Math. Proc. Cambridge Philos. Soc. \textbf{98} (1985), no.~3,
  459--480. \MR{803606 (87g:55030)}

\bibitem[Kuh82]{Kuhn-Whitehead}
Nicholas~J. Kuhn, \emph{A {K}ahn-{P}riddy sequence and a conjecture of {G}.
  {W}. {W}hitehead}, Math. Proc. Cambridge Philos. Soc. \textbf{92} (1982),
  no.~3, 467--483. \MR{677471 (85f:55007a)}

\bibitem[Lyd99]{Lydakis-Gamma}
Manos Lydakis, \emph{Smash products and {$\Gamma$}-spaces}, Math. Proc.
  Cambridge Philos. Soc. \textbf{126} (1999), no.~2, 311--328. \MR{1670245}

\bibitem[Man10]{Mandell-Inverse}
Michael Mandell, \emph{An inverse $K$-theory functor}, Doc. Math. \textbf{15} (2010), 765–791. \MR{2735988}

\bibitem[May72]{May-Geometry}
J.~P. May, \emph{The geometry of iterated loop spaces}, Springer-Verlag,
  Berlin-New York, 1972, Lectures Notes in Mathematics, Vol. 271. \MR{0420610}

\bibitem[Rie14]{Riehl-Categorical}
Emily Riehl, \emph{Categorical homotopy theory}, New Mathematical Monographs,
  vol.~24, Cambridge University Press, Cambridge, 2014. \MR{3221774}

\bibitem[Rog92]{Rognes-Topology}
John Rognes, \emph{A spectrum level rank filtration in algebraic {$K$}-theory},
  Topology \textbf{31} (1992), no.~4, 813--845. \MR{1191383 (94d:19007)}

\bibitem[Seg74]{Segal-Categories}
Graeme Segal, \emph{Categories and cohomology theories}, Topology \textbf{13}
  (1974), 293--312. \MR{0353298}

\end{thebibliography}

\providecommand{\bysame}{\leavevmode\hbox to3em{\hrulefill}\thinspace}
\providecommand{\MR}{\relax\ifhmode\unskip\space\fi MR }
\providecommand{\MRhref}[2]{%
  \href{http://www.ams.org/mathscinet-getitem?mr=#1}{#2}
}
\providecommand{\href}[2]{#2}

\end{document}